\documentclass[smallcondensed]{svjour3}          
\smartqed  

\usepackage{secdot}

\usepackage{amsmath}
\usepackage{textcomp}
\usepackage{amsfonts}
\usepackage{graphicx}
\usepackage{chemarr}
\usepackage{float}

\usepackage{xpatch}
\makeatletter
\xpatchcmd{\@thm}{\thm@headpunct{.}}{\thm@headpunct{}}{}{}
\makeatother

\usepackage[]{natbib}
\usepackage{multicol}
\numberwithin{figure}{section}

\usepackage{natbib}
 \usepackage{float}
\usepackage{caption}
\usepackage{subcaption}
\usepackage[font=small,labelfont=bf,tableposition=top]{caption}
\usepackage{multirow}

\usepackage[titletoc,toc,title]{appendix}

\usepackage{lineno,hyperref}
\modulolinenumbers[5]

\begin{document}

\title{Non-stiff methods for Airy flow and the modified Korteweg de Vries equation.
}


\author{Mariano Franco-de-Leon         \and
        John Lowengrub 
}


\institute{Mariano Franco-de-Leon \at
              University of California, Irvine. 540 H Rowland Hall. Irvine, CA 92697-3875 \\
              Tel.: 949 307 6787\\
              \email{mfrancod@uci.edu}           
           \and
           John Lowengrub \at
              University of California, Irvine. 540 H Rowland Hall. Irvine, CA 92697-3875\\
              Tel.: 949-751-7557\\
              \email{jlowengr@uci.edu}                                      
}

\date{Received: date / Accepted: date}

\maketitle

\begin{abstract}
 In this paper, we implement non-stiff interface tracking methods for the evolution of 2-D curves that follow Airy flow, a curvature-dependent dispersive geometric evolution law. The curvature of the curve satisfies the modified Korteweg de Vries equation, a dispersive non-linear soliton equation. We present a fully discrete space-time analysis of the equations (proof of convergence) and numerical evidence that confirms the accuracy, convergence, efficiency, and stability of the methods.
\keywords{dispersive\and solitons \and numerical solution \and mKdV \and Airy flow}
\subclass{65Gxx,65Mxx,65Txx,35Q58,42Axx}
\end{abstract}

\section{Introduction}
\label{intro}
   Geometric curve flow models are an important class of methods for interface motion; where we understand an interface as a geometrical one-dimensional surface with no thickness.  Under these laws, curves evolve according to local functionals of their geometrical properties. A classical example is the mean curvature flow (\cite{Helfrich}, \cite{HLS2}, \cite{CNFSF}). The governing equations are parabolic partial differential equations. In the materials science context, mean curvature flows are related to the motion of grain boundaries that separate crystallites (grains) with different crystallographic symmetries. Another type of geometric evolution, where the governing equations of evolution are dispersive rather than parabolic, has been garnering increased attention. Dispersive equations arise in a variety of applications (collision-free hydromagnetic waves, ion-acoustic waves in cold plasma, electrostatic fields of graphene, human arm movement, computer vision (\cite{SmKdV}, \cite{appmKdV}, \cite{MiuraKdV}, \cite{affinearm}, \cite{HewlettP})), and their mathematical theories have revealed strong relations with differential geometry, geometrical analysis, soliton theory, and integrable systems (\cite{MiuraKdV}, \cite{RPalaissolitons}, \cite{DGCF}, \cite{IEMPC}, \cite{TTaoAiry}).
   
   In this article, we present the development, implementation, and analysis of schemes to obtain numerical (periodic in space) solutions for the modified Korteweg de Vries (mKdV) equation,
  \begin{equation}\label{mKdV}
  k_t=k_{sss}+\frac{3}{2}k^2k_s.
  \end{equation}  
  
   The mKdV equation is the first non-linear generalization of the KdV soliton model for water waves (1895). The contributions of Zabusky, Gardner, Green, Kruskal, Miura, \cite{lax} displayed their striking properties including: the preservation of form through non-linear interactions, decomposition of waves into smaller solitons, different families of solutions, infinite number of conservation laws, its relations with the Schr\"odinger operator and the eigenvalue problem (\cite{MiuraKdV}, \cite{perKdVper}), the Miura transform to obtain solutions (and well-posedness) of KdV from solutions (and well-posedness) of mKdV (\citet{MiuraKdV}, \citet{gwplwpTaomKdv}, \citet{GGKM}, \citet{somepermKdV}), and the inverse scattering transform (IST). For other type of mKdV solutions (e.g. kinks, breathers), or periodic domains (also well-posed \cite{TTaoAiry}), this approach is not plausible since decay at $\infty$ is a crucial hypothesis for IST. Other analytical techniques to find solutions of (periodic and non-periodic) mKdV-like equations include the use of Jacobi, Weierstrass functions, Hamiltonian structures, B\"acklund-Darboux transforms, the tangent hyperbolic method (\cite{SolitonChuuLian}, \cite{jacobimKdV}, \cite{tanmethod}, \cite{allsolmKdV}, \cite{LaxperiodicKdV}, \cite{abundantmkdv}). Nevertheless, there is a lot of work to develop regarding the orbital stability of these periodic (and non-periodic) waves under perturbations of the underlying solution, soliton resolution conjecture, collisions, multisolitons, compactons, generalizations (gKdV), nonlinear Schrodinger-Airy system, new solutions, and its relations with other equations (\cite{gwplwpTaomKdv}, \cite{Compactons}, \cite{Whysolstab}, \cite{solitonconjecture}, \cite{mKdVnewsol2011}, \cite{nlairyschro}, \cite{nonde}, \cite{Quantumele}, \cite{WellpossKdV}).  
  
  Over the past 30 years, the numerical study of initial value problems of free surface flows has been increasingly important in representing systems of partial differential equations, not just for physical modeling, but also as an empirical tool to analyze theoretical aspects of the underling system. The primary classes of algorithms (Lagrangian and Eulerian), as well as mixed approaches, had been focus on the solution of parabolic (dissipative case) partial differential equations (\cite{SmKdV}, \cite{WNLPDE}, \cite{levelset}, \cite{Hsiao}, \cite{CNAd}, \cite{Engi}, \cite{intc}, \cite{CNFSF}, \cite{CiCP}, \cite{WJ}, \cite{SHongkai}). There are far fewer methods (Adomain decomposition, finite differences, radial basis functions, pseudo-spectral methods) developed for simulating dispersive geometric evolution equations  (\cite{SmKdV}, \cite{WNLPDE}, \cite{CNAd}, \cite{numFEMKdVreal}, \cite{KdVrbf}).

   In this paper, we exploit the theory behind dispersive equations and geometric curve flows by evolving solutions of a closed curve under Airy flow. Then we recover mKdV solutions from the curvature of the curve (See $(\ref{kevt})$), instead of solving mKdV directly, and gaining one degree of smoothness in the numerical implementation. The evolution of any $2$-D closed smooth planar curve $X=(x(\alpha,t),y(\alpha,t))$ with spatial ($2\pi$ periodic) parameter $\alpha$,  time variable $t$, can be described as
$$X_t=V\textbf{n}+T\textbf{s},$$
where $\textbf{s}$, $\textbf{n}$ denote the tangent and outward-normal unit vectors respectively, and $V,T$ are the corresponding normal and tangential velocities. In Airy flow the normal and tangential velocities are $V=-k_s$, $T=\frac{k^2}{2}$: 
 \begin{equation}\label{Airy}
X_t=(-k_s)\textbf{n}+(\frac{k^2}{2})\textbf{s},
\end{equation}
where $k$ is the curvature along the curve, $s$ denotes the arc-length parameter, and subscripts represent partial differentiation. 

   The high number of spatial derivatives, nonlinearity, and dispersive effects represent particular challenges when solving these equations numerically. Explicit time stepping methods undergo severe time constraints. In addition, certain spatial discretizations may lead to numerical instabilities. As observed previously (\cite{Ceniceros}, \cite{FSFSSE}, \cite{JH}, \cite{StabilityI}), even spectral accuracy does not guarantee stability. Further, time-step constraints may be amplified during the evolution due to clustering of points at the interface. The tangential velocity $T=\frac{k^2}{2}$ for Airy flow enforces equal arc-length parametrization at all times provided it is satisfied at the initial step. In this way, $s_{\alpha}$ is everywhere equal to its mean and evolves according to the length $L$ of the curve, a uniform discretization in $\alpha$ is then uniform in $s$ (i.e. $s(\alpha,t)=\frac{\alpha L(t)}{2\pi}$). Numerically, this choice of frame avoids the time-step restrictions for stability due to clustering of grid points at the interface. Another feature for curvature-dependent problems is the relation $k=\theta_s$, between the curvature $k$ and $\theta$, the angle that makes the tangent vector $\textbf{s}=\frac{dX}{ds}$ and the $x$-axis ($\theta=tan^{-1}(\frac{y_{\alpha}}{x_{\alpha}}$)). Using the arc-length parameter and $\theta$, $L$, ($\theta$-$L$ formulation), as dynamical variables  (\cite{repara}) instead of $(x,y)$ coordinates, equation $(\ref{Airy})$ becomes:
 \begin{equation}\label{eqL}
L_t=0,
\end{equation}
\begin{equation}\label{eqT}
\theta_t=(\frac{2\pi}{L})^3[\theta_{\alpha \alpha \alpha}+\frac{\theta_{\alpha}^3}{2}].
\end{equation}

We can then obtain $(x,y)$ by integrating the expression $(x_{\alpha},y_{\alpha})=s_{\alpha}(\sin\theta,\cos\theta)$, \cite{Lo1}, and recover solutions of the mKdV equation from $k=\theta_s$.

  The linear term of the equation $(\ref{eqT})$ displays the reason of stiffness whose stability constraint for an explicit method has the form $\Delta t\leq C\cdot (\overline{ s_{h }} h)^3$ where $\overline{s_h}=min_{\alpha}s_{\alpha}$ and $h$ is the grid spacing in $\alpha$. A stable and accurate discretization must guarantee a perfect balance between nonlinear and dispersive effects.   We use the small-scale decomposition (SSD) of the equations, developed by Lowengrub, and Shelley (HLS) \cite{Lo1} to examine the source of stiffness at small scales at which curvature acts as a linear operator.

          Linear analysis and numerical conservation of first integrals of motion (conservation of mass, momentum, and energy for the problem of the real line (\cite{MiuraKdV}, \cite{KdVConservLaws}, \cite{wwpoub})) for mKdV equation are used to test the accuracy of the numerical methods. Semi-discrete (continuous time) analysis (e.g. \cite{JH}, \cite{Ceniceros}) suggested that numerical filters need to be used to overcome instabilities generated by truncation and aliasing errors arising when computing spatial derivatives (\cite{Krasny}). In contrast, our fully discrete space-time analysis of convergence demonstrates that the use of the filter is not related to convergence, but may enhance stability. 
  
 The paper is organized as follows: In section $\ref{sec:1}$ we describe the numerical schemes used to treat the nonlinear dispersive equation ($\ref{eqT}$) needed to evolve Airy flow. Our most important theoretical (proof of) convergence results are given in section $\ref{sec:2}$.  As a first accuracy test, linear versions of the solution for Airy flow and mKdV are derived in section $\ref{sec:4}$ and compared against the numerical solutions. Additionally, numerical results including accuracy, convergence, stability, dynamics and the use of filters, is covered on section $\ref{sec:5}$. Concluding remarks are given in section $\ref{affinefuture}$,  and technical computations in the appendix $\ref{appe1}$.


\section{Numerical Methods}
\label{sec:1}

 Next, we introduce the notation to describe the schemes and the convergence analysis (\cite{JH}, \cite{Ceniceros}, \cite{canuto}): arbitrary smooth functions  are expressed by $f,g$, and constants (independents of discretization) are written generically as $C$. For a complex valued function $f$ defined over $[0,2\pi]$, the (continuous) Fourier coefficients of $f$ are:
\begin{equation}\label{contFourier}
\mathcal{F}\widehat{f_m}=\frac{1}{2\pi}\int_0^{2\pi}f(x)e^{-imx}dx,\:\:m=0,\pm 1,\pm 2,... 
\end{equation}

Then, the Fourier series of $f$ is 
\begin{equation}\label{InvcontFourier}
\mathcal{I}f(x)=\sum_{m={-\infty}}^{m=\infty}\widehat{f}_me^{imx}.
\end{equation}

 Denote by $S_h$, $Int_h$ (for $2\pi$ periodic functions of zero mean) the spectral derivative and integral operators used in this problem defined in Fourier space by:
  \begin{equation}
\widehat{S_h f_m}= i m\widehat{f_m},\label{PseudoSpectralderiv}
\end{equation}

 \begin{equation}
\widehat{Int_h f_m}=\begin{cases}
    \frac{ \widehat{f_m} }{im}, & \text{if $m\neq 0$},\\
        0, & \text{if $k=0$.}
  \end{cases}\label{PseudoSpectralInt}
\end{equation}

Observe that the linear part of equation $(\ref{eqT})$ is diagonalizable by the Fourier transform in the following way
 \begin{equation}
 \frac{\partial \mathcal{F} \widehat {\theta_m^t}}{\partial t}+i (\frac{2\pi m}{L})^3\mathcal{F} \widehat{\theta_m^t}=\mathcal{F}\widehat{NL_m^t},\label{intfactor1}
 \end{equation}
  where $L$ is the length of the curve (constant), $m$ is the wavenumber, the super index $t$ represents time and
 \begin{equation}
  NL(\alpha,t)= ((\frac{2\pi}{L})^3  \frac{\theta_{\alpha}(\alpha,t)^3}{2}).   
   \end{equation}   

 Consider a linear propagator method to absorb the leading order (linear term) prior to discretization. Several researchers in different contexts have used linear propagator schemes, e.g. simulations for Navier-Stokes equations, Hele-Shaw flows, reaction-diffusion systems, multicomponent fluids, multiphase materials (\cite{HLS2}, \cite{Lo1}, \cite{G.D.}, \cite{Micro}, \cite{Nie}) to name a few. Consider the integrating factor $r_m^t:=e^{i(2\pi m)^3 tL^{-3}}$ and the function $\Psi(m,t):=r_m^t\mathcal{F}\widehat{\theta_m^t}$, thus equation $(\ref{intfactor1})$  is equivalent to 
  \begin{equation}
 \frac{\partial\left(\Psi(m,t) \right) }{\partial t} =r_m^t\mathcal{F}\widehat{NL_m^t}.\label{intfactor2}
 \end{equation}

  These formulation motivates the use of Discrete Fourier Transform (DFT). In parallel with the continuous case $(\ref{contFourier})$,$(\ref{InvcontFourier})$, given a periodic  function $f$,  whose values are known on a uniform grid of mesh size $h=\frac{2\pi}{N}$($N=2^p$ is a power of two), the $m$-th  discrete Fourier coefficients of $f$ are defined as
 \begin{equation}
\widehat{f_m^t}=\frac{1}{N}\sum_{k=-N/2+1}^{N/2}f(\alpha_k,t)e^{-im\alpha_k};\alpha_k= kh,\label{Fouriertransform}
\end{equation}
with inverse Fourier formula given by
 \begin{equation}
f_k^t=\sum_{m=-N/2+1}^{N/2}\widehat{f_m^t}e^{im\alpha_k}.\label{InverseFourier}
\end{equation}

  Non-linear terms are treated in physical space and to avoid convolutions. In other words:
   \begin{equation}
  \widehat{NL_m^t}=\frac{1}{N}\sum_{k=-N/2+1}^{N/2}(S_h\theta)^{(3)}(\alpha_k,t)e^{-im\alpha_k},
 \end{equation}
 Implicit time integration methods can now be easily applied.

  \subsection{Linear propagator method and Adams-Bashforth (ADB)} Based on $(\ref{intfactor2})$, the first step is computed using an Euler implementation and the integrating factor method: 
 \begin{equation}\label{EulerADB}
 \widehat{\theta_m^1}=\zeta_m (\widehat{  \theta_m^0}+\Delta t \widehat{NL_m^0}),
 \end{equation}
 where $\Delta t$ denotes the time step discretization and 
 \begin{equation}
 \zeta_m:=exp(-i\Delta t(m\frac{2\pi}{L})^3). \label{discnl}
\end{equation}

 Subsequent steps are calculated with the second order Adams-Bashforth (ADB) method:
  \begin{equation}\label{discth}
  \widehat{\theta_m^{j+1}}=\zeta_m \widehat{\theta_m^j}+\frac{\Delta t}{2}\left[ 3\zeta_m\widehat{NL_m^j}-(\zeta_m)^2\widehat{NL_m^{j-1}}\right].
  \end{equation}
   Notice how $\widehat{\theta}$ at the $j$th time-step is propagated forward to the next step $(j+1)$ at the exact exponential rate associated with the linear term. If $NL=0$, this yields to the exact solution of the linear problem. In the case of Airy Flow, the length of the curve is constant, thus $\zeta_m$ is constant over time.
  \subsection{Crank-Nicholson (CN): }  It is also possible to discretize $(\ref{eqT})$ using an Euler discretization for the first step
  \begin{equation}\label{EulerCN}
\widehat{{\theta^{1}_m}}=\widehat{{\theta^0_m}}+\Delta t (\widehat{{L^0_m}}+\widehat{{NL^0_m}}),
  \end{equation}
and a Crank-Nicholson-like (CN) method for later steps of the form: 
\begin{equation}\label{FirstCN}
\widehat{\theta_m^{j+1}}- \widehat{\theta_m^{j-1}}=\Delta t (\widehat{Lin^{j+1}_m}+\widehat{Lin^{j-1}_m})+2\Delta t\widehat{NL_m^j},
\end{equation}
where $\widehat{Lin^{j}_m}=S_hS_hS_h\widehat{\theta}_m^j$. 

Defining $\gamma_m=\Delta t(\frac{2\pi m}{L})^3$, and
\begin{equation}
\zeta_m^1:=\frac{1-\gamma_m^2}{1+\gamma_m^2}+i\frac{-2\gamma_m}{1+\gamma_m^2}=\frac{1-i\gamma_m}{1+i\gamma_m},\:\: \zeta_m^2:=\frac{1-i\gamma_m}{1+\gamma_m^2},
\end{equation}
then, $(\ref{FirstCN})$ is equivalent to 
\begin{equation}\label{thetaeqCN}
\widehat{\theta_m^{j+1}}=\zeta_m^1\widehat{\theta_m^{j-1}}+2\Delta t\zeta_m^2\widehat{ NL_m^j},
\end{equation}
for each wave number $m$. 
  \subsection{Crank-Nicholson and Adams-Bashforth (CNADB)} 
 The scheme CNADB is a modification of CN, where the first step after initialization is the average of the schemes used for CN and ADB discretizations, that is
\begin{equation}\label{EulerCNADB}
 \widehat{ \theta_m^1}= \widehat{ \theta_m^0}\frac{1}{2}[e^{-i\gamma_m}+(1-i\gamma_m)]+\widehat{NL_m^0}\frac{\Delta t}{2}[1+e^{-i\gamma_m}].
\end{equation}

\section{Analytical convergence }
\label{sec:2}
To prove the convergence of the presented schemes, we denote to the exact continuous solution evaluated at the grid points by $L,\theta_m^j=\theta(\alpha_m,t_j)$ , and we use $ \widetilde{L}, \widetilde{\theta^n_m}$ for  the discrete approximations. Purely imaginary terms are denoted by $I_j$. For simplicity, we omit the time notation where the specific time is not relevant for the computation.

 We work with the following space of functions:
\begin{equation}
\begin{split}
C^r[0,2\pi]:=& \{f: \text{first  $r$ derivatives exist over $(0,2\pi)$,  are of bounded variation }\\
 &  \text{ over $[0,2\pi]$, and whose first $r-1$ derivatives are $2\pi$-periodic.}\} 
\end{split}
\end{equation}
The existence of the first $r$ derivatives is understood in the almost everywhere Riemann-Stieltjes sense (\cite{canuto}).

 The main tool to handle truncation error is the spectral accuracy of the method. In other words, the Fourier coefficients of any $f\in C^{r}[0,2\pi]$ satisfies the decay condition  (\cite{canuto})
 \begin{equation}\label{specdecay}
\widehat{f_m}=O(\frac{1}{m^{r+1}}),
\end{equation}
which implies (\cite{Eitan}) 
\begin{equation}
|S_hf(\alpha_i)-f_{\alpha}(\alpha_i)|=O(h^{r-1}).
\end{equation}
Also, the accuracy of the trapezoidal rule can be estimated (\cite{Trapezoid}) by
\begin{equation}
|\sum_{j=-N/2+1}^{N/2}|f(\alpha_j)|h-\int_{-\pi}^{\pi}f(\alpha)d\alpha|=O(\frac{1}{N^{r+1}}).
\end{equation}

Approximations are computed with the discrete inner products
\begin{equation}\label{innerproduct}
\langle f,g\rangle_h:=\sum_{m=-N/2+1}^{N/2}hf_m\overline{g_m},\text{ } \langle\hat{f},\hat{g}\rangle :=\sum_{m=-N/2+1}^{N/2} \hat{f}_m\overline{\hat{g}_m},
\end{equation}
 and the associated norms
 \begin{equation}
||f||_{l^2}^2=\sum_{m=-N/2+1}^{N/2}|f_m|^2h\text{, } ||\hat{f}||^2=\sum_{m=-N/2+1}^{N/2}|\hat{f}_m|^2. \label{l2norm}
\end{equation}

An immediate consequence of the trapezoidal rule accuracy is 
\begin{equation}\label{Lerrors}
||f||_{l^2}^2-||f||_{L^2}^2=O(h^{r+1}).
\end{equation}

The key ideas for treating stability error besides algebraic manipulation is Plancherel theorem 
\begin{equation} \label{plancherel} 
\frac{1}{2\pi}\langle f,g \rangle_h=\langle \hat f,\hat g \rangle\Rightarrow ||\widehat{f}||= \frac{1}{\sqrt{2\pi}}||f||_{l^2},
\end{equation}
that allow us to compute inner products at Fourier or Physical space interchangeably.

 The main theoretical results in this section are the following theorems:
\begin{theorem}{}\label{disctheoADB}
Assume that for $0\leq t\leq T$ there exists a regular solution of the $\theta-L$ system of evolution equations $(\ref{eqL})$ and $(\ref{eqT})$ (for \textbf{Airy flow and the mKdV} equation) with $\theta(\cdot,t)$  belonging to $C^{r+3}[0,2\pi]$ for $4\leq r$ and whose second derivative is continuous with respect to time. If $\widetilde{\theta_m^j}$ denotes the numerical solution obtained with the scheme $(\ref{EulerADB})$,$(\ref{discnl})$,$(\ref{discth})$  then for $h\leq h_0(T, k)$ and $\frac{\Delta t}{h}\leq C_0(T,k)$\label{CstabCN} we have,
 \begin{equation}
||\widetilde{\theta^j}-\theta(\cdot,t_j)||_{l^2}\leq C(h^r+\Delta t^2). \label{ThetaConvergence}
\end{equation}
\end{theorem}

\begin{theorem}{}\label{disctheoCN}
Assume that for $0\leq t\leq T$ there exists a regular solution of the $\theta-L$ system of evolution equations $(\ref{eqL})$ and $(\ref{eqT})$ (for \textbf{Airy flow and the mKdV} equation) with $\theta(\cdot,t)$  belonging to $C^{r+3}[0,2\pi]$ for $6\leq r$ and whose third derivative is continuous with respect to time. If $\widetilde{\theta_m^j}$ denotes the numerical solution obtained with the scheme $(\ref{EulerCN})$,$(\ref{thetaeqCN})$ then for $h\leq h_0(T, k)$ and $\frac{\Delta t}{h}\leq C_0(T,k)$ we have,
 \begin{equation}
||\widetilde{\theta^j}-\theta(\cdot,t_j)||_{l^2}\leq C(T)(h^r+\Delta t^2). \label{ThetaConvCN}
\end{equation}
\end{theorem}

\begin{corollary}{}\label{cnadblemma} Assume that for $0\leq t\leq T$ there exists a regular solution of the $\theta-L$ system of evolution equations $(\ref{eqL})$ and $(\ref{eqT})$ (for \textbf{Airy flow and the mKdV} equation) with $\theta(\cdot,t)$  belonging to $C^{r+3}[0,2\pi]$ for $6\leq r$ and whose third derivative is continuous with respect to time. If $\widetilde{\theta_m^j}$ denotes the numerical solution obtained with the scheme $(\ref{thetaeqCN})$,$(\ref{EulerCNADB})$,  then for $h\leq h_0(T, k)$ and $\frac{\Delta t}{h}\leq C_0(T,k)$ we have,
 \begin{equation}
||\widetilde{\theta^j}-\theta(\cdot,t_j)||_{l^2}\leq C(T)(h^r+\Delta t^2). \label{ThetaConvCNADB}
\end{equation}
\end{corollary}

At this point, we introduce the notation that will be used in the error analysis. We define the discrete $n$-th order smoothing operator, written generically $A_{-n}$, as an operator satisfying,
$$||S_h^l(A_{-l}(\dot{\theta^j}))||_{l^2}\leq C||\dot{\theta^j}||_{l^2},\:\:and\:\: ||A_{-n}(S_h^l\dot{\theta^j})||_{l^2}\leq C||\dot{\theta^j}||_{l^2},\:for\:0\leq l\leq n.$$

In particular $||A_0(\dot{\theta^j})||_{l^2}\leq C||\dot{\theta^j}||_{l^2}$ and $h^nA_{0}(\dot{\theta^j})=A_{-n}(\dot{\theta^j})$. For estimates in time we write $A_0(\Delta t^n)$, for an operator satisfying
$$||A_0(\Delta t^n)||_{l^2}\leq \Delta t^n||f||_{l^2},$$
where $f$ is $l^2$ integrable.

The proofs of  theorems (\ref{disctheoADB}) and (\ref{disctheoCN}) are similar. We focus on CN discretization and refer the reader to the appendix $(\ref{discreteproofADB})$ for the ADB case.
\begin{proof}[Theorem \ref{disctheoCN}]
The error between numerical and exact solution (at a given time $j$) is denoted by:
\begin{equation}
\dot{\theta}_m^j:=\widetilde{\theta_m^j}-\theta(\alpha_m,t_j).\label{ThetaerrorVarCN}
\end{equation}

 Defining the auxiliary time 
\begin{equation}\label{timeHyCN}
T^*=Sup\{t|t\leq T,|\dot{L}|<h^{r+3},||\dot{\theta}^j||_{l^2}=O(h^r+\Delta t^2)\},\\
\end{equation}
   for $j=0,1,...,n$. We aim to prove that the error of theta at  the $n+1$ step also satisfies the estimate $||\dot{\theta}^{n+1}||_{l^2}=O(h^r+\Delta t^2)$ and this will imply $T^*=T$ by induction.

\paragraph*{Taylor approximations: \:}  for the first step of the induction argument, we calculate upper bounds for Euler step using the Taylor expansion:
\begin{equation}\label{Estep}
\theta^{1}_m=\theta^0_m+\Delta t (L^0_m+NL^0_m)+\frac{\Delta t^2}{2}(\theta_{tt})^0_m+O(\Delta t^3).
\end{equation}

Similarly, after the second step. The Crank-Nicholson discretization derived from the Taylor expansion of $\theta^{n+1}$ around $\theta^n$ and $\theta^{n-1}$ around $\theta^n$ has the form:
$$\widehat{\theta_m^{j+1}}- \widehat{\theta_m^{j-1}}=2\left(\Delta t\widehat{(\theta_{t})^j_m} +\frac{\Delta t^3}{6}\widehat{(\theta_{ttt})^j_m}\right) +O(\Delta t^4),$$
 where, as usual $f_t$ denotes the temporal derivative of $f$. Using the approximation 
$$\frac{\widehat{Lin^{j+1}_m}+\widehat{Lin^{j-1}_m}}{2}=\widehat{Lin^{j}_m}+\Delta t^2\widehat{{Lin_{tt}}^j_m}+O(\Delta t^5),$$
we obtain

 \begin{align}
  &\widehat{\theta_m^{j+1}}- \widehat{\theta_m^{j-1}}=\\
  &2\left(\Delta t\left( \frac{1}{2}(\widehat{Lin^{j+1}_m}+\widehat{Lin^{j-1}_m})+\widehat{NL_m^j} \right) +\frac{\Delta t^3}{6}(\widehat{(\theta_{ttt})^j_m}-6\widehat{{Lin_{tt}}_m^j})\right) +O(\Delta t^4), 
 \end{align}

where $\widehat{Lin^{j}_m}=S_hS_hS_h\widehat{\dot{\theta}_m^j}$. This is equivalent to 
\begin{equation}\label{mainthetaeq}
\widehat{\theta_m^{j+1}}=\zeta_m^1\widehat{\theta_m^{j-1}}+2\Delta t\zeta_m^2  \widehat{ NL_m^j}+\frac{\Delta t^3}{3}({(\widehat{\theta_{ttt})^j}-6\widehat{{Lin_{tt}}_m^j)}}+O(\Delta t^4),
\end{equation}
where $\gamma_m=\Delta t(\frac{2\pi m}{L})^3$, for each wave number $m$. The numerical solution satisfies
  \begin{equation}\label{schemeE}
\widehat{\widetilde{\theta^{1}_m}}=\widehat{\widetilde{\theta^0_m}}+\Delta t (\widehat{\widetilde{L^0_m}}+\widehat{\widetilde{NL^0_m}}),
  \end{equation}
and
  \begin{equation}\label{schemeCN}
  \widehat{\widetilde{\theta_m^{j+1}}}=\zeta_m^1\widehat{\widetilde{\theta_m^{j-1}}}+2\Delta t\zeta_m^2\widehat{\widetilde{ NL_m^j}},
  \end{equation}
for $j=1,..,n$.

 We start simplifying $(\ref{Estep})$ and $(\ref{mainthetaeq})$ noticing that
\begin{equation}\label{thetatt}
(\theta_t)_t=(\theta_{sss}+\frac{\theta_s^3}{2})_t=\theta_{ssst}+\frac{3}{2}\theta_s^2\theta_{st}=[\theta_{sss}+\frac{\theta_s^3}{2}]_{sss}+[\theta_{sss}+\frac{\theta_s^3}{2}]_s\frac{3}{2}\theta_s^2,
\end{equation}
involves spatial derivatives of order $6$ for theta, and we have used the fact that $\theta$ is at least two times continuously differentiable with respect time to commute derivatives. Similarly, we compute:
\begin{equation}
\begin{split}
((\theta_t)_t)_t &=\left( (\theta_{sss}+\frac{1}{2}\theta_s^3)_{sss}+\frac{3}{2}\theta_s^2(\theta_{sss}+\frac{1}{2}\theta_s^3)_s\right)_t\\
& =\theta_{sssssst}+\frac{1}{2}(\theta_s^3)_{ssst}+\frac{3}{2}\theta_s^2(\theta_{sss}+\frac{\theta_s^3}{2})_{st}+(\theta_{sss}+\frac{1}{2}\theta_s^3)_s\frac{3}{2}2\theta_s\theta_{st}=
\end{split}
\end{equation}
\begin{equation}\label{deriv9}
\theta_{tssssss}+\frac{3}{2}[\theta_s^2\theta_{t ssss}+\theta_{ts}(\theta_s^2)_{sss}]+\frac{3}{2}\theta_s^2[\theta_{tssss}+3\theta_s\theta_{ss}\theta_{st}+\frac{3}{2}\theta_s^2\theta_{tss}]+\theta_s[\theta_{ssss}+(\frac{\theta_s^2}{2})_s]^2,
\end{equation}
note that $\theta_{t\underbrace{ss...s}_{k\:times}}=[\theta_{sss}+\frac{\theta_s^3}{2}]_{\underbrace{ss...s}_{k\:times}}$
involves derivatives of order $k+3$ in space. In the expression $(\ref{deriv9})$ the term $\theta_{tssssss}$ contains the most (9 to be precise) derivatives for $\theta$, which by hypothesis $6\leq r$, we know these are $l^2$ integrable. Using computation $(\ref{thetatt})$, observe that ${Lin_{tt}}_m^j={(\theta_{sss})_{tt}}_m^j={(\theta_{ttsss})}_m^j$ also involves spatial derivatives of order 9 for $\theta$.

If $(\ref{Estep})$ is substracted from $(\ref{schemeE})$ and using the fact that $\theta_{tt}$ is $l^2$ integrable  we obtain the equation for the first step,
\begin{equation}\label{Evar}
\widehat{\dot{\theta}^{1}_m}=\widehat{\dot{\theta}^0_m}+\Delta t (\widehat{\dot{L}^0_m}+\widehat{\dot{NL}^0_m})+A_0(\Delta t^2).
\end{equation}

Since the error at initial step $\dot{\theta}_m^0$ is zero,  Plancherel theorem shows
\begin{equation}\label{Eupbo}
||\widehat{\dot{\theta}^1}||^2=O(\Delta t^4).
\end{equation}

From $(\ref{mainthetaeq})$, $(\ref{schemeCN})$ and the fact that $\theta_{ttt}, Lin_{tt}$ are $l^2$ integrable, the error evolution after the second step is:
\begin{equation}\label{CNvar}
\widehat{\dot{\theta}_m^{j+1}}=\zeta_m^1\widehat{\dot{\theta}_m^{j-1}}+2\Delta t\zeta_m^2  \widehat{\dot{ NL}_m^j}+A_0(\Delta t^3).
\end{equation}

 To estimate the error consider the inner product:
\begin{equation} \label{LHS}
\langle \widehat{\dot{\theta}^{j+1}}-\widehat{\dot{\theta}^{j-1}}, \widehat{\dot{\theta}^{j+1}}+\widehat{\dot{\theta}^{j-1}} \rangle.
\end{equation}

A direct calculation $(\ref{dircalc1})$ shows how to rewrite this inner product as
\begin{equation}\label{RHSj}
\begin{split}
&\underbrace{\langle |\zeta^1|^2( \widehat{\dot{\theta}^{j-1}}-\widehat{\dot{\theta}^{j-3}}) ,\widehat{\dot{\theta}^{j-1}}+\widehat{\dot{\theta}^{j-3}}   \rangle}_{J_1^j}+ \\
& \underbrace{4 \Delta t^2\langle |\zeta^2|^2( \widehat{\dot{NL}^j}+\widehat{\dot{NL}^{j-2}}), \widehat{\dot{NL}^j}-\widehat{\dot{NL}^{j-2}}) \rangle}_{J_2^j}+\\
& \underbrace{\langle \zeta^1(\widehat{\dot{\theta}^{j-1}}-\widehat{\dot{\theta}^{j-3}}) ,2\Delta t \zeta^2(  \widehat{\dot{NL}^j}+\widehat{\dot{NL}^{j-2}})\rangle}_{J_3^j}+\underbrace{\langle2\Delta t \zeta^2(  \widehat{\dot{NL}^j}-\widehat{\dot{NL}^{j-2}}),\zeta^1(\widehat{\dot{\theta}^{j-1}}+\widehat{\dot{\theta}^{j-3}}\rangle}_{J_4^j}+\\
& \underbrace{\langle \zeta^1\left(\widehat{\dot{\theta}^{j-1}} -\widehat{\dot{\theta}^{j-3}}\right)+2\Delta t \zeta^2\left( \widehat{\dot{ NL}^j}- \widehat{\dot{ NL}^{j-2}}\right),A_0(\Delta t^3)\rangle}_{J_5^j}+\\
& \underbrace{\langle A_0(\Delta t^3),  \zeta^1\left(\widehat{\dot{\theta}^{j-1}} +\widehat{\dot{\theta}^{j-3}}\right)+2\Delta t \zeta^2\left( \widehat{\dot{ NL}^j}+ \widehat{\dot{ NL}^{j-2}}\right)\rangle}_{J_6^j}.
\end{split}
\end{equation}

When taking the sum over time of the left-hand side $(\ref{LHS})$ we obtain a telescopic sum
\begin{equation}\label{LHSsumt}
\begin{split}
\sum_{j=2}^n \langle \widehat{\dot{\theta}^{j+1}}-\widehat{\dot{\theta}^{j-1}}, \widehat{\dot{\theta}^{j+1}}+\widehat{\dot{\theta}^{j-1}} \rangle&=\sum_{j=2}^n\left(||\widehat{\dot{\theta}^{j+1}}||^2-|| \widehat{\dot{\theta}^{j-1}}||^2+2iIm(\langle \widehat{\dot{\theta}^{j+1}},  \widehat{\dot{\theta}^{j-1}} \rangle)\right) \\
& =||\widehat{\dot{\theta}^{n+1}}||^2+|| \widehat{\dot{\theta}^{n}}||^2-\left( ||\widehat{\dot{\theta}^{2}}||^2+|| \widehat{\dot{\theta}^{1}}||^2\right)+I_1,
\end{split}
\end{equation}
where $I_1$ is an imaginary term.

Now we analyze the sum over time of the right-hand side terms $(\ref{RHSj})$.

\paragraph*{$J_1$ contribution: \:} a direct calculation shows that
$$J_1^j=|| \widehat{\dot{\theta}^{j-1}}||^2-|| \widehat{\dot{\theta}^{j-3}}||^2+2iIm\langle \widehat{\dot{\theta}^{j-1}}, \widehat{\dot{\theta}^{j-3}}\rangle.$$

Therefore, the sum over time is telescopic too
\begin{equation}\label{upper1}
\sum_{j=2}^{n}J_1^j=|| \widehat{\dot{\theta}^{n-1}}||^2+|| \widehat{\dot{\theta}^{n-2}}||^2-|| \widehat{\dot{\theta}^{1}}||^2-|| \widehat{\dot{\theta}^{0}}||^2+I_2,
\end{equation}
 where $I_2$ is a purely imaginary term.
\paragraph*{$J_2$ contribution:\: }similarly, 
$$J_2^j=4\Delta t^2\left( || |\zeta^2|^2\widehat{\dot{NL}^{j}}||^2-|||\zeta^2|^2\widehat{\dot{NL}^{j-2}}||^2+2iIm(\langle \zeta^2\widehat{\dot{NL}^{j-2}},\zeta^2\widehat{\dot{NL}^{j}} \rangle)\right),$$
and the sum over time is
$$\sum_{j=2}^nJ_2^j=4\Delta t^2\left(  || |\zeta^2|^2\widehat{\dot{NL}^{n}}||^2+|||\zeta^2|^2\widehat{\dot{NL}^{n-1}}||^2-|||\zeta^2|^2\widehat{\dot{NL}^{2}}||^2-|||\zeta^2|^2\widehat{\dot{NL}^{1}}||^2 +I_3\right),$$
 where $I_3$ is a purely imaginary term. And the following inequality holds
 \begin{equation}\label{upper2}
 Re(\sum_{j=2}^nJ_2^j)\leq 4\Delta t^2( || |\zeta^2|^2\widehat{\dot{NL}^{n}}||^2+|||\zeta^2|^2\widehat{\dot{NL}^{n-1}}||^2).
 \end{equation}
\paragraph*{$J_3+J_4$ contribution: \:} consider the sum
\begin{equation}
\begin{split}
& (J_3+J_4)^j=\\
& \langle \zeta^1\widehat{\dot{\theta}^{j-1}},2\Delta t \zeta^2 \widehat{\dot{NL}^j}\rangle+\langle \zeta^1\widehat{\dot{\theta}^{j-1}},2\Delta t \zeta^2 \widehat{\dot{NL}^{j-2}}\rangle\\
& -\langle \zeta^1\widehat{\dot{\theta}^{j-3}},2\Delta t \zeta^2 \widehat{\dot{NL}^j}\rangle-\langle \zeta^1\widehat{\dot{\theta}^{j-3}},2\Delta t \zeta^2 \widehat{\dot{NL}^{j-2}}\rangle\\
& +\langle 2\Delta t \zeta^2 \widehat{\dot{NL}^j}, \zeta^1\widehat{\dot{\theta}^{j-1}}\rangle+\langle 2\Delta t \zeta^2 \widehat{\dot{NL}^j}, \zeta^1\widehat{\dot{\theta}^{j-3}}\rangle\\
& -\langle 2\Delta t \zeta^2 \widehat{\dot{NL}^{j-2}}, \zeta^1\widehat{\dot{\theta}^{j-1}}\rangle-\langle 2\Delta t \zeta^2 \widehat{\dot{NL}^{j-2}}, \zeta^1\widehat{\dot{\theta}^{j-3}}\rangle\\
& =2Re\left( \langle \zeta^1\widehat{\dot{\theta}^{j-1}},2\Delta t \zeta^2 \widehat{\dot{NL}^j}\rangle-\langle \zeta^1\widehat{\dot{\theta}^{j-3}},2\Delta t \zeta^2 \widehat{\dot{NL}^{j-2}}\rangle \right)\\
& +2iIm\left(\langle \zeta^1\widehat{\dot{\theta}^{j-1}},2\Delta t \zeta^2 \widehat{\dot{NL}^{j-2}}\rangle+\langle 2\Delta t \zeta^2 \widehat{\dot{NL}^j}, \zeta^1\widehat{\dot{\theta}^{j-3}}\rangle \right). 
\end{split}
\end{equation}

Therefore, the sum over time is also a telescopic sum
\begin{equation}\label{upper3}
\begin{split}
\sum_{j=2}^n(J_3+J_4)^j &=2Re( \langle \zeta^1\widehat{\dot{\theta}^{n-1}},2\Delta t \zeta^2 \widehat{\dot{NL}^n}\rangle+ \langle \zeta^1\widehat{\dot{\theta}^{n-2}},2\Delta t \zeta^2 \widehat{\dot{NL}^{n-1}}\rangle) +\\
& 2Re(-\langle \zeta^1\widehat{\dot{\theta}^{1}},2\Delta t \zeta^2 \widehat{\dot{NL}^2}\rangle- \langle \zeta^1\widehat{\dot{\theta}^{0}},2\Delta t \zeta^2 \widehat{\dot{NL}^1}\rangle)+I_4,
\end{split}
\end{equation}

 where $I_4$ is a purely imaginary term.

\paragraph*{$J_5+J_6$ contribution: \:} first notice that
\begin{equation}|J_5|\leq ||\zeta_m^1\left(\widehat{\dot{\theta}_m^{j-1}} -\widehat{\dot{\theta}_m^{j-3}}\right)+2\Delta t \zeta_m^2\left( \widehat{\dot{ NL}_m^j}- \widehat{\dot{ NL}_m^{j-2}}\right)|| C \Delta t^3.\end{equation}

According to $(\ref{mainthetaeq})$ and computation $(\ref{deriv9})$ the constant $C$ depends on spatial derivatives of order 9 for $\theta$. A sharper and technical expression for  the error of the nonlinear term is required $(\ref{DtupperNL})$. The following lemma is proved in Appendix $(\ref{proofnlvar})$.
\begin{lemma}{}\label{nlvarlemma} The nonlinear error satisfies the upper bounds
 \begin{equation}
   |\dot{NL}^j|_{\infty}=h^{-1/2}||\dot{NL}^j||_{l^2}\leq C,
 \end{equation}
 and
   \begin{equation}\label{DtupperNL}
\Delta t ||\dot{NL}^j||_{l_2}=O(h^{r}+\Delta t^2), 
  \end{equation}
  for $j=1,...,n$ provided that $\frac{\Delta t}{h}$ is bounded. 
\end{lemma}

  Then, using $|\zeta_m^1|=1$ and $|\zeta_m^2|\leq 1$, induction  and the previous estimate we get
$$|J_5|\leq  \left(  ||\left(\widehat{\dot{\theta}^{j-1}} -\widehat{\dot{\theta}^{j-3}}\right)||+ ||\zeta^2||2\Delta t|| \widehat{\dot{ NL}^j}- \widehat{\dot{ NL}^{j-2}}||\right) C \Delta t^3= \Delta t^3O(h^r+\Delta t^2).$$

In a similar way $|J_6|\leq \Delta t^3O(h^r+\Delta t^2)$. Therefore, 
\begin{equation} \label{upper4}
|\sum_{j=2}^n(J_5^j+J_6^j)|\leq  \Delta t^2O(h^r+\Delta t^2).
\end{equation}

With this information $(\ref{upper1})$,$(\ref{upper2})$,$(\ref{upper3})$,$(\ref{upper4})$, and considering the real part of $(\ref{LHSsumt})$ we obtain
\begin{equation}\label{lasteq}
\begin{split}
&||\widehat{\dot{\theta}^{n+1}}||^2+|| \widehat{\dot{\theta}^{n}}||^2-\left( ||\widehat{\dot{\theta}^{2}}||^2+|| \widehat{\dot{\theta}^{1}}||^2\right) \leq \\
& || \widehat{\dot{\theta}^{n-1}}||^2+|| \widehat{\dot{\theta}^{n-2}}||^2-|| \widehat{\dot{\theta}^{1}}||^2-|| \widehat{\dot{\theta}^{0}}||^2\\
& +4\Delta t^2\left(  || |\zeta^2|^2\widehat{\dot{NL}^{n}}||^2+|||\zeta^2|^2\widehat{\dot{NL}^{n-1}}||^2 \right)\\
& +2Re\left( \langle \zeta^1\widehat{\dot{\theta}^{n-1}},2\Delta t \zeta^2 \widehat{\dot{NL}^n}\rangle+ \langle \zeta^1\widehat{\dot{\theta}^{n-2}},2\Delta t \zeta^2 \widehat{\dot{NL}^{n-1}}\rangle\right)\\
&2Re\left(-  \langle \zeta^1\widehat{\dot{\theta}^{1}},2\Delta t \zeta^2 \widehat{\dot{NL}^2}\rangle- \langle \zeta^1\widehat{\dot{\theta}^{0}},2\Delta t \zeta^2 \widehat{\dot{NL}^1}\rangle\right)+ \Delta t^2O(h^r+\Delta t^2).
\end{split}
\end{equation}

From $(\ref{Eupbo})$, we get $||\widehat{\dot{\theta}^1}||^2=O(\Delta t^4)$. For the second step $\widehat{\dot{\theta}_m^{2}}=\zeta_m^1\widehat{\dot{\theta}_m^{0}}+2\Delta t \zeta_m^2  \widehat{\dot{ NL}_m^1}+A_0(\Delta t^3)$, we use the estimate $(\ref{DtupperNL})$ to get
\begin{equation}\label{secondstepvar}
||\widehat{\dot{\theta}^2}||^2=O((h^r+\Delta t^2)^2).
\end{equation}

Then by induction, we find that
$$4\Delta t^2\left(  || |\zeta^2|^2\widehat{\dot{NL}^{n}}||^2+|||\zeta^2|^2\widehat{\dot{NL}^{n-1}}||^2\right),||\widehat{\dot{\theta}^n}||^2,||\widehat{\dot{\theta}^{n-1}}||^2,||\widehat{\dot{\theta}^{n-2}}||^2=O((h^r+\Delta t^2)^2).$$

Also, using  Cauchy-Schwarz, triangle inequality, and Plancherel theorem we find
$$| \langle \zeta_m^1\widehat{\dot{\theta}_m^{n-1}},2\Delta t \zeta_m^2 \widehat{\dot{NL}^n}\rangle+ \langle \zeta_m^1\widehat{\dot{\theta}_m^{n-2}},2\Delta t \zeta_m^2 \widehat{\dot{NL}^{n-1}}\rangle|=O((h^r+\Delta t^2)^2),$$
$$| \langle \zeta_m^1\widehat{\dot{\theta}_m^{1}},2\Delta t \zeta_m^2 \widehat{\dot{NL}^2}\rangle|\leq \Delta t^2O(h^r+\Delta t^2).$$

Finally, from $(\ref{lasteq})$ we conclude
\begin{equation}||\widehat{\dot{\theta}^{n+1}}||^2\leq O(\Delta t^4)+O((h^r+\Delta t^2)^2)+\Delta t^2O(h^r+\Delta t^2)=O((h^r+\Delta t^2)^2).\end{equation}

Therefore $||\dot{\theta}^{n+1}||_{l^2}=O(h^r+\Delta t^2)$. As a consequence, the upper bound holds for a longer time ($j=n+1$) than $T^*$ $(\ref{timeHyCN})$, and $T^*=T$ as desired.{\hfill\ensuremath{\square}}
\end{proof}
\begin{proof}[Proof of Corollary \ref{cnadblemma}]\label{proofcnadb}
The average of the upper bound for CN scheme $(\ref{Eupbo})$, and ADB scheme $(\ref{vt2a})$ can be applied as upper bound for this $(\ref{EulerCNADB})$ scheme. In this way, the proof of convergence for CN implies the convergence for CNADB.{\hfill\ensuremath{\square}}
  \end{proof}
 \section{Linear Analysis}
  \label{sec:4}
A first accuracy test can be derived comparing the numerical results with a linear approximation of the solution. Consider the dynamics of nearly circular planar curve of the form:
\begin{equation}\label{linearX}
X(\alpha,t)=r(\alpha,t)(\cos \alpha,\sin\alpha),\:\:r(t)=R(t)+\delta_R \cos(m\alpha)-\delta_I\sin(m\alpha),
\end{equation}
where $m\in\mathbb Z$ is the wave number, $\delta_R(t)$, $\delta_I(t)$ are perturbations and $\alpha\in [0,2\pi]$.

Using equation $(\ref{linearX})$ in Airy flow $(\ref{Airy})$ and equation $(\ref{velproj})$ we find that  
\begin{equation}\label{linone}
\begin{split}
X_t&=\left( R_t +(\delta_R)_t \cos(m\alpha)-(\delta_I)_t \sin(m\alpha) \right)(\cos \alpha,\sin \alpha)\\
& =\delta_I \tau \cos(m\alpha)+ \delta_R \tau \sin(m\alpha)+O(\delta^2),
\end{split}
\end{equation}
where $f_t$ denotes the temporal derivative of $f$ and
\begin{equation}
\tau=\frac{(m^3-1.5m)}{R^3}.
\end{equation}

Matching the left and right hand sides of equations $(\ref{linearX})$ and $(\ref{linone})$ yields $R_{t}=0$ and ${\delta_R}_{t t}=\tau {\delta_I}_{t}=-\tau^2\delta_R$, which can be solved exactly to give:
  \begin{equation}
   \delta_R=-\delta_I(0)\sin(\tau t)+\delta_R(0)\cos(\tau t), \label{deltaR}
 \end{equation}
   \begin{equation}
  \delta_I=\delta_I(0)\cos(\tau t)+\delta_R(0)\sin(\tau t), \label{deltaI}
 \end{equation}
and $R(t)=R_0$.  For simplicity, we take $\delta_I(0)=0$ and denote $\delta_0=\delta_R(0)$ to obtain the linear evolution for Airy flow 
  \begin{equation}\label{linearxy}
   X_L(\alpha,t)=r_L(\alpha,t)(\cos \alpha,\sin\alpha),\:\:  r_L(t)=R_0+\delta_0 cos (\tau t+m\alpha).
 \end{equation}
 
Also, the curvature is:
  \begin{equation}\label{lineark}
  k_L(\alpha,t)=\frac{1}{R_0}+\frac{m^2-1}{R_0^2}\delta_0cos(\tau t+m \alpha)+O(\delta_0^2).
  \end{equation}
  
  From the numerical solution, we may calculate the corresponding radius $\widetilde{R_N}$ and perturbation $\widetilde{\delta_N}$
 
Now we recover the numerical perturbation $\widetilde{\delta_N}$ and  radius $\widetilde{R_N}$ using the approximations: 
\begin{equation}\label{RecRad}
\widetilde{R_N}\approx \sqrt{\frac{\widetilde{Area}}{\pi}},\:\: \widetilde{Area}=\frac{1}{2}\int (x,y)\cdot \textbf{n}s_{\alpha}d\alpha,
\end{equation}
 and 
 \begin{equation}\label{RecPer}
 \widetilde{\delta_N}\approx max_{\alpha}(\sqrt{x^2+y^2}-\widetilde{R_0}).
 \end{equation}

\begin{figure}[H]
\begin{multicols}{2}
\begin{minipage}[ht]{1\textwidth}{}
    \begin{subfigure}[b]{0.5\textwidth}
  \includegraphics[width=\linewidth]{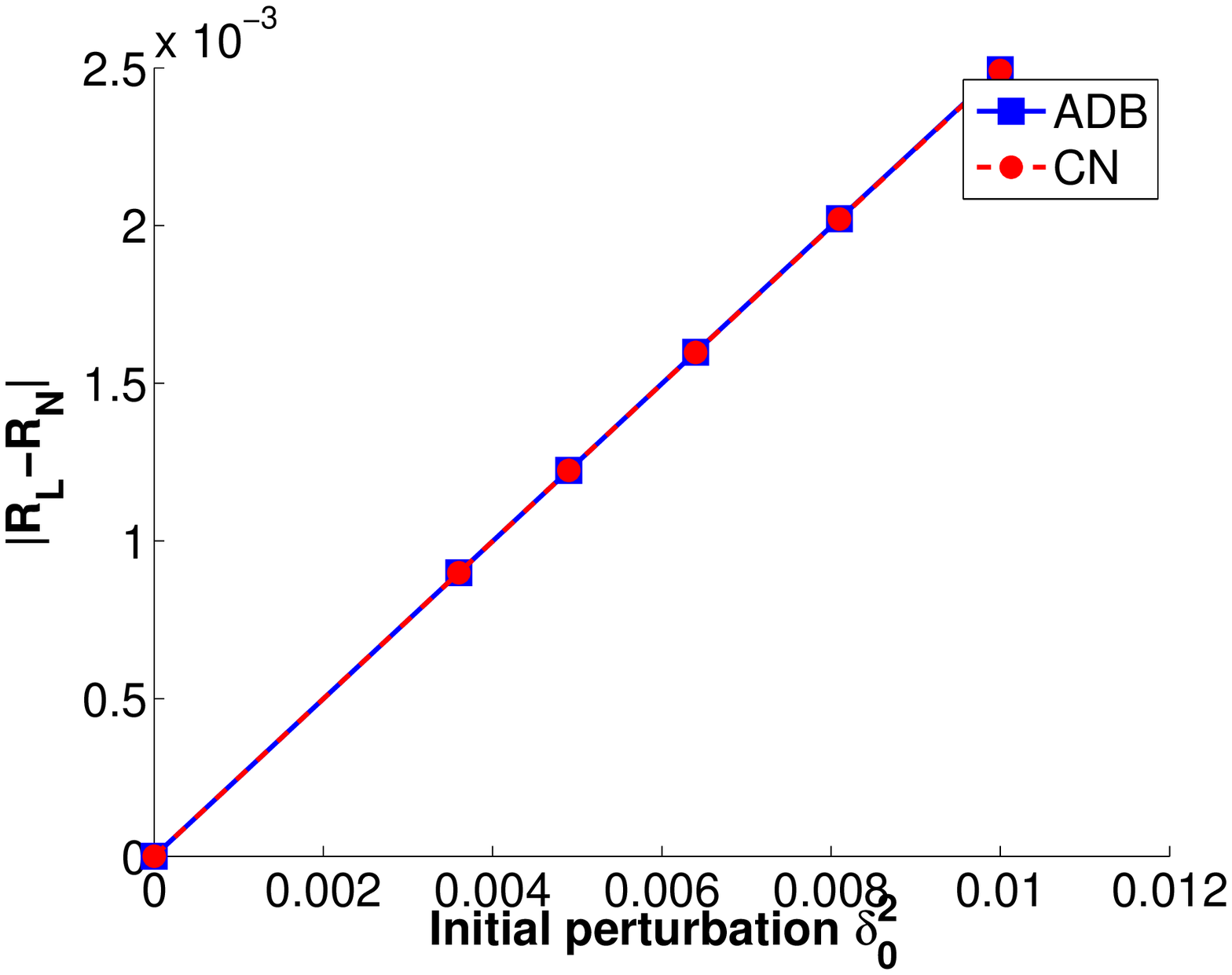}
      \caption{A comparison of the radii from linear analysis and nonlinear simulations of a perturbed circle}
  \label{fig:LinR}
  \end{subfigure}
\end{minipage}

\columnbreak

\begin{minipage}[ht]{1\textwidth}{}
    \begin{subfigure}[b]{0.5\textwidth}
  \includegraphics[width=\linewidth]{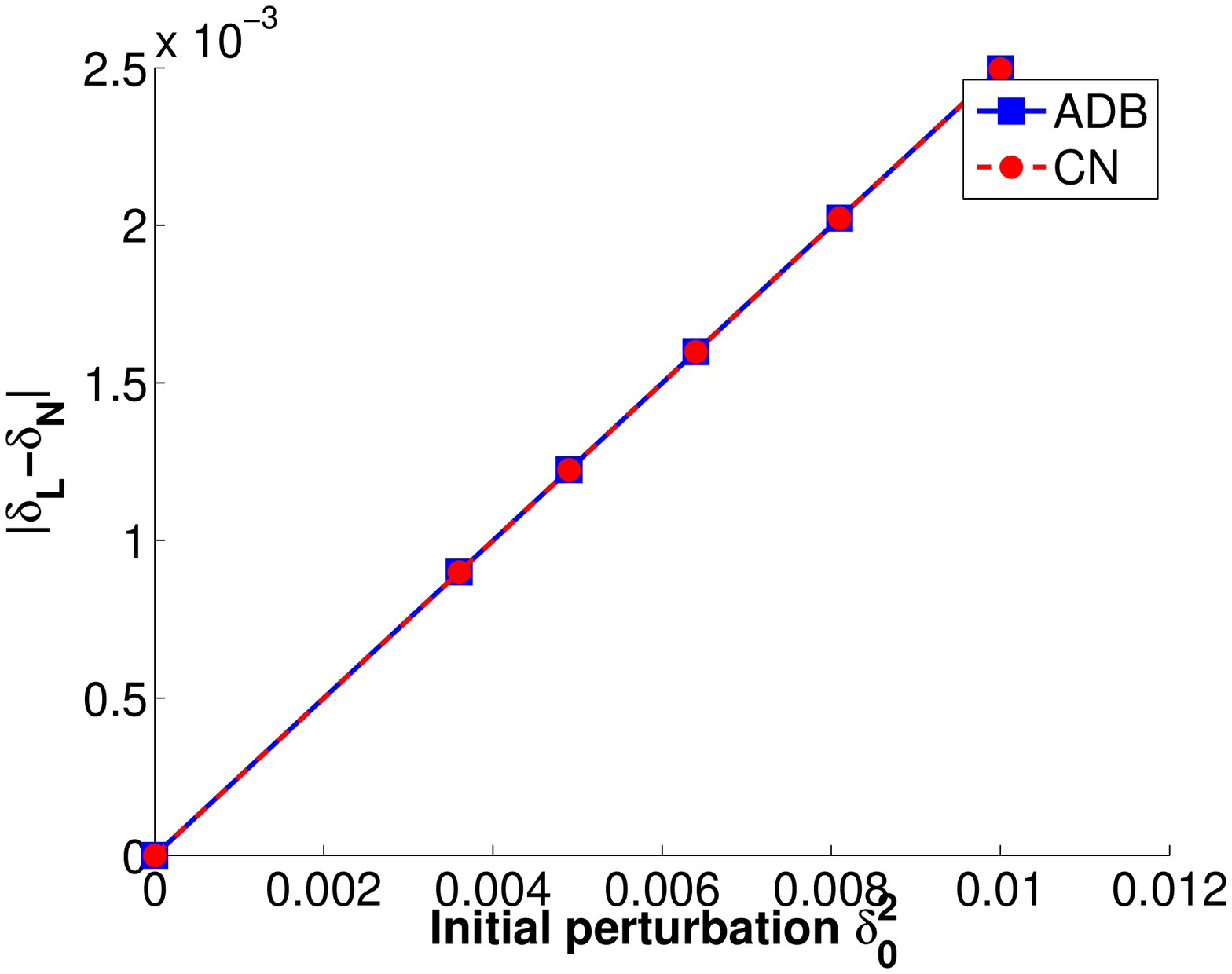}
    \caption{A comparison of the perturbation size from linear analysis and nonlinear simulations of a perturbed circle}
  \label{fig:Lind}
    \end{subfigure}
\end{minipage}
\end{multicols}
\caption{Linear analysis solutions computed at time $t=.1$, using $N=512,\Delta t=1\times 10^{-3}$, initial perturbations  $\delta_0\in\{0,.06,.07,.08,.09,.1\}$, and the wavelength $m=2$ was perturbed }
\label{Linm2}
\end{figure}

In figures $(\ref{fig:LinR})$ and $(\ref{fig:Lind})$  we show a comparison between the analytical and numerical results computed with ADB and CN schemes using an initial radius $R_0=1$. The results confirm accuracy up to order two with respect to the original perturbation size $\delta_0$ for small perturbations at early times
\begin{equation}\label{linerror}
\delta_L-\widetilde{\delta_N}\approx O(\delta_0^2),\:\: R_L-\widetilde{R_N}\approx O(\delta_0^2).
\end{equation}

This provides numerical evidence that the numerical solution is converging to the correct analytical solution at early times. For longer periods of time or bigger perturbations, the error $(\ref{linerror})$ displayed more variation due nonlinear interactions. This variation  remained bounded by a factor depending on the original perturbation size.

\section{Numerics }
\label{sec:5}
As a second accuracy test, we consider the numerical conservation of three first integrals of motion for the mKdV equation (quantities that must be preserved over time, see computation (\ref{conservq})):
\begin{equation}\label{firstint}
\begin{split}
&M1=\int kds,\\
&M2=\int k^2ds,\\
&M3=\int (\frac{1}{2}k_s^2-\frac{1}{8}k^4)ds.
\end{split}
\end{equation}

Physically, $M1$, $M2$, and $M3$ can be interpreted as the corresponding mass, momentum and energy (\cite{MiuraKdV}, \cite{KdVConservLaws}, \cite{wwpoub}) of the system.  We focus on $M3$ as it is more sensitive to the choice of numerical scheme.

\subsection{Convergence and numerical conservation} 
We analyze the accuracy, stability, and convergence of the code by considering the evolution of several curves, starting with an ellipse:
\begin{equation}\label{Eeq2}
E:= (x(\alpha,0),y(\alpha,0))=(cos \alpha,\frac{1}{2}\sin \alpha),\alpha\in [0,2\pi].
\end{equation} 

In figure $(\ref{fig:Exy})$ we observe the evolution of this initial condition $(\ref{Eeq2})$ over a period of time $T=2$ under Airy flow. The mKdV solutions are obtained from the curvature of the curve in figure $(\ref{fig:Ek})$.

\begin{figure}[H]
\begin{multicols}{2}
\begin{minipage}[ht]{1\textwidth}{}
    \begin{subfigure}[b]{0.45\textwidth}
  \fbox{\includegraphics[width=1\linewidth]{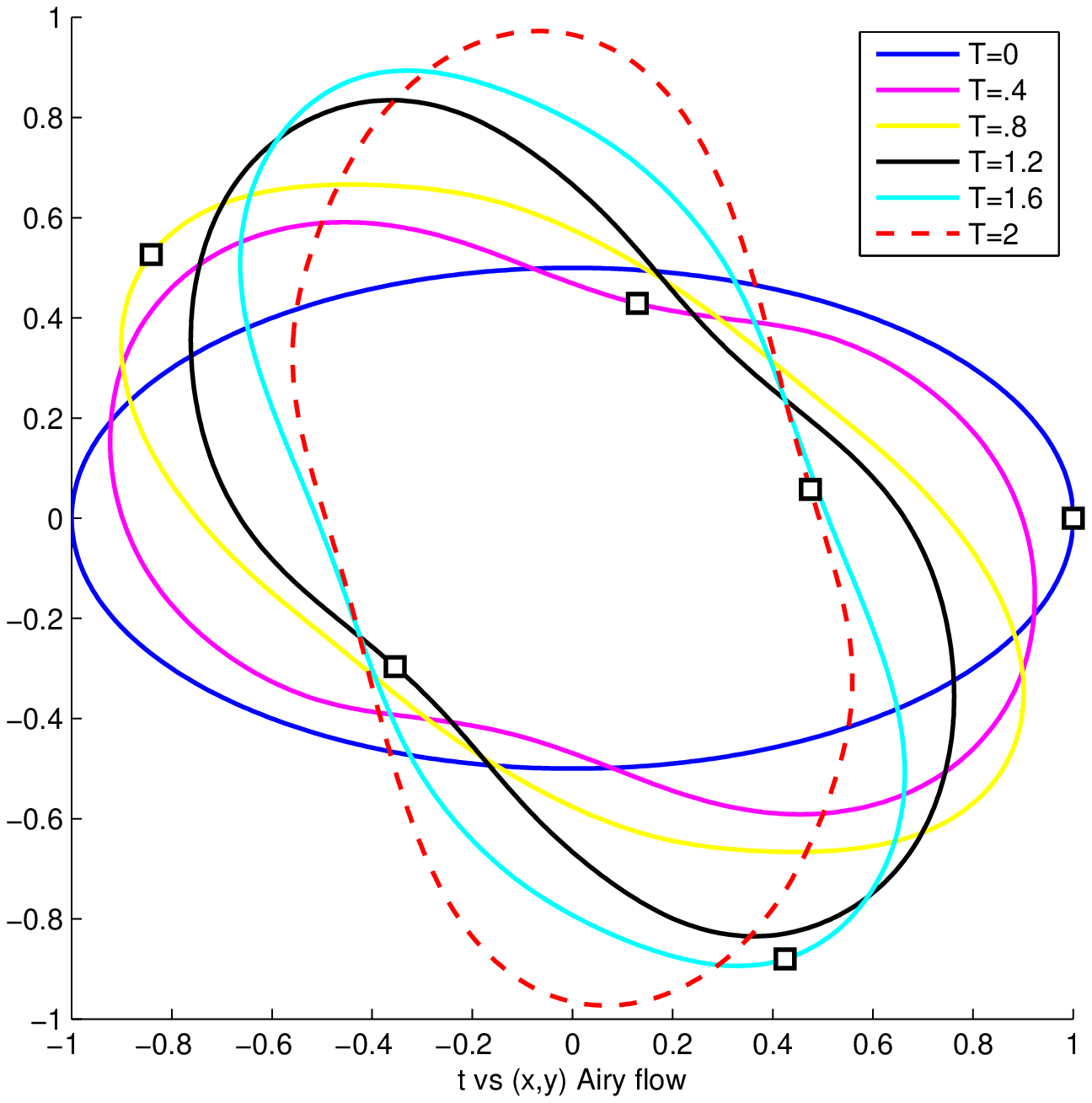}{}}
    \caption{ Airy flow evolution}
  \label{fig:Exy}
  \end{subfigure}
\end{minipage}

\columnbreak

\begin{minipage}[ht]{1\textwidth}{}
    \begin{subfigure}[b]{0.45\textwidth}
  \fbox{\includegraphics[width=1\linewidth]{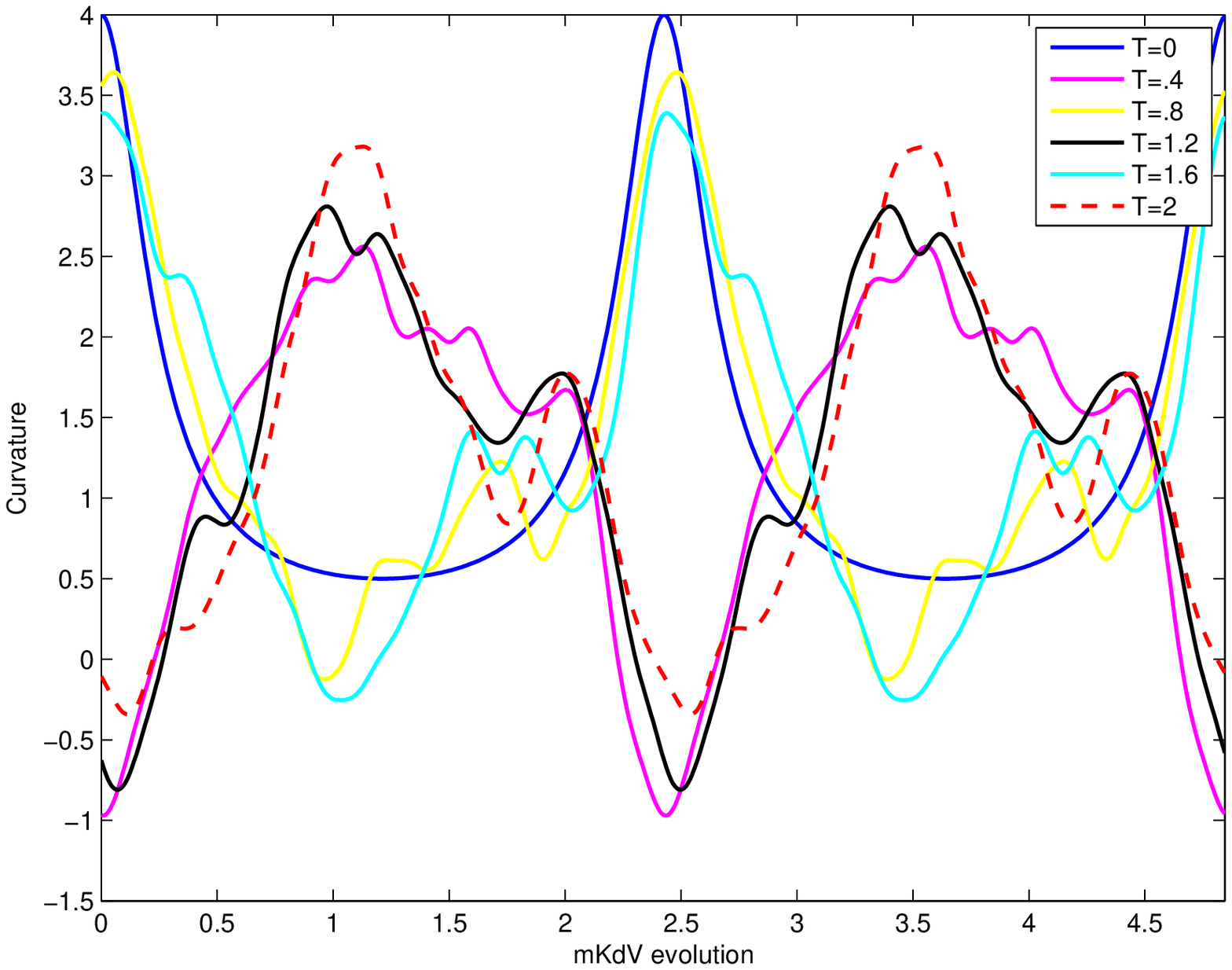}{}}
      \caption{mKdV evolution}
  \label{fig:Ek}
    \end{subfigure}
\end{minipage}
\end{multicols}
  \caption{Morphology evolution for the interface $(\ref{fig:Exy})$ and its curvature $(\ref{fig:Ek})$ with an ellipse $(\ref{Eeq2})$ as initial configuration. Results computed using the CNADB code, $N=512$, and $\Delta t=5\times 10^{-4}$}
  \label{figElmKdV} 
\end{figure}

 The analysis of convergence in time for the ellipse (E) given in table $(\ref{tconv})$ shows that the three schemes converge with second order accuracy. ADB is the most accurate implementation at early times. However, as shown in figure $(\ref{fig:ETH3})$ this scheme also exhibits instabilities (discussed later). For this reason smaller time steps were used to compute ADB than those used for CN and CNADB schemes. The figure $(\ref{fig:ETH3})$, shows the errors over time for the CN, ADB and CNADB schemes. Observe how CNADB combines the accuracy of the ADB scheme with the stability of the CN scheme. In table $(\ref{sconv})$ convergence in space is analyzed and the schemes are found to converge  with spectral accuracy  \footnotemark\footnotetext[1]{The error is obtained with norm $(\ref{Lerrors})$}. Results are dominated by temporal errors.

\begin{figure}[H]
\minipage{1\textwidth}
  \hbox{\includegraphics[width=\linewidth]{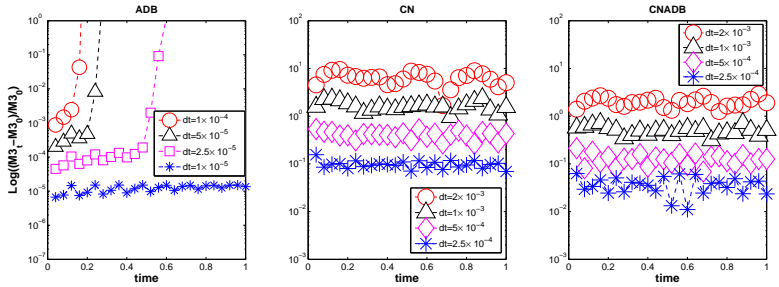}{}}\hfill
\endminipage
  \caption{The relative error in $M_3$ for the ADB, CN, CNADB schemes. The methods use $N=512$ for an ellipse $(\ref{Eeq2})$ as initial configuration}
\label{fig:ETH3}
\end{figure}

\begin{table}[H]
			\caption{Convergence in time. Computed with $N=512$ points in space. The perturbed circle (PC) $(\ref{pcm3})$, was computed with $\Delta t\in\{2\times 10^{-5}, 1\times 10^{-5},5\times 10^{-6}\}$ for the ADBDPR (ADB with DPR filter $(\ref{Filter}),(\ref{Filter2})$), CN schemes, while using $\Delta t\in\{1\times 10^{-5}, 5\times 10^{-6},2.5\times 10^{-6}\}$ for CNADB. The cardioid (C) was calculated with $\Delta t\in\{4\times 10^{-4}, 2\times 10^{-4},1\times 10^{-4}\}$ for ADBDPR and CNADB schemes, while using $\Delta t\in\{2\times 10^{-4}, 1\times 10^{-4},5\times 10^{-5}\}$ for CN. The ellipse (E) was computed with $\Delta t\in\{1\times 10^{-4},5\times 10^{-5},2.5\times 10^{-5}\}$ for ADB, while using $\Delta t\in\{2\times 10^{-3},1\times 10^{-3},5\times 10^{-4}\}$ for CN, and CNADB schemes}
	\label{tconv}
\begin{center}			
	\begin{tabular} {|c|c|c|c|c|c|c|}
		\hline
			Curve &Scheme  & $t_0$ &$||\theta_{\Delta t}-\theta_{\Delta t/2}||_{l^2}$ & $||\theta_{\Delta t/2}-\theta_{\Delta t /4}||_{l^2}$ & $Log_2(\frac{||\theta_{\Delta t}-\theta_{\Delta t/2}||_{l^2}}{||\theta_{\Delta t/2}-\theta_{\Delta t/4}||_{l^2}})$\\  
		\hline
			E &ADB& .16&7.13971e-05 & 1.36433e-05 & 2.38767     \\ 
			 E & CN & .92 &0.670386 & 0.180224 &  1.8952  \\ 
			 E & CNADB & .92 &   0.297091 & 0.0823008  &  1.85193  \\
		\hline
			PC &ADBDPR& .5& 0.0667705  & 0.0113749 & 2.55335 \\ 			
			 PC & CN & 2.8 &0.474292 & 0.119589& 1.98769   \\ 			 
			 PC & CNADB & 4.4 & 0.277312 & 0.0509842&  2.44339  \\ 
		\hline
			C &ADBDPR& .5& 0.0285535 & 0.0068092 & 2.06811  \\ 
			 C & CN & .5 & 0.43557 & 0.107085 &  2.02415  \\ 
			 C & CNADB & .5 & 0.584338 & 0.10983  &  2.41153  \\ 
		\hline		
		\end{tabular}
\end{center}
\end{table}

 \begin{table}[H]
 		        \caption{Convergence in space. The perturbed circle (PC) was computed with $\Delta t=5\times 10^{-6}$, and  $N\in\{512,1024,2048\}$ for the CN,CNADB schemes, while using  $N\in\{256,512,1024\}$ for the ADBDPR (ADB with DPR filter $(\ref{Filter}),(\ref{Filter2})$). The cardioid (C) was calculated with $N\in\{256,512,1024\}$, $\Delta t=5\times 10^{-5}$.  The ellipse (E) was computed with $N\in\{128,256,512\}$, $\Delta t=5\times 10^{-4}$ for the CN CNADB schemes, while using  $\Delta t=1\times 10^{-5}$ for ADB}
			\label{sconv}
\begin{center}
	\begin{tabular} {|c|c|c|c|c|c|c|c|}
		\hline
			Curve& Scheme & $t_0$&  $||\theta_h-\theta_{h/2}||_{l^2}$ & $||\theta_{h/2}-\theta_{h/4}||_{l^2}$  & $Log_2(\frac{||\theta_{h}-\theta_{h/2}||_{l^2}}{||\theta_{h/2}-\theta_{h/4}||_{l^2}})$ \\  
		\hline
			E &ADB& .1& 1.93074e-10 & 8.56718e-13&  7.81612  \\ 			
			 E & CN &1 &  8.1654e-08 & 2.20176e-11 &  11.8566   \\ 				 
			 E & CNADB & 1 &  4.22694e-08 & 1.09587e-11   &  11.9133 \\  			
		\hline
			PC&ADBDPR& 1 & 0.0132975 & 0.000542068  &4.61654  \\ 
			 PC &CN & 4.5 & 0.146292 & 2.98049e-07   & 18.9049   \\ 
			   PC &CNADB & 4.5 & 0.0154042 & 8.14715e-08 &17.5286  \\ 			  
		\hline
			C &ADBDPR& 1 & 0.000452806 & 1.51282e-05& 4.90358  \\ 
			C &CN & 5 & 4.46467e-05 & 1.61903e-08 &   11.4292  \\ 
			C &CNADB & 5 &3.18072e-05 & 4.16487e-09 &  12.8988   \\ 			 
		\hline			
		\end{tabular}
\end{center}
		
\end{table}


\subsection{Influence of curvature} \label{LvsNL} 
We considered the evolution of other two ellipses over a period of time $T=2$ to test the dependence of the temporal discretization size $\Delta t$ with respect the curvature and spatial resolution $h$,
\begin{equation}\label{Eeq13}
\begin{split}
E_1&:= (x(\alpha,0),y(\alpha,0))=(\cos \alpha,\frac{\sqrt 2}{2}\sin \alpha),\alpha\in [0,2\pi],\\
 E_2&:= (x(\alpha,0),y(\alpha,0))=(\cos \alpha,\frac{\sqrt{\sqrt 2}}{2}\sin \alpha),\alpha\in [0,2\pi].
\end{split}
\end{equation} 

 \begin{table}[H]
 	\caption{Accuracy analysis for three ellipses over a period of time $T=2$ with CNADB scheme. Columns $\Delta t_i$'s represent the time-discretization used to evolve ellipse $E_i$ with the number of points in space indicated on each row. The $\xi_i$'s represent the maximum relative error in $M_3$ over a period of time $0\leq T\leq 2$}
	\label{Es}
 \begin{center}
	\begin{tabular} {|c|c|c|c|c|c|c|c|}
		\hline
			N & $\Delta t_1$   & $\xi_1$ & $\Delta t_2$  & $\xi_2$ & $\Delta t_3$ & $\xi_3$  \\
		\hline
			256 & .001 & .018 & $5\times 10^{-4}$ & .03 & $2.5\times 10^{-4} $& .045   \\ 
			512 &$ 5\times 10^{-4}$ &.01& $2.5\times 10^{-4}$ & .012& $1.25\times 10^{-4} $&.022 \\ 
		       1024 & $2.5\times 10^{-4}$& $5\times 10^{-3}$  & $1.25\times 10^{-4}$&$ 4.5\times 10^{-3}$ & $6.25\times 10^{-5}$ & $7\times 10^{-3}$\\ 
		\hline
			\end{tabular}
\end{center}
\end{table}
 The maximum value of the square of the initial curvature $(max|k^0|^2)$ is 16 for $E=E_3$, 8 for E2, and 4 for E1. We fix an interval of evolution for these curves $T=2$. Dynamics at the interface changes for different initial configurations. A particle at the interface of the ellipse $E1$ completes less than half of a loop around its center of mass. While a particle at  the interface of $E3$ covers a complete loop in the same period of time. In each row of table $(\ref{Es})$, the number of points in space $N$ is fixed and the size of the time discretization is adjusted according to $\Delta t\approx (1/N)/(max|k^0|^2)$. Observe that the errors $\xi_i=({M_3}_t-{M_3}_0)/{M_3}_0$ are about the same order in each case. Increasing the number of points at the interface by a factor of two (moving between consecutive rows) corresponds to reducing by half the time step size to attain errors of order $O(10^{-2})$ or smaller. These results confirm the theoretical linear constraint  $\Delta t\leq Ch$ between space-time discretizations, where $C$ has (at least) a direct dependence with the square of the overall curvature $(\ref{upperNL})$ during this period of time. In fact the proof indicates $(\ref{Evar}),(\ref{CNvar})$ that C depends on derivatives of order 9 for $\theta$.

\subsection{Stability and filters} 
Instabilities may arise for shapes with larger curvatures, or longer computations in time due to aliasing error, high order derivatives and nonlinearities involved (\cite{Eitan}, \cite{hussaini}). These instabilities have a stronger effect on the ADB scheme. The analysis of convergence in time for ADB scheme (see figure $(\ref{fig:ETH3})$, and table $(\ref{tconv})$) show that reducing the time steps helps to regulate these errors. We found that a more efficient way to control these instabilities and retain accuracy is to apply a high-order Fourier filter to the nonlinear term. The combination of the spectral derivative and the filtering is denoted by $D_h$, (\cite{Ceniceros}):
\begin{equation}
\widehat{D_h(f_m)}=\widehat{S_h(f_m)^p}:=im\rho(\frac{mh}{\pi}\widehat{f_m})\widehat{f_m},\label{filterd}
\end{equation}
for $m=-\frac{N}{2}+1,..,\frac{N}{2}$.

The filter has two components $\rho(\frac{mh}{\pi}\widehat{f_m})=\rho_1(\frac{mh}{\pi})\rho_2(\widehat{f_m})$. The first one is defined as
 \begin{equation}\label{Filter}
\rho_1(x)=\begin{cases}
     e^{   1-\frac{1}{16(1+x)^4} }   & \text{if $x\in[-1,-.5]$}\\
    1 & \text{if $x\in(-.5,.5)$}\\
        e^{   1-\frac{1}{16(1-x)^4} }     & \text{if $x\in[.5,1]$}\\
    0 & \text{otherwise},\\
  \end{cases}
\end{equation}
which  damps out higher mode instabilities (\cite{Lo1}). Motivated by the time-continuous and spatially discrete analysis (\cite{JH}, \cite{Ceniceros}), the filter $(\ref{Filter})$ is used to stabilize aliasing error from the discrete product rule, we call this DPR filter. Following (\cite{JH}, \cite{Ceniceros}), $\rho_1$ is constructed to satisfy the following:  $\rho_1(x)=\rho_1(-x)$,  $\rho_1(\pm1)=0$, $\rho_1$ is positive, and is $C^2$, with  $\rho_1(x)=1$ for $0\leq x<.5$ this last condition ensures spectral accuracy $(\ref{specdecay})$. The proof of convergence presented in section $(\ref{sec:2})$ works in the same way even with the application of this filter. 
    
 The second filter  (\cite{Krasny}),
  \begin{equation} \label{Filter2}
\rho_2(\widehat{x_m})=\begin{cases}
     0 & \text{if $||\widehat{x_m}||<10^{-13}$}\\
    1 & \text{otherwise},
  \end{cases}
\end{equation}
cuts off the effect of Fourier modes with small amplitudes. This prevents the accumulation of round off error. 

We call the scheme ADBDPR when applying $\rho_1$, and ADBK when applying $\rho_2$. For comparison, the filter is also applied to the CN scheme, referred to as CNK, CNDPR correspondingly. Results of these schemes are presented in $(\ref{fig:Spectrum})$,$(\ref{fig:M3})$ at time $T=.5$ for the ellipse $E_3$. In figure $(\ref{fig:Spectrum})$ the power spectrum $||\widehat{\theta_m}||^2$ is plotted while in figure $(\ref{fig:M3})$ the relative error in $M_3$  is shown. The application of the filter shows a remarkable improvement in stability for ADB. In particular, the error in $M_3$ at $T=.5$ is bounded by $\xi_3=.01$ with $N=512$ points in space and $\Delta t=1\times 10^{-4}$ (compare with table $(\ref{Es})$). The DPR filter only slightly improves the stability of the CN scheme and is not needed over this time period. Filter is not needed for CNADB scheme. In general Krasny filter has no significant effects over the results (ADBK,CNK).
\begin{figure}[H]
    \begin{subfigure}[b]{1\textwidth}
  \hbox{\includegraphics[width=\linewidth]{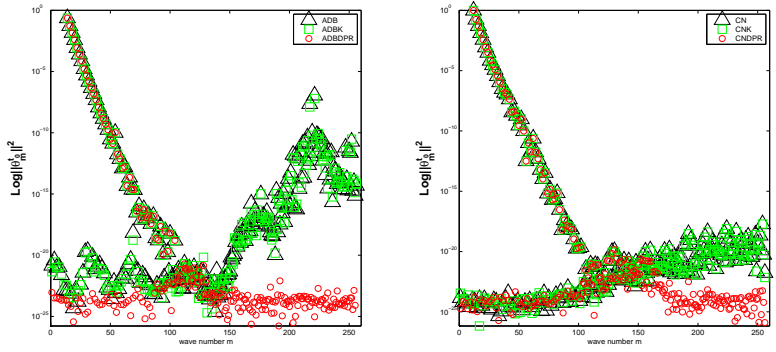}{}}\hfill
  \caption{The effects of filtering on the power spectrum $||\widehat{\theta_m}||^2$ at $t=.5$ for ADB and CN codes}\label{fig:Spectrum}
\end{subfigure}
\begin{center}
    \begin{subfigure}[b]{0.75\textwidth}
  \hbox{\includegraphics[width=\linewidth]{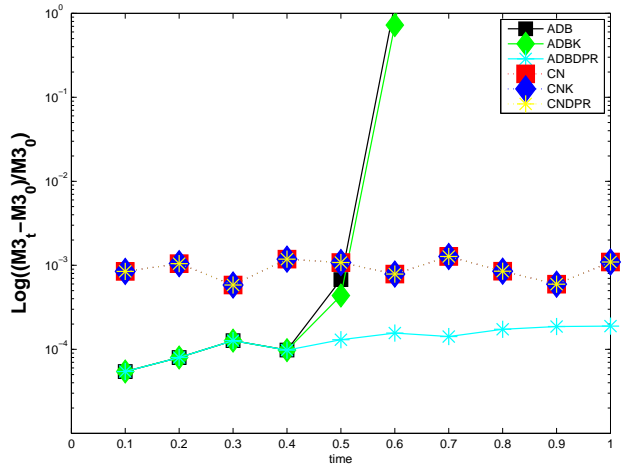}{}}\hfill
  \caption{The effects of filtering on the relative error in $M3$ over time}\label{fig:M3}
\end{subfigure}
\end{center}
  \caption{Relative error and power spectrum computed with $N=512$ and $\Delta t=2.5\times 10^{-5}$ using the initial curve E3 $(\ref{Eeq2})$}
\label{spect}
\end{figure}

\subsection{Linear vs nonlinear dynamics (a stiff perturbed circle) }
Next, we consider a more complex initial curve given by a large perturbation of a circle (PC):
\begin{equation}\label{pcm3}
(x(\alpha,0),y(\alpha,0))=r(\cos\alpha,\sin\alpha),\alpha\in [0,2\pi]\:,r=1+0.4\cos(3\alpha).
\end{equation}

In figure $(\ref{PCdyn})$, we present the morphologies of the evolving curve (solid) and, in figures $(\ref{fig:PCp4xy0})-(\ref{fig:PCp4xy4p5})$, the linear approximation (dashed). The corresponding curvatures are given in figures $(\ref{fig:PCp4k0})-(\ref{fig:PCp4k4p5})$. The linear and nonlinear results differ significantly because of the large initial perturbation. Further, nonlinear simulation shows the development of dispersive waves that travel along the interface. This is clearly evident in the curvature. 

 \begin{figure}[H]
    \begin{subfigure}[b]{.22\textwidth}
  \hbox{\includegraphics[width=\linewidth]{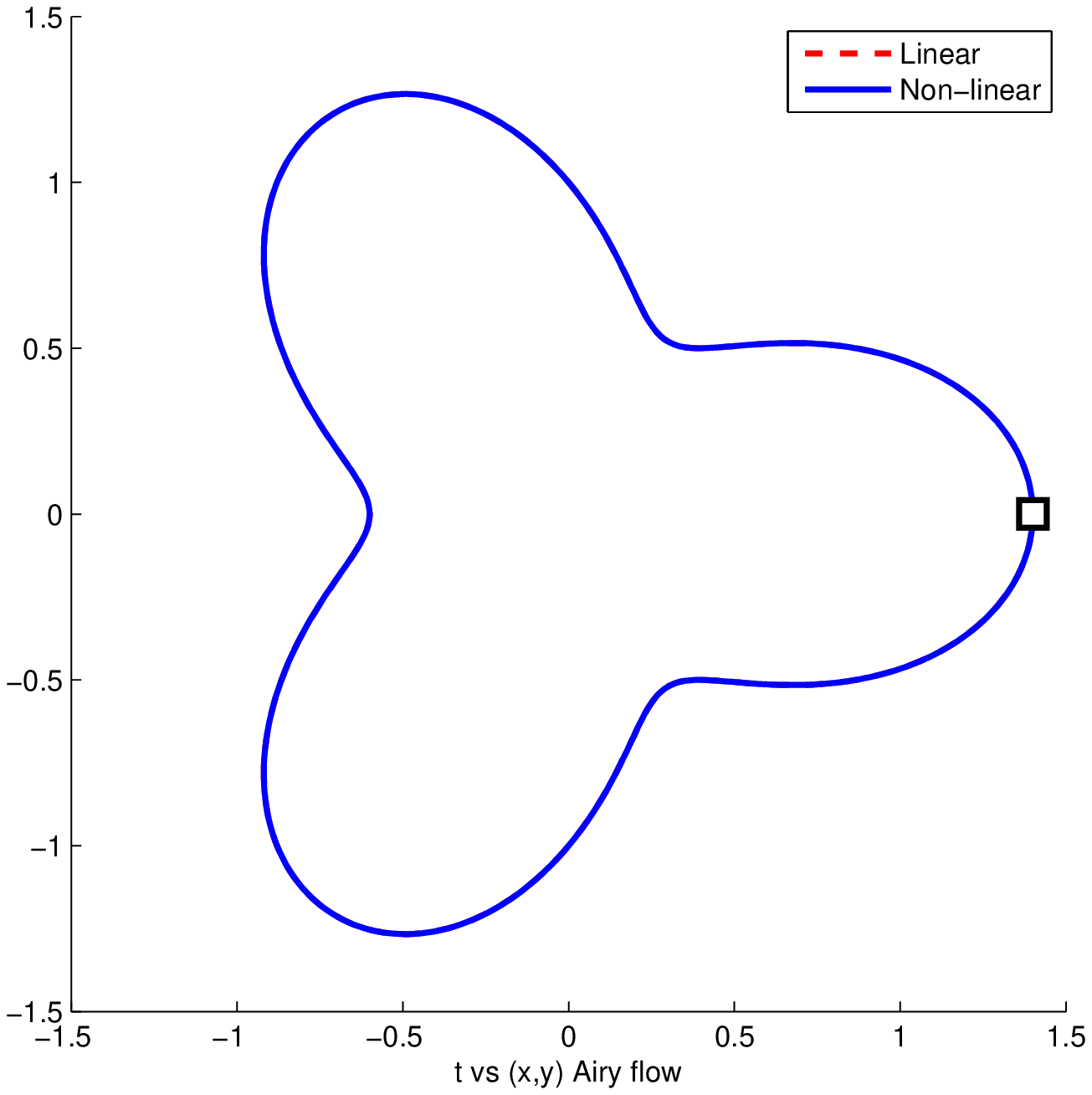}{}}\hfill
      \caption{T=0}\label{fig:PCp4xy0}
 \end{subfigure}
    \begin{subfigure}[b]{.22\textwidth}
  \hbox{\includegraphics[width=\linewidth]{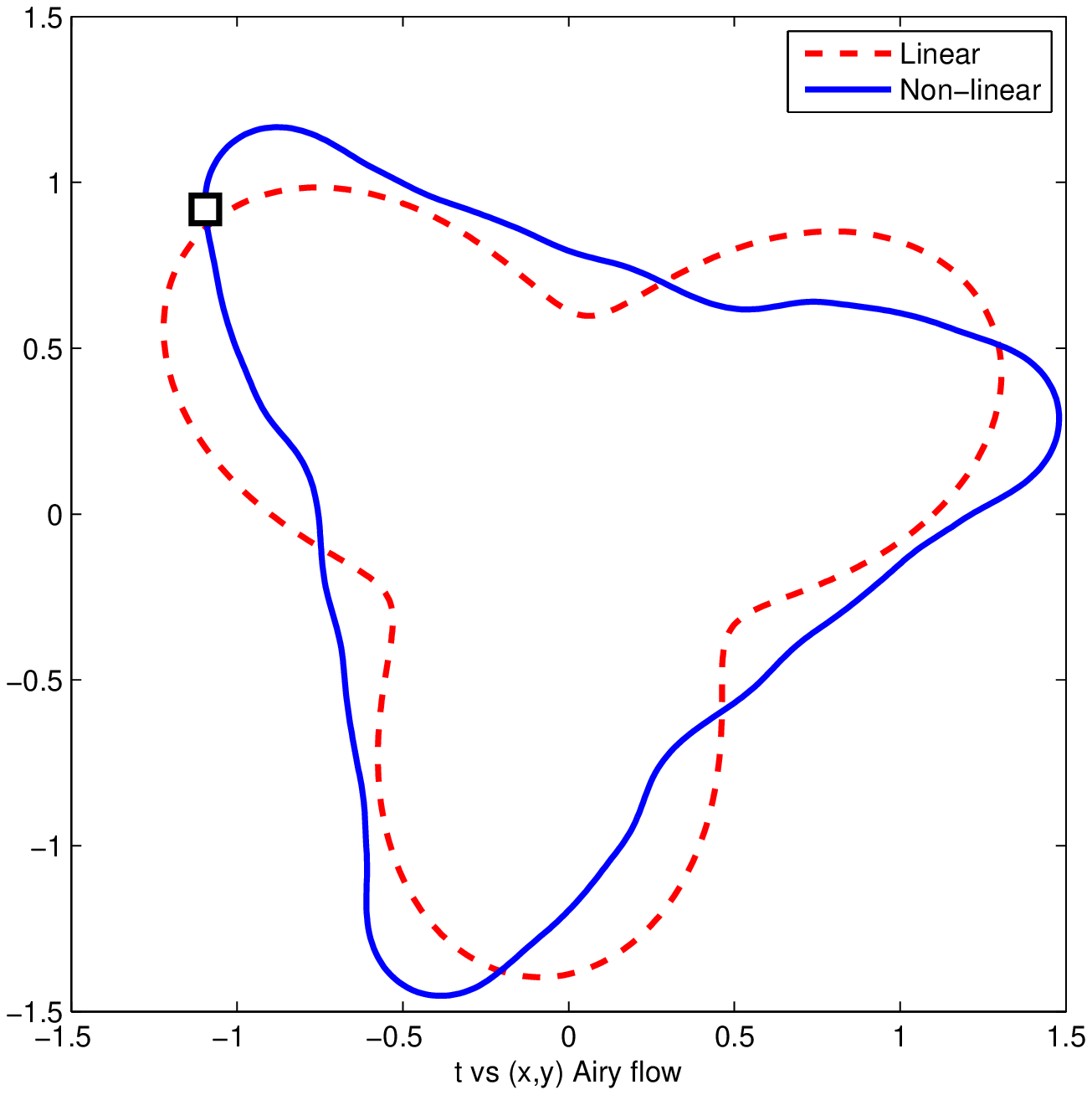}{}}\hfill
      \caption{T=.5}\label{fig:PCp4xyp5}
 \end{subfigure}
    \begin{subfigure}[b]{.22\textwidth}
  \hbox{\includegraphics[width=\linewidth]{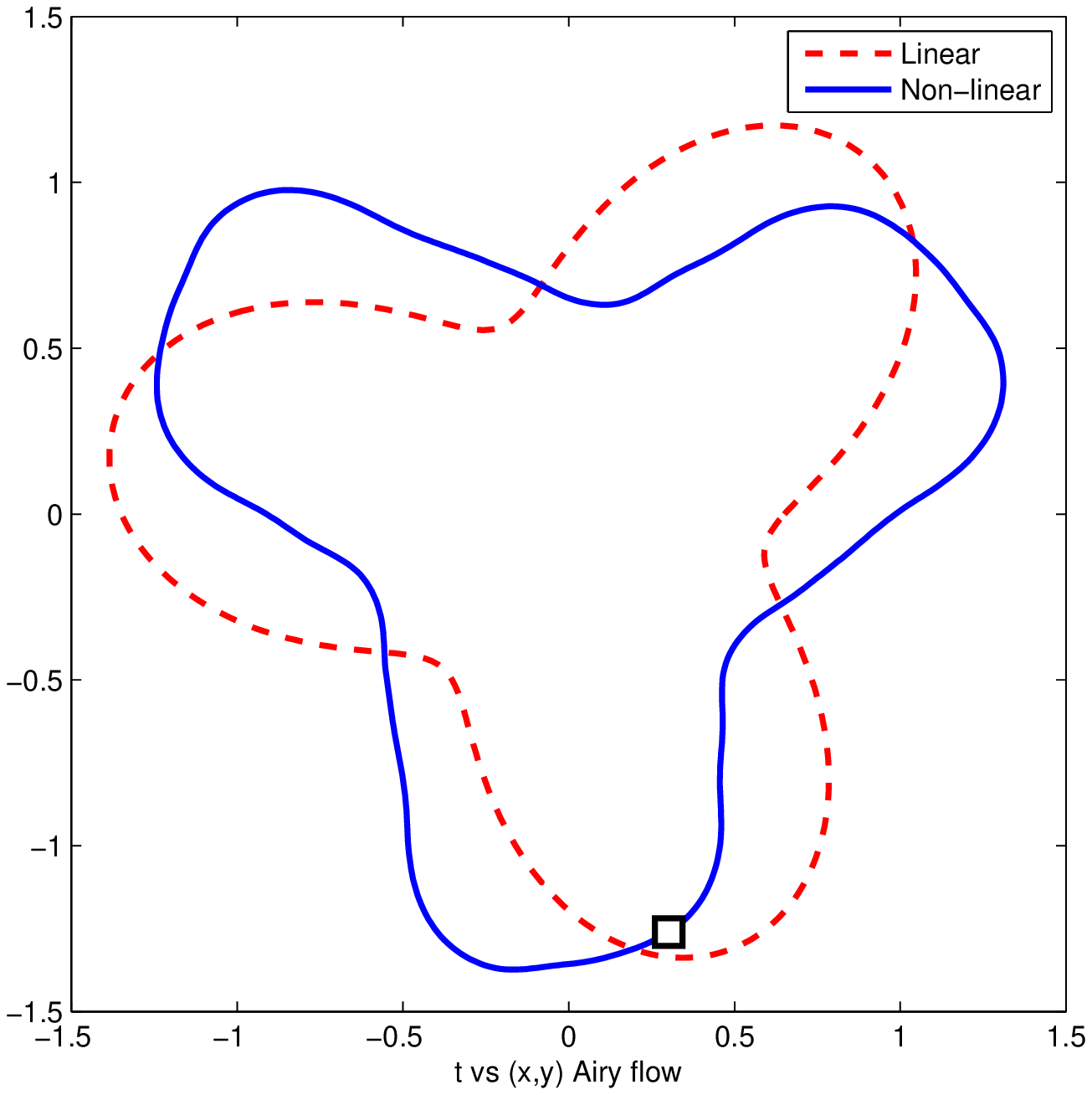}{}}\hfill
      \caption{T=1} \label{fig:Pcp4xy1}
 \end{subfigure}
    \begin{subfigure}[b]{.22\textwidth}
  \hbox{\includegraphics[width=\linewidth]{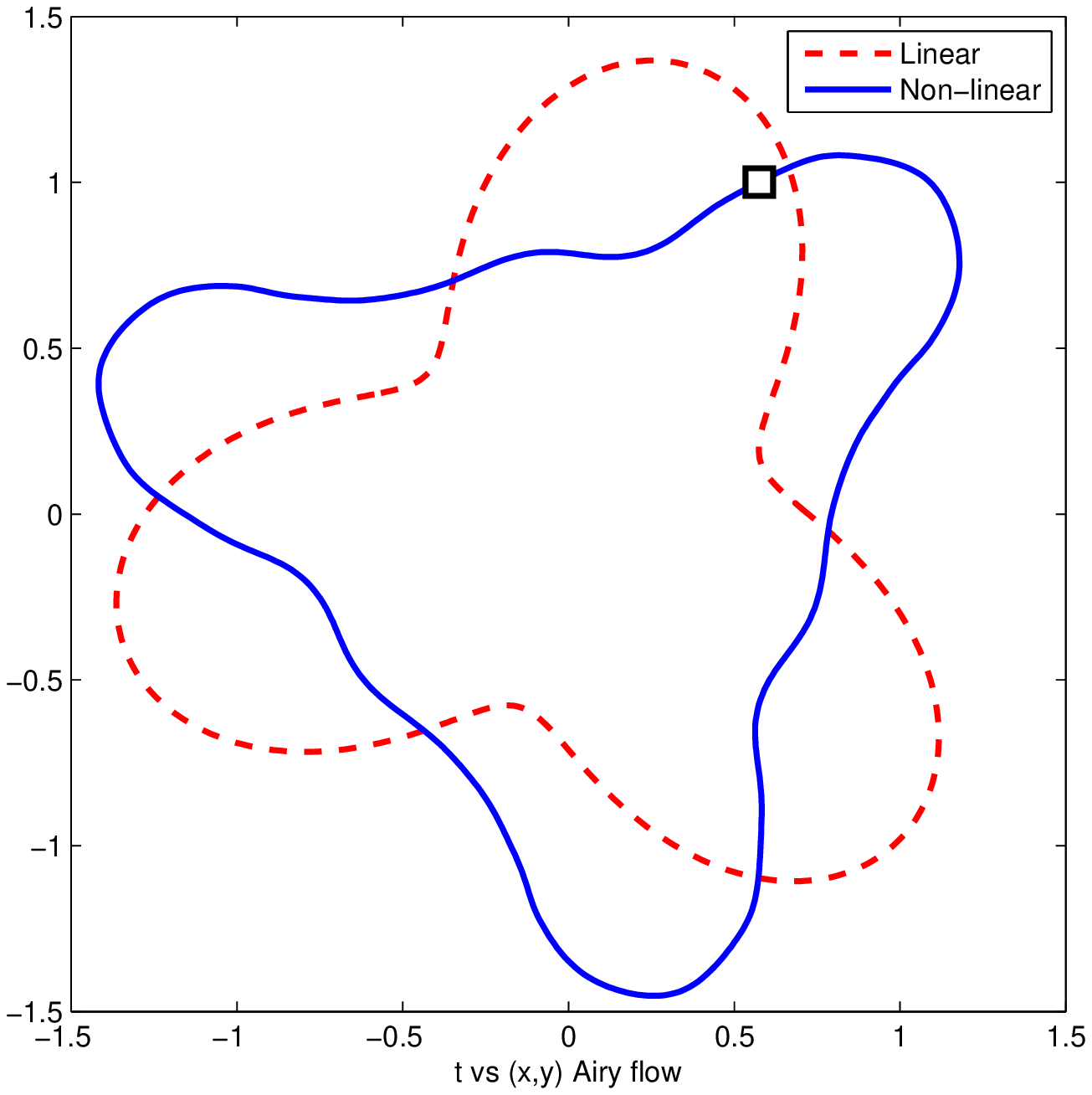}{}}\hfill
      \caption{T=1.5}\label{fig:PCp4xy1p5}
 \end{subfigure}
    \begin{subfigure}[b]{.22\textwidth}
  \hbox{\includegraphics[width=\linewidth]{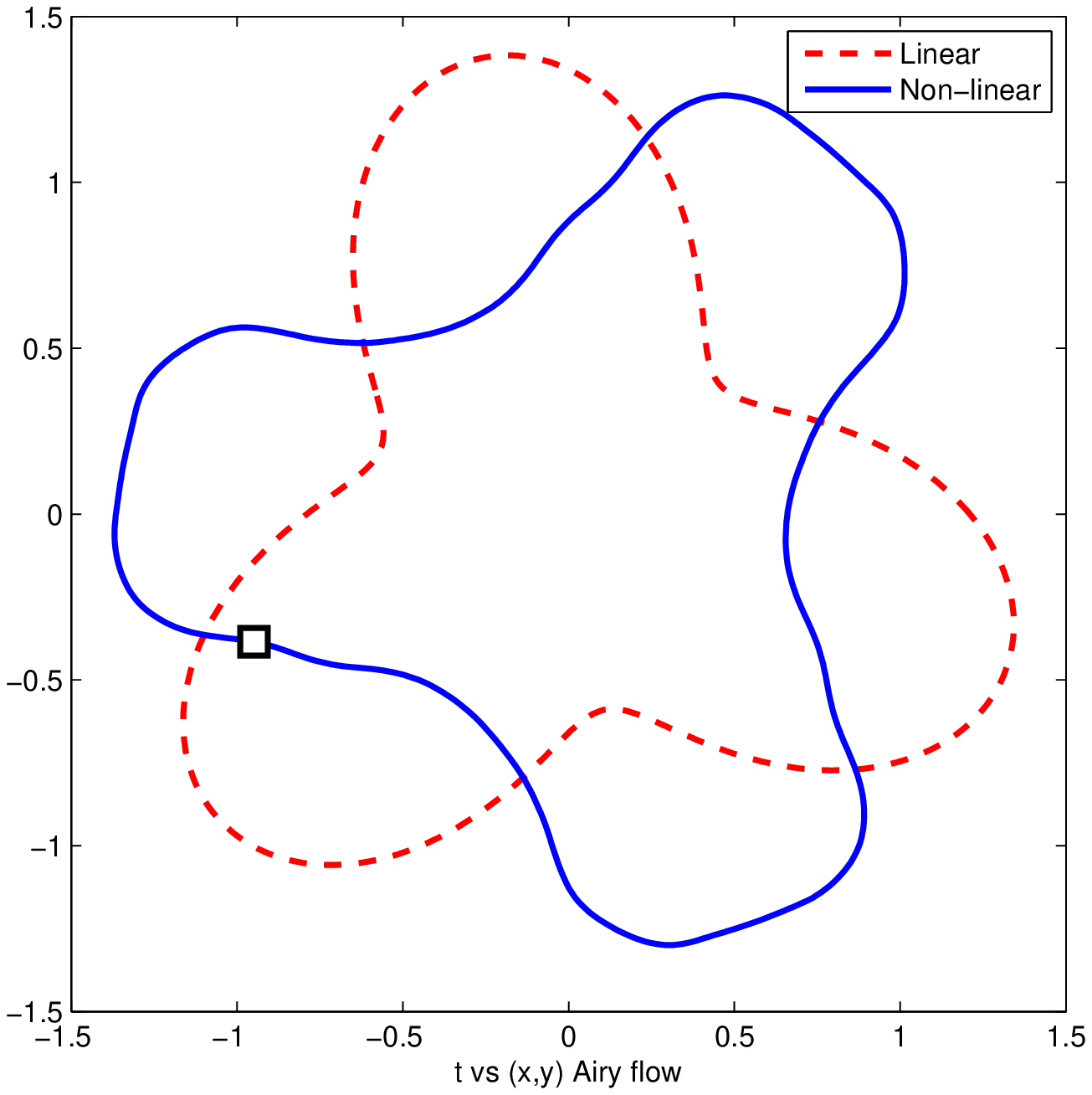}{}}\hfill
      \caption{T=2}\label{fig:PCp4xy2}
 \end{subfigure}
     \begin{subfigure}[b]{.22\textwidth}
  \hbox{\includegraphics[width=\linewidth]{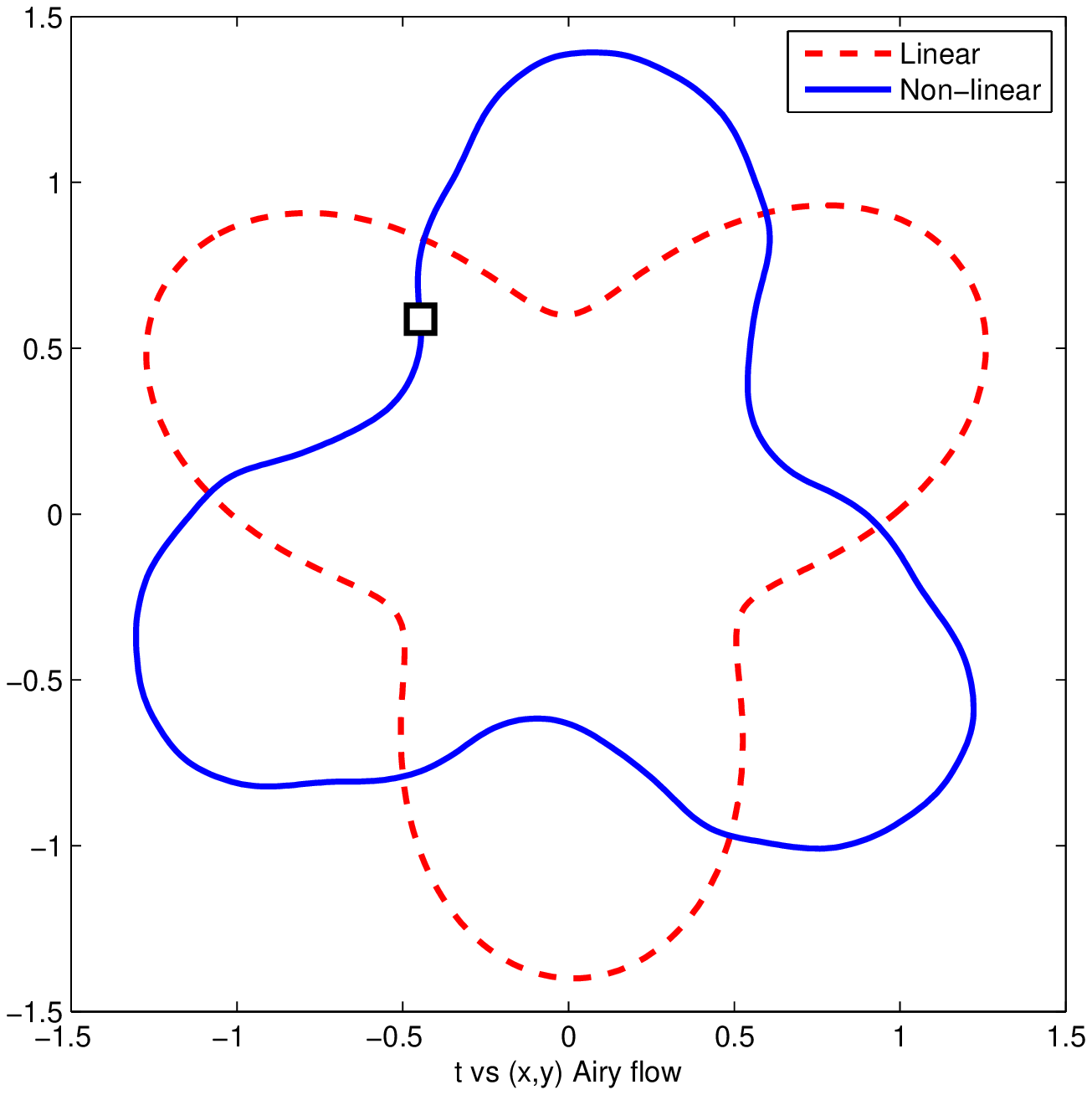}{}}\hfill
      \caption{T=3} \label{fig:PCp4xy3}
 \end{subfigure}
    \begin{subfigure}[b]{.22\textwidth}
  \hbox{\includegraphics[width=\linewidth]{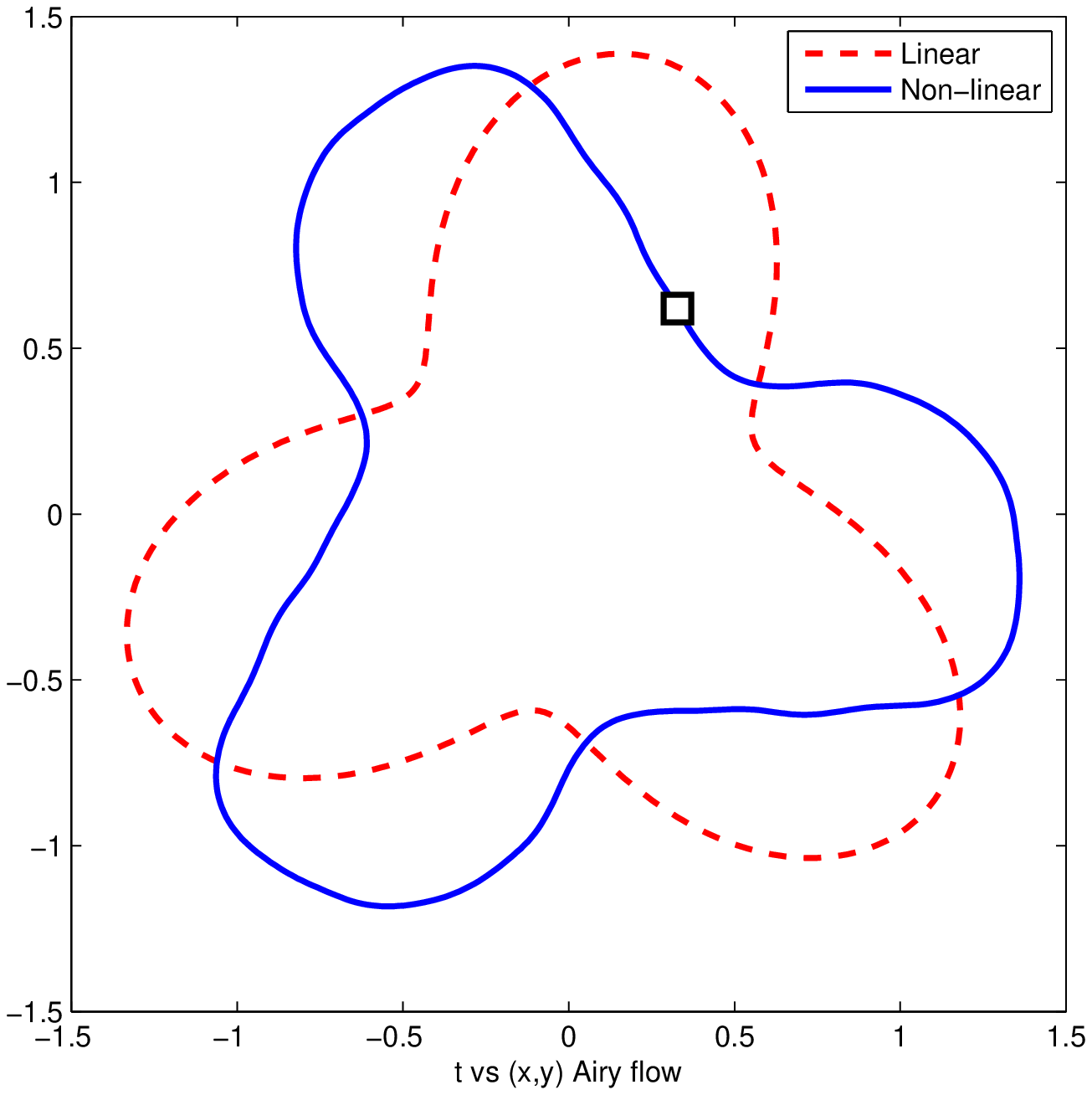}{}}\hfill
      \caption{T=4}\label{fig:PCp4xy4}
 \end{subfigure}
    \begin{subfigure}[b]{.22\textwidth}
  \hbox{\includegraphics[width=\linewidth]{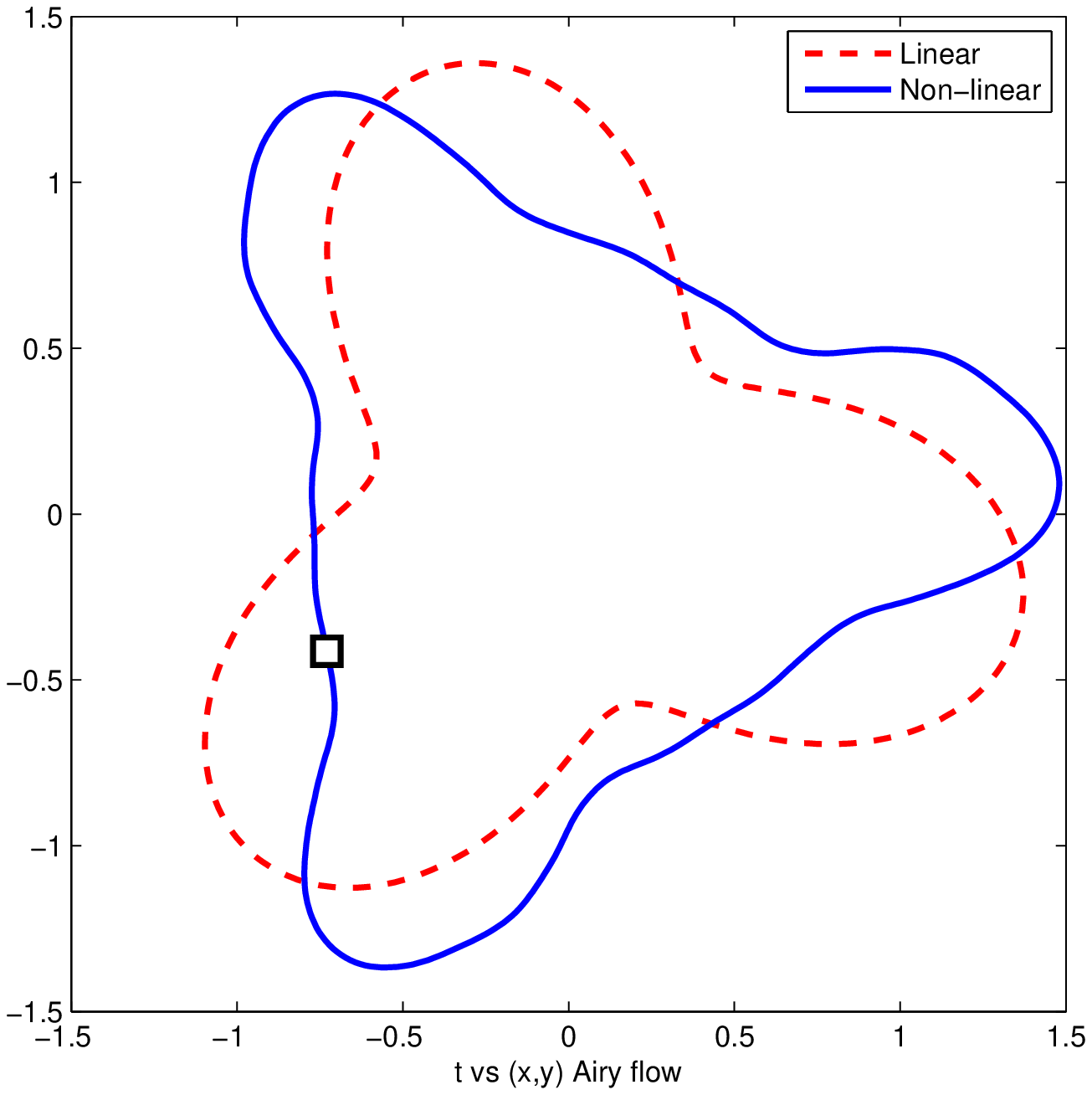}{}}\hfill
      \caption{T=4.5}\label{fig:PCp4xy4p5}
 \end{subfigure}
    \begin{subfigure}[b]{.22\textwidth}
  \hbox{\includegraphics[width=\linewidth]{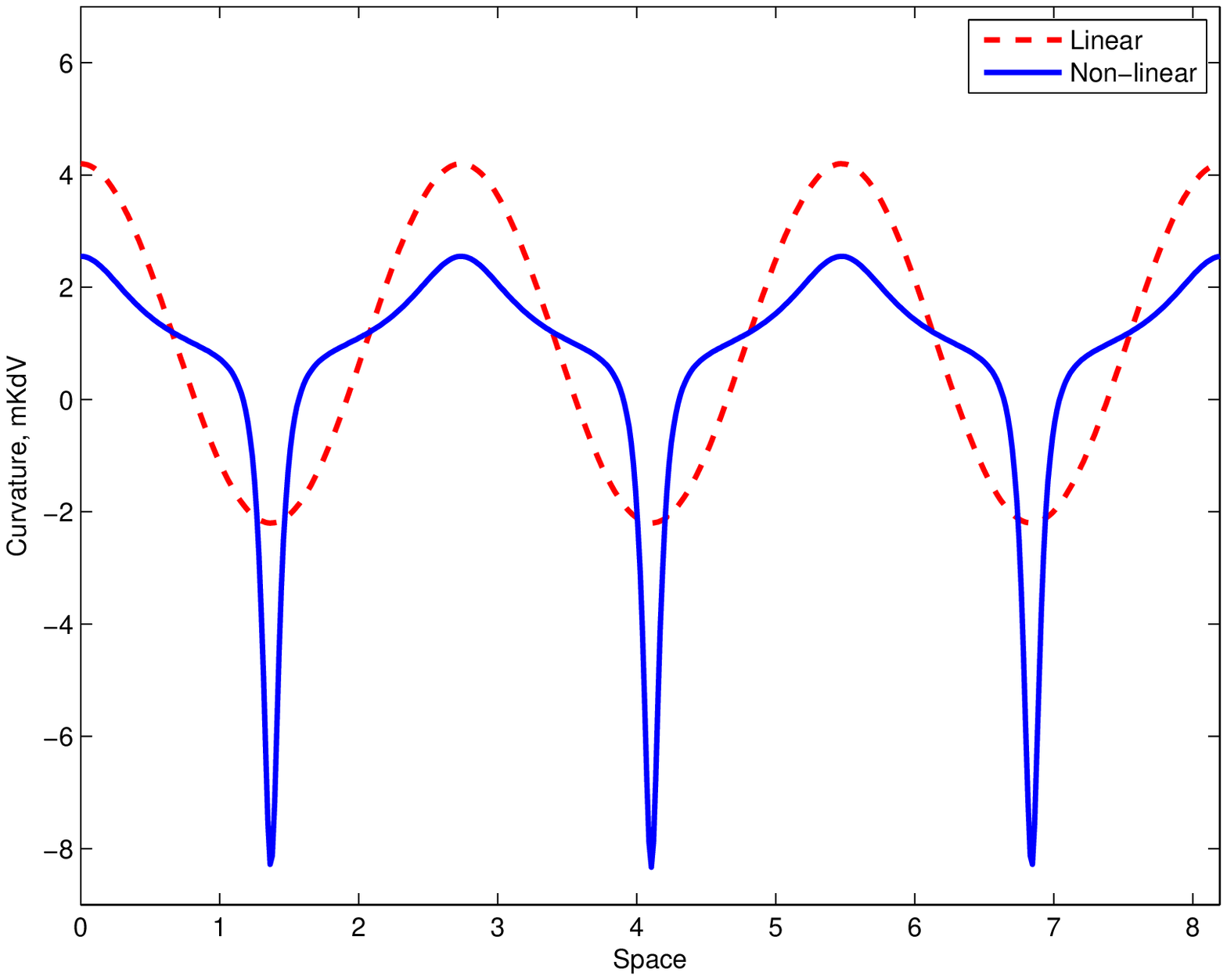}{}}\hfill
      \caption{T=0}\label{fig:PCp4k0}
 \end{subfigure}
    \begin{subfigure}[b]{.22\textwidth}
  \hbox{\includegraphics[width=\linewidth]{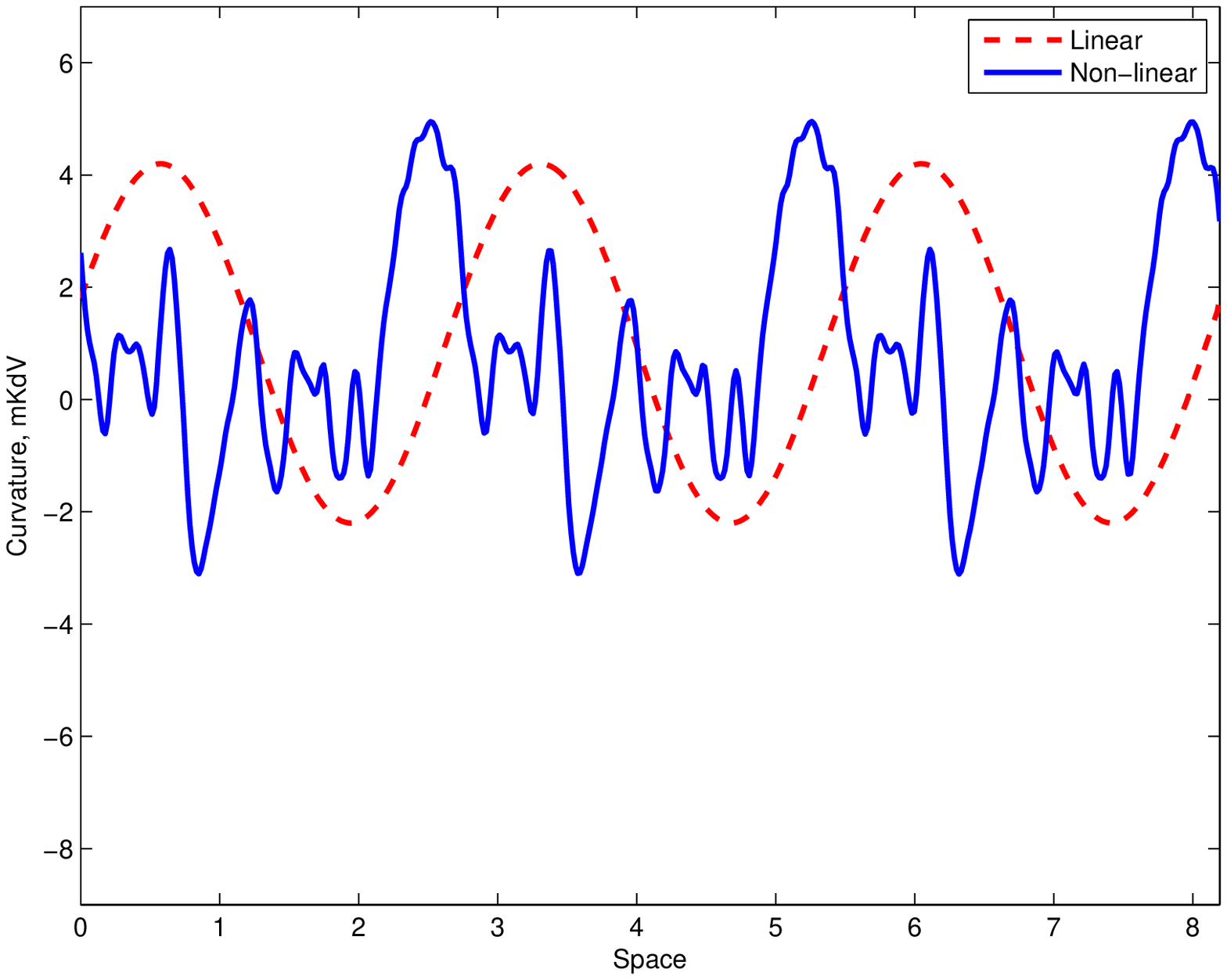}{}}\hfill
      \caption{T=.5}\label{fig:PCp4kp5}
 \end{subfigure}
    \begin{subfigure}[b]{.22\textwidth}
  \hbox{\includegraphics[width=\linewidth]{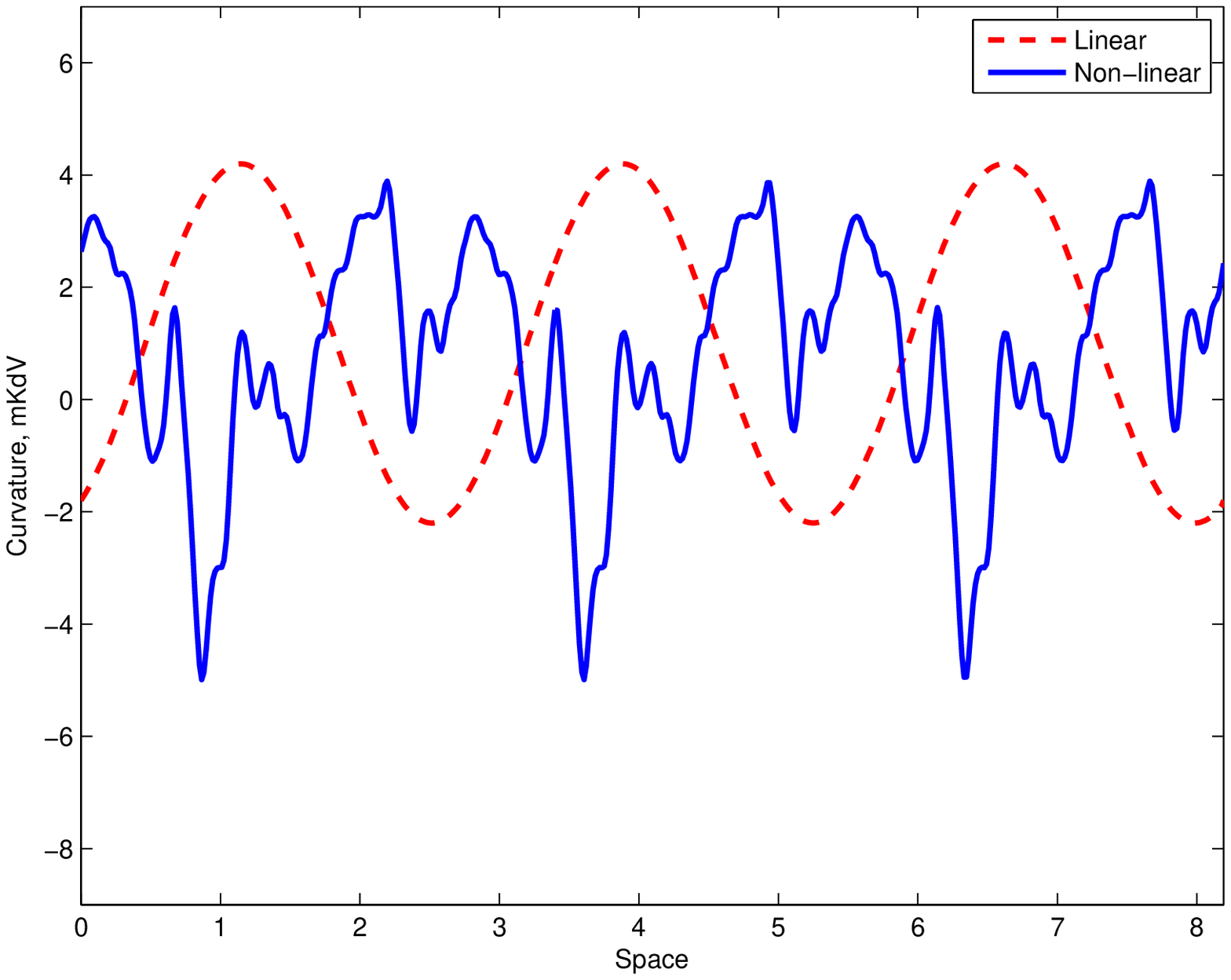}{}}\hfill
      \caption{T=1} \label{fig:PCp4k1}
 \end{subfigure}
    \begin{subfigure}[b]{.22\textwidth}
  \hbox{\includegraphics[width=\linewidth]{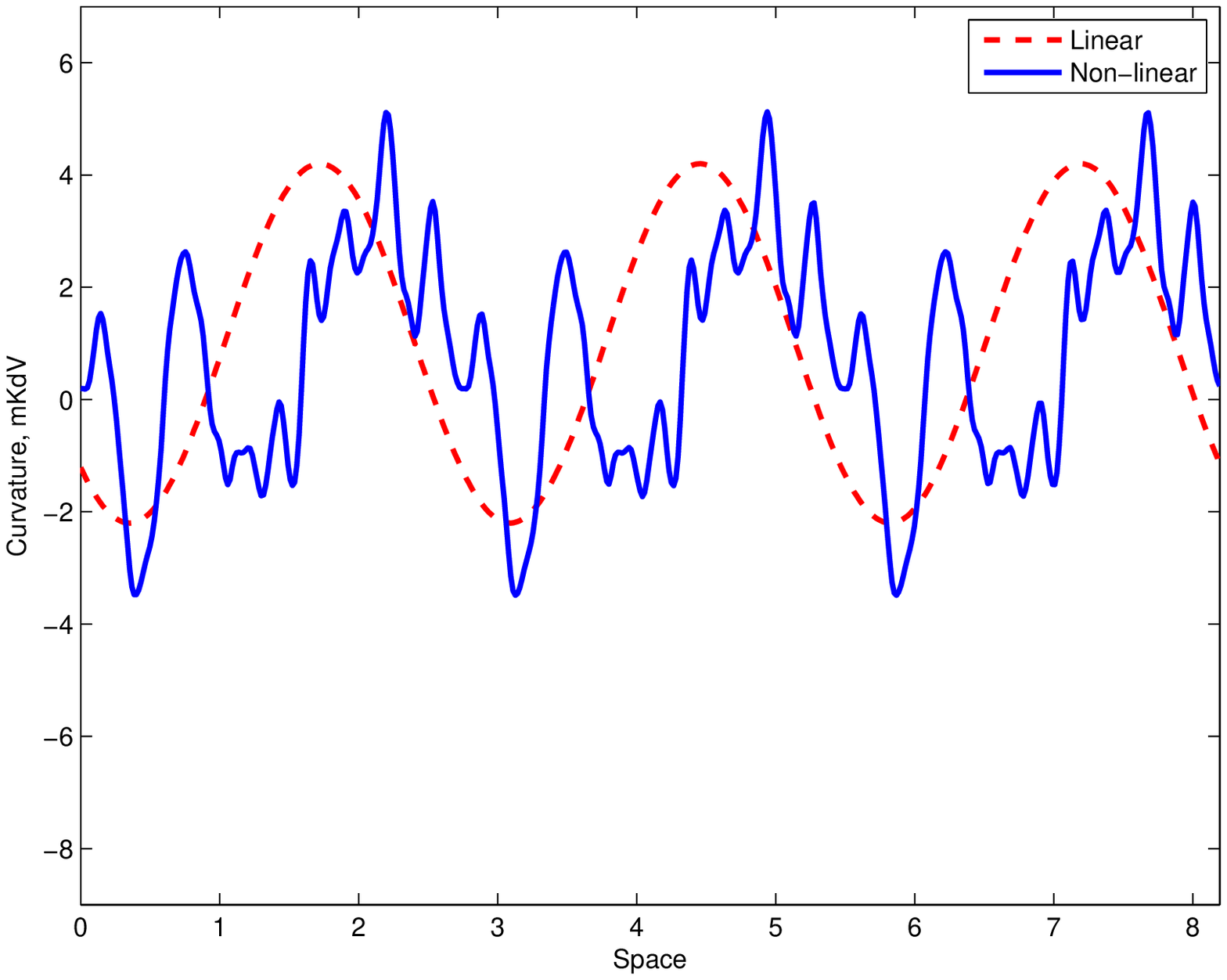}{}}\hfill
      \caption{T=1.5}\label{fig:PCp4k1p5}
 \end{subfigure}
    \begin{subfigure}[b]{.22\textwidth}
  \hbox{\includegraphics[width=\linewidth]{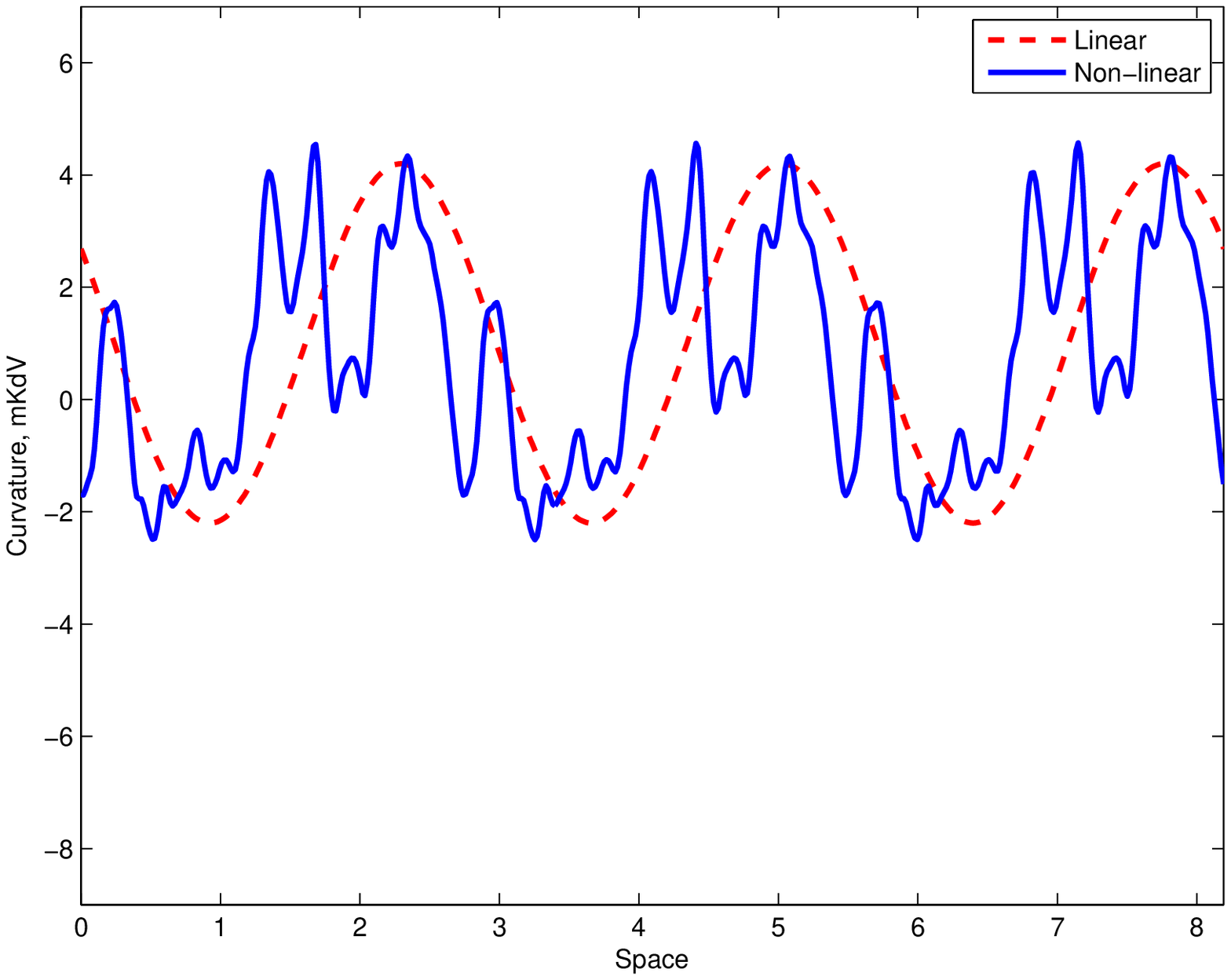}{}}\hfill
      \caption{T=2}\label{fig:PCp4k2}
 \end{subfigure}
    \begin{subfigure}[b]{.22\textwidth}
  \hbox{\includegraphics[width=\linewidth]{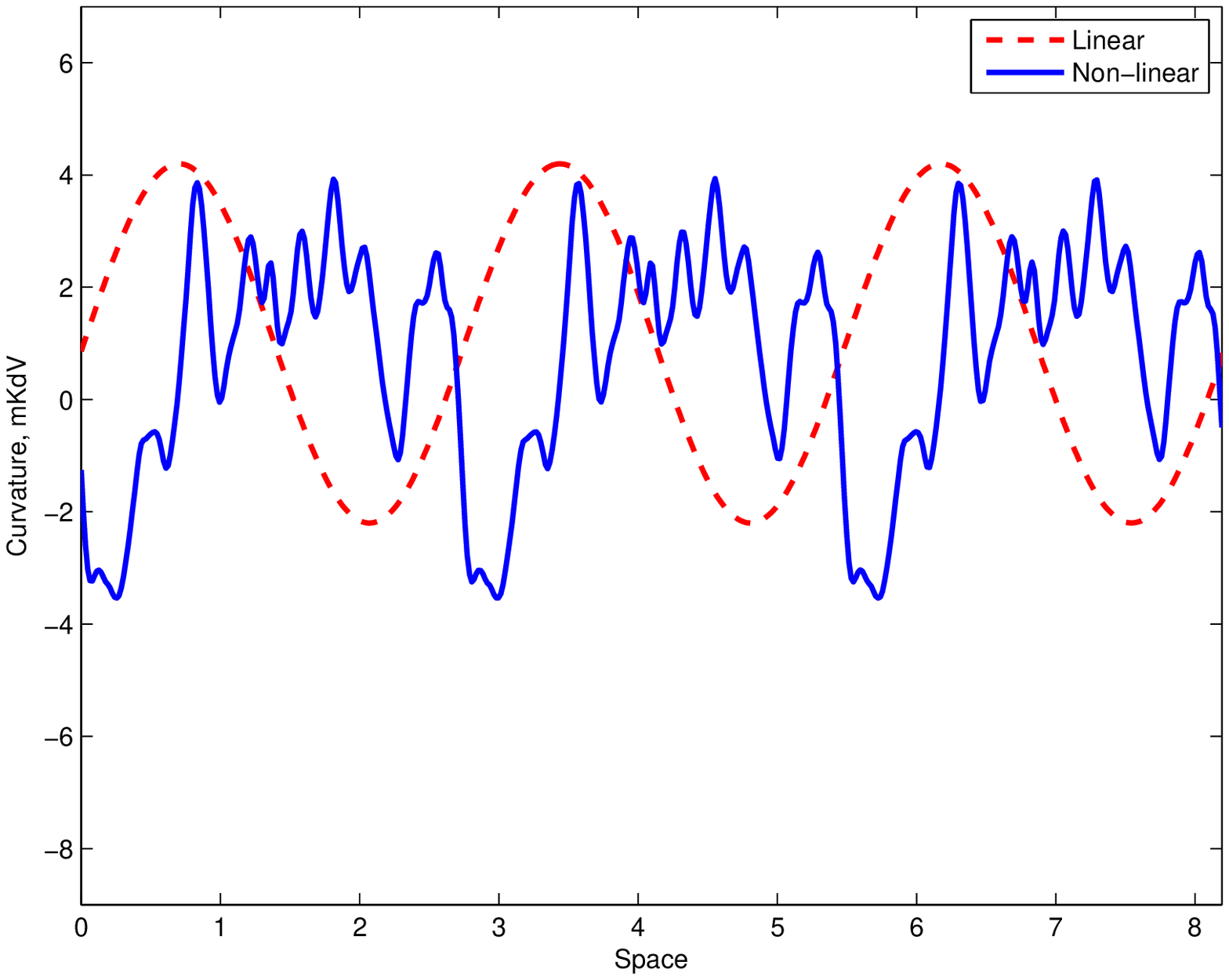}{}}\hfill
      \caption{T=3} \label{fig:PCp4k3}
 \end{subfigure}
    \begin{subfigure}[b]{.22\textwidth}
  \hbox{\includegraphics[width=\linewidth]{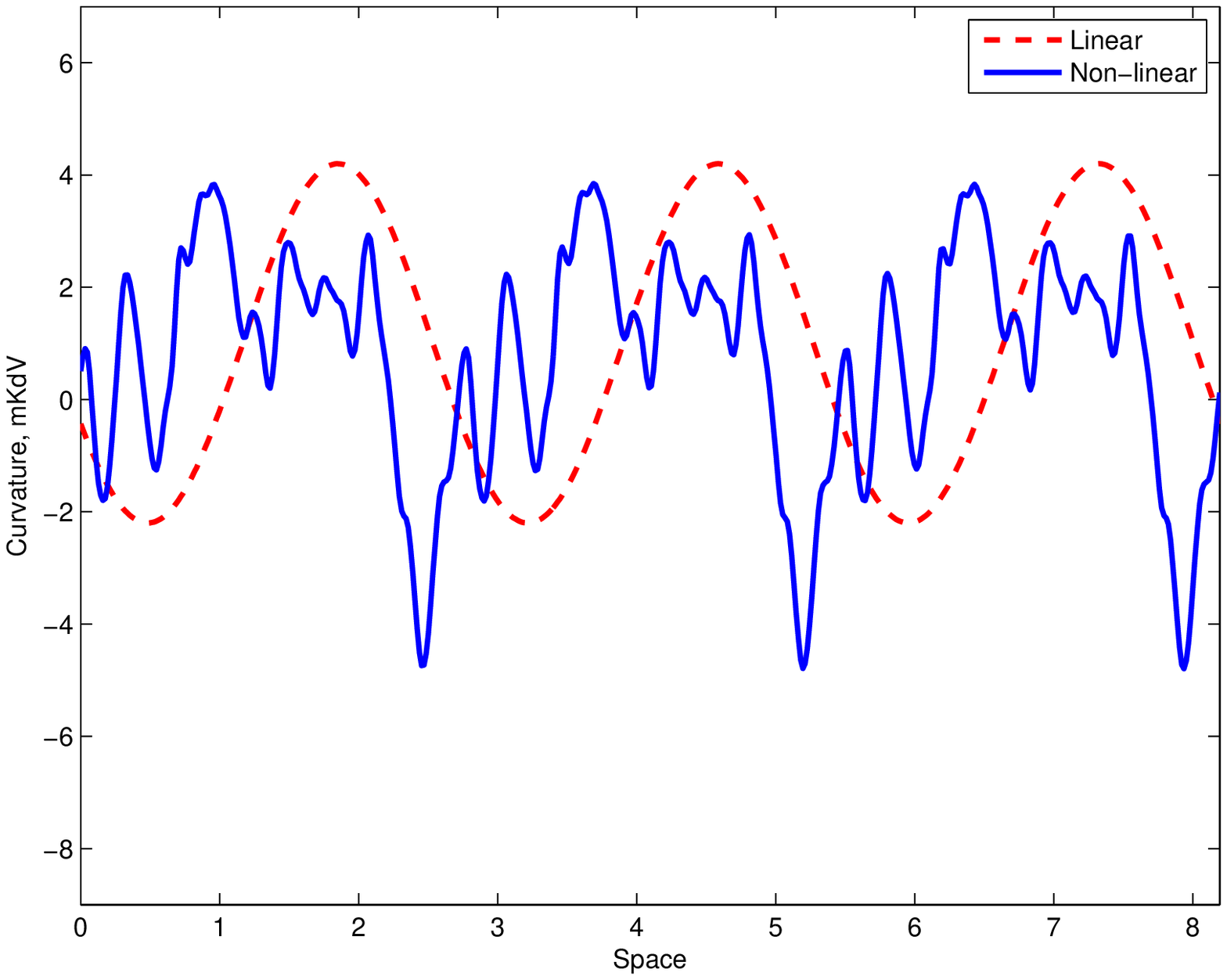}{}}\hfill
      \caption{T=4}\label{fig:PCp4k4}
 \end{subfigure}
    \begin{subfigure}[b]{.22\textwidth}
  \hbox{\includegraphics[width=\linewidth]{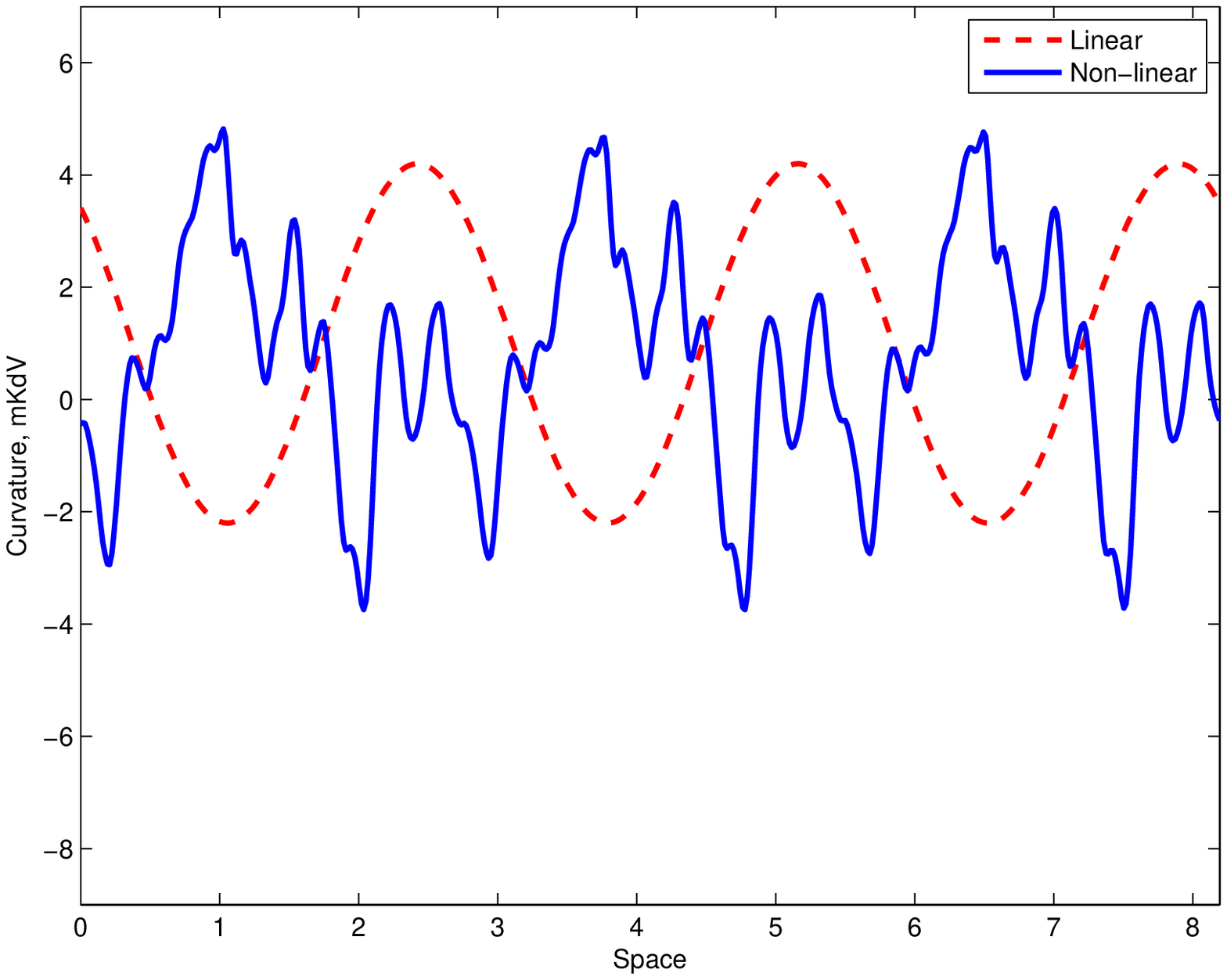}{}}\hfill
      \caption{T=4.5}\label{fig:PCp4k4p5}
 \end{subfigure}
  \caption{Comparison of the linear (red-dashed) and nonlinear (blue solid) interface morphologies $(\ref{fig:PCp4xy0})-(\ref{fig:PCp4xy4p5})$ and curvature $(\ref{fig:PCp4k0})-(\ref{fig:PCp4k4p5})$ for the initial PC curve $(\ref{pcm3})$. The numerical results are generated using the CNADB scheme with $N=512$ points in space and $\Delta t=5\times 10^{-6}$}
\label{PCdyn}
\end{figure} 

It is necessary to take $\Delta t=5\times 10^{-6}$, for the error $Max|\frac{M_{3 (t)}-M_{3(0)}}{M_{3(0)}}|$ to be bounded below $10^{-2}$ with $N=512$ (see figure $\ref{fig:PC}$ (upper)). During this period $(T=4.5)$, a particle at the interface covers more than 3 loops around its center of mass. In tables $(\ref{tconv})$ and $(\ref{sconv})$, we present the convergence analysis for the perturbed circle (PC) using CN, CNADB and ADBDPR schemes that confirm second order accuracy in time and spectral accuracy in space.
From figure $(\ref{PCdyn})$, we can observe that although the nonlinear shape does not rotate without changing shape (as the linear solution does), the overall shape of the nonlinear evolving curve still retains a large contribution from the initial perturbation. To quantify this effect, we calculate the numerical perturbation as before: 

\begin{equation}\label{delshape}
 \widetilde{\delta_N}\approx max_{\alpha}(\sqrt{x^2+y^2}-R_0),
 \end{equation}
where
\begin{equation}
R_0=\sqrt{\frac{Initial \:area}{\pi}}.
\end{equation}

The results are shown in figure $(\ref{fig:PC})$ (bottom). Observe that the perturbation stably oscillates about $\delta_0=.4$ and that the evolution is nearly, but not exactly periodic.
  
\begin{figure}[H]
    \begin{subfigure}[b]{.42\textwidth}
  \hbox{\includegraphics[width=\linewidth]{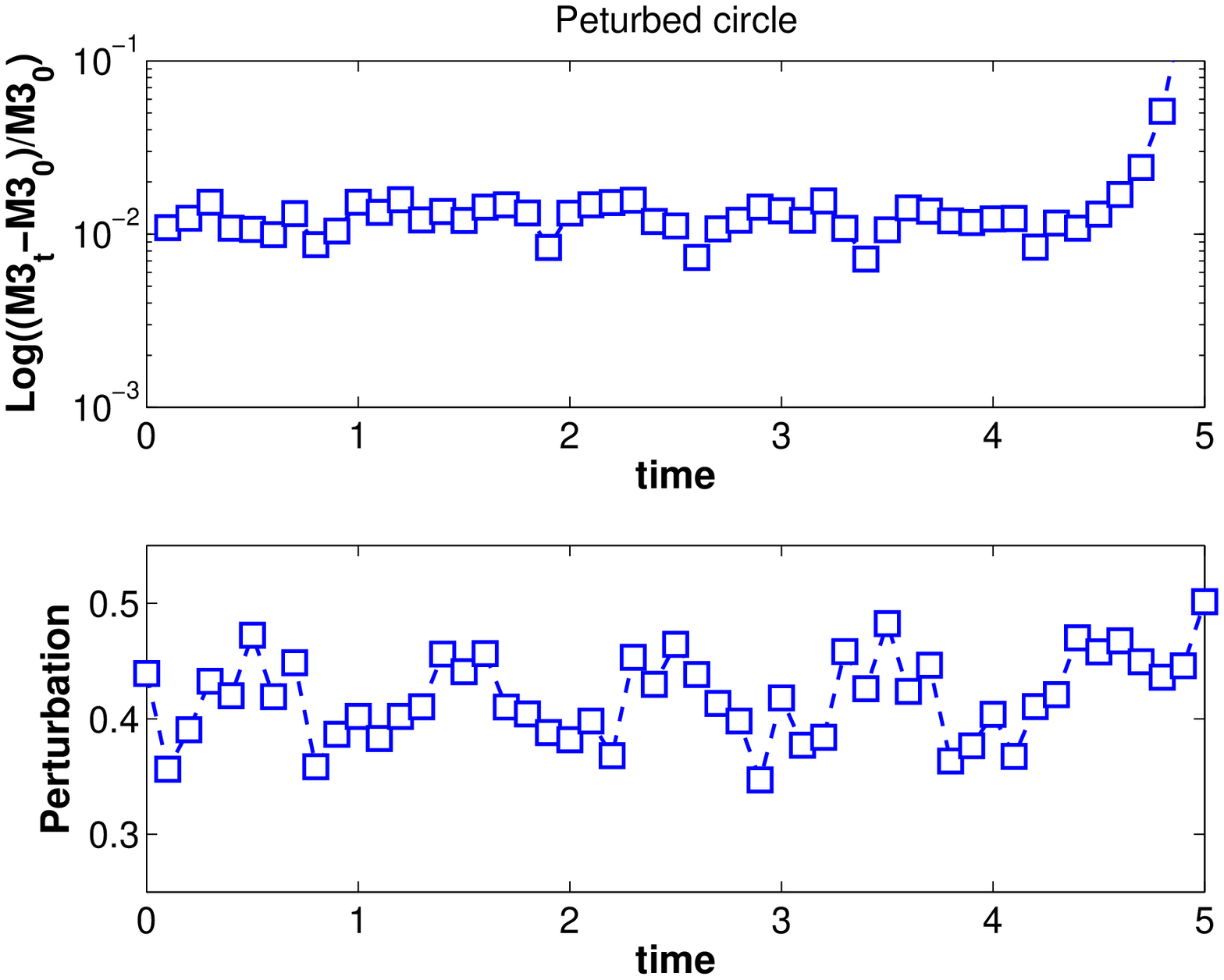}{}}\hfill
  \caption{Perturbed circle, $\Delta t=5\times 10^{-6}$}\label{fig:PC}      
    \end{subfigure}
    \begin{subfigure}[b]{.42\textwidth}
  \hbox{\includegraphics[width=\linewidth]{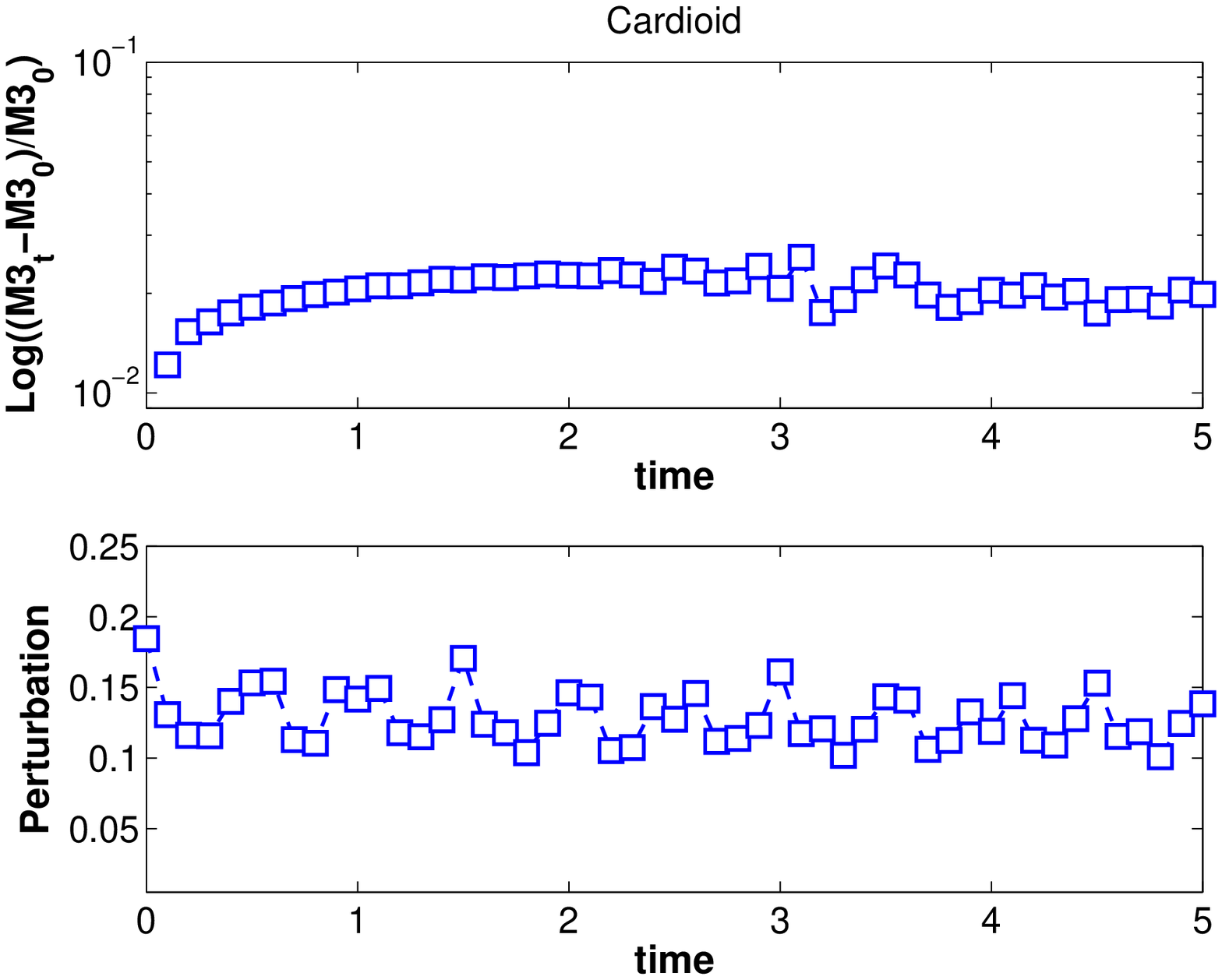}{}}\hfill
  \caption{Cardioid, $\Delta t=1\times 10^{-5}$}\label{fig:Car}
    \end{subfigure}
  \caption{Evolution of the relative error in $M_3$ (top) and nonlinear perturbation magnitude $\widetilde{\delta_N}$ (bottom) over time for the PC $(a)$ and the cardioid $(b)$ initial curves. These results were generated using CNADB with $N=512$ and $\Delta t=5\times 10^{-6}$ for (a) and $\Delta t=1\times 10^{-5}$ for (b)}
\label{fig:PCCar}
\end{figure}

\subsubsection{Dynamics for a cardioid curve.}\label{cardioidcurve} 
As a final example, we use a cardioid as an initial curve:
\begin{equation}\label{carcurve}
(x(\alpha,0),y(\alpha,0))=(\cos (\alpha)+.35 \sin (2 \alpha),\sin(\alpha)+.7\sin(\alpha)^2),\alpha \in[0,2\pi].
\end{equation}

This choice is motivated by the study of \cite{Hsiao} who considered the same curve and its dynamics.
The curve evolution is shown in figure $(\ref{fig:CCNADBxy0})$-$(\ref{fig:CCNADBxy5})$, and the corresponding curvature in figure $(\ref{fig:CCNADBxy0})$-$(\ref{fig:CCNADBxy5})$. The high curvature of the bottom of the cardioid rapidly generates dispersive waves that reduce the curvature and travel around the interface. On the time period $T=5$, a particle at the interface moves about $3/4$ of a loop around its center of mass. The evolution of the curve and curvature are smooth and as such is in distinct contrast with that presented in (\cite{Hsiao}), where the evolution was found to be irregular and the morphologies of the evolving curve were not smooth. To test the accuracy of our results, we present evidence for convergence of the method ($2^{nd}$ order accurate in time and spectral accuracy in space) in tables $(\ref{tconv}),(\ref{sconv})$. We also calculate the relative error in figure $(\ref{fig:Car})$ (upper) and find that it evolves stably and is less than $2\times 10^{-2}$ over the whole time interval ($0\leq T\leq 5$). The evolution of the nonlinear perturbation $(\ref{fig:Car})$ (bottom) is also stable, and looks to be nearly periodic. Thus, we conclude that our simulation, unlike that presented in (\cite{Hsiao}), is accurate.
\begin{figure}[H]
    \begin{subfigure}[b]{.22\textwidth}
  \hbox{\includegraphics[width=\linewidth]{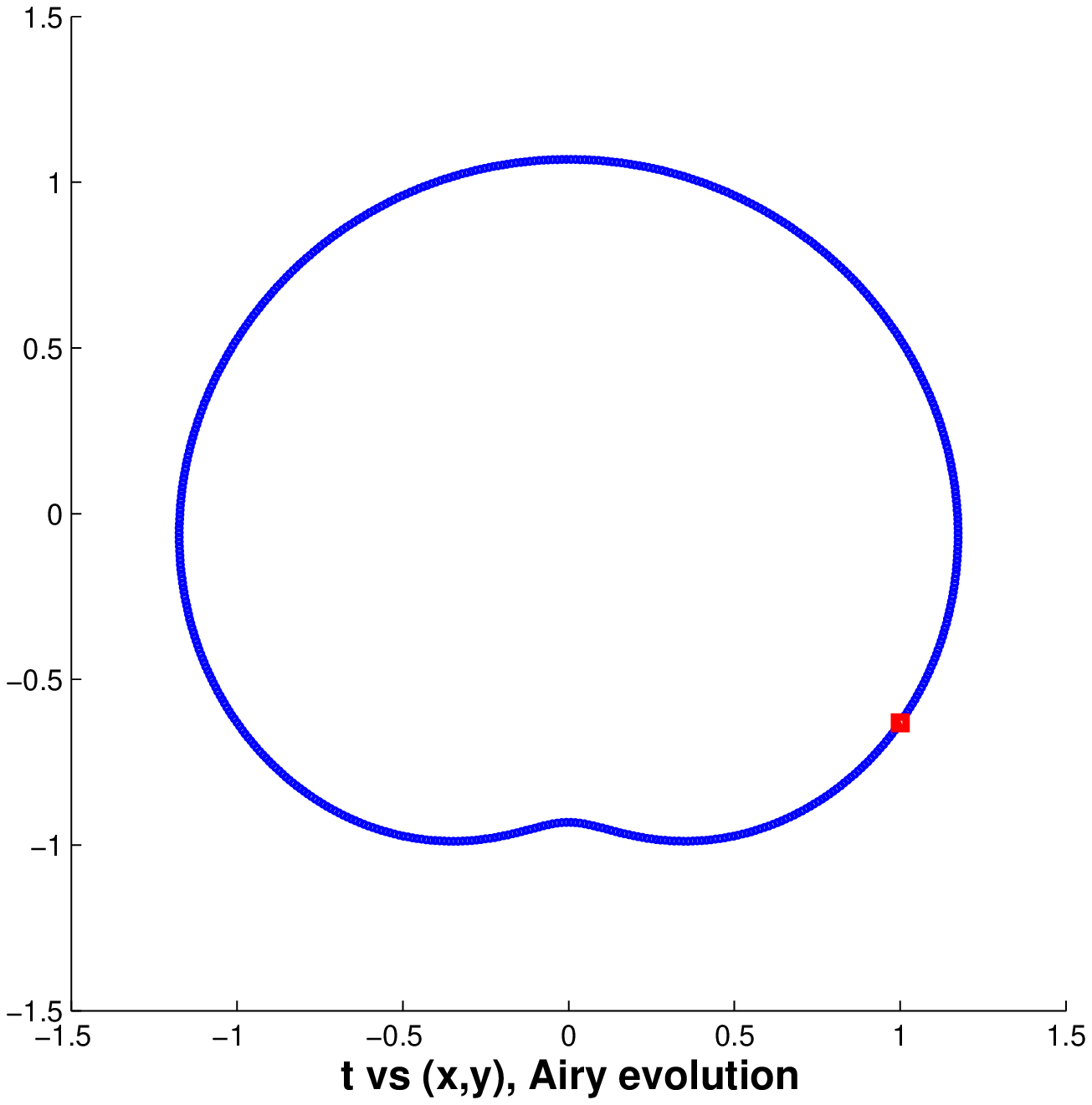}{}}\hfill
      \caption{T=0}\label{fig:CCNADBxy0}
    \end{subfigure}
    \begin{subfigure}[b]{.22\textwidth}
  \hbox{\includegraphics[width=\linewidth]{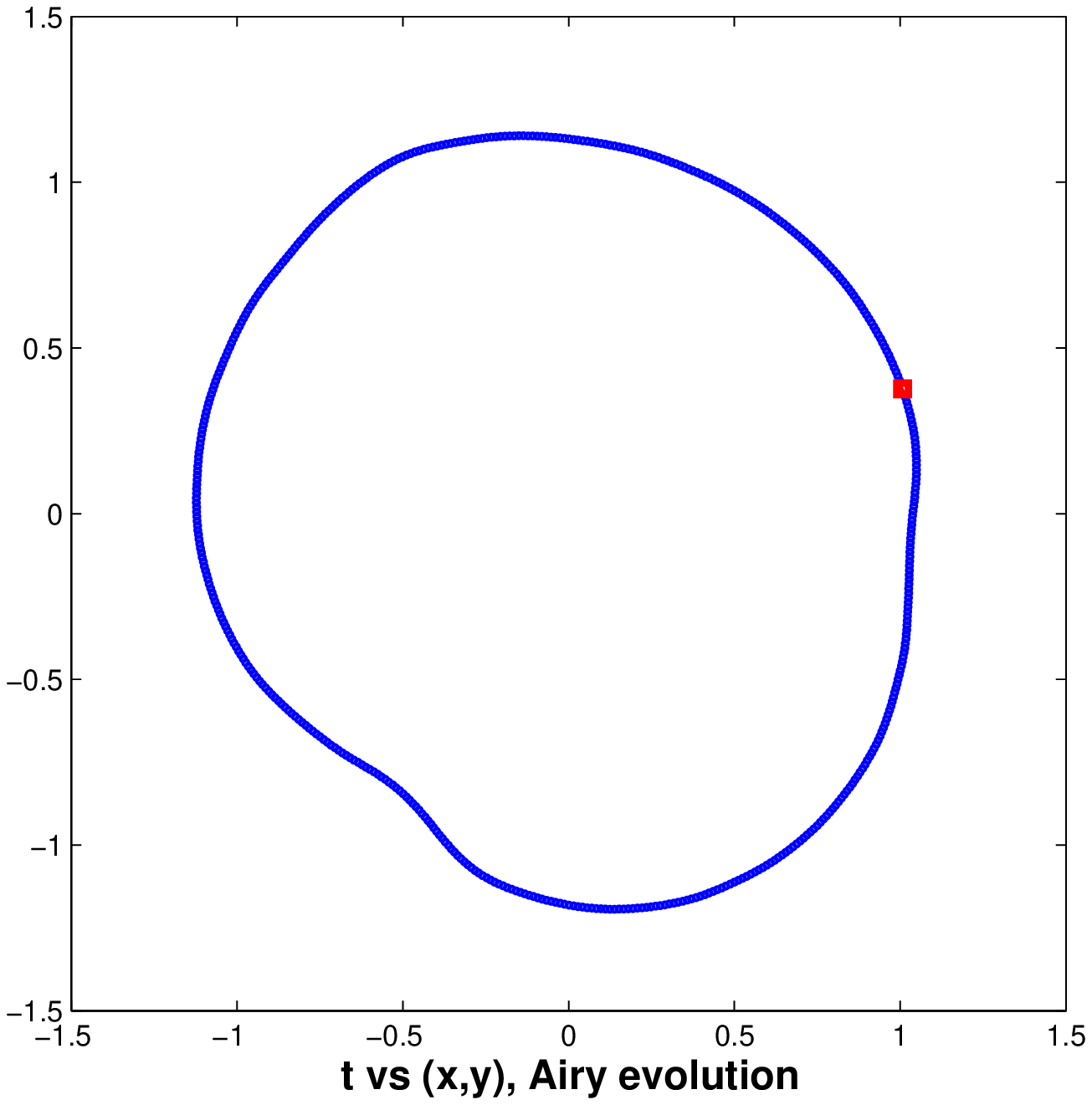}{}}\hfill
      \caption{T=1}\label{fig:CCNADBxy1}
    \end{subfigure}
    \begin{subfigure}[b]{.22\textwidth}
  \hbox{\includegraphics[width=\linewidth]{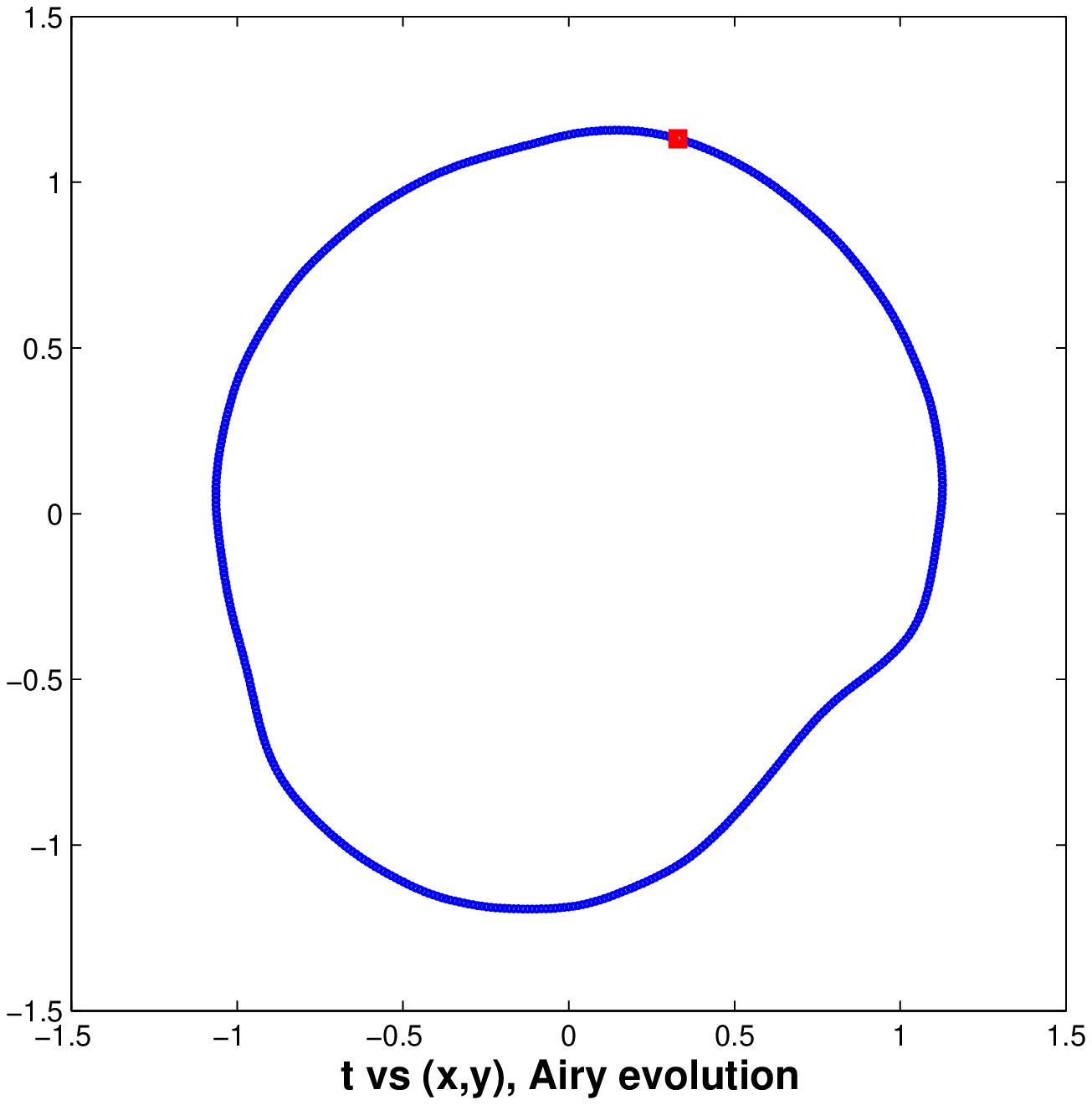}{}}\hfill
      \caption{T=2} \label{fig:CCNADBxy2}
    \end{subfigure}
    \begin{subfigure}[b]{.22\textwidth}
  \hbox{\includegraphics[width=\linewidth]{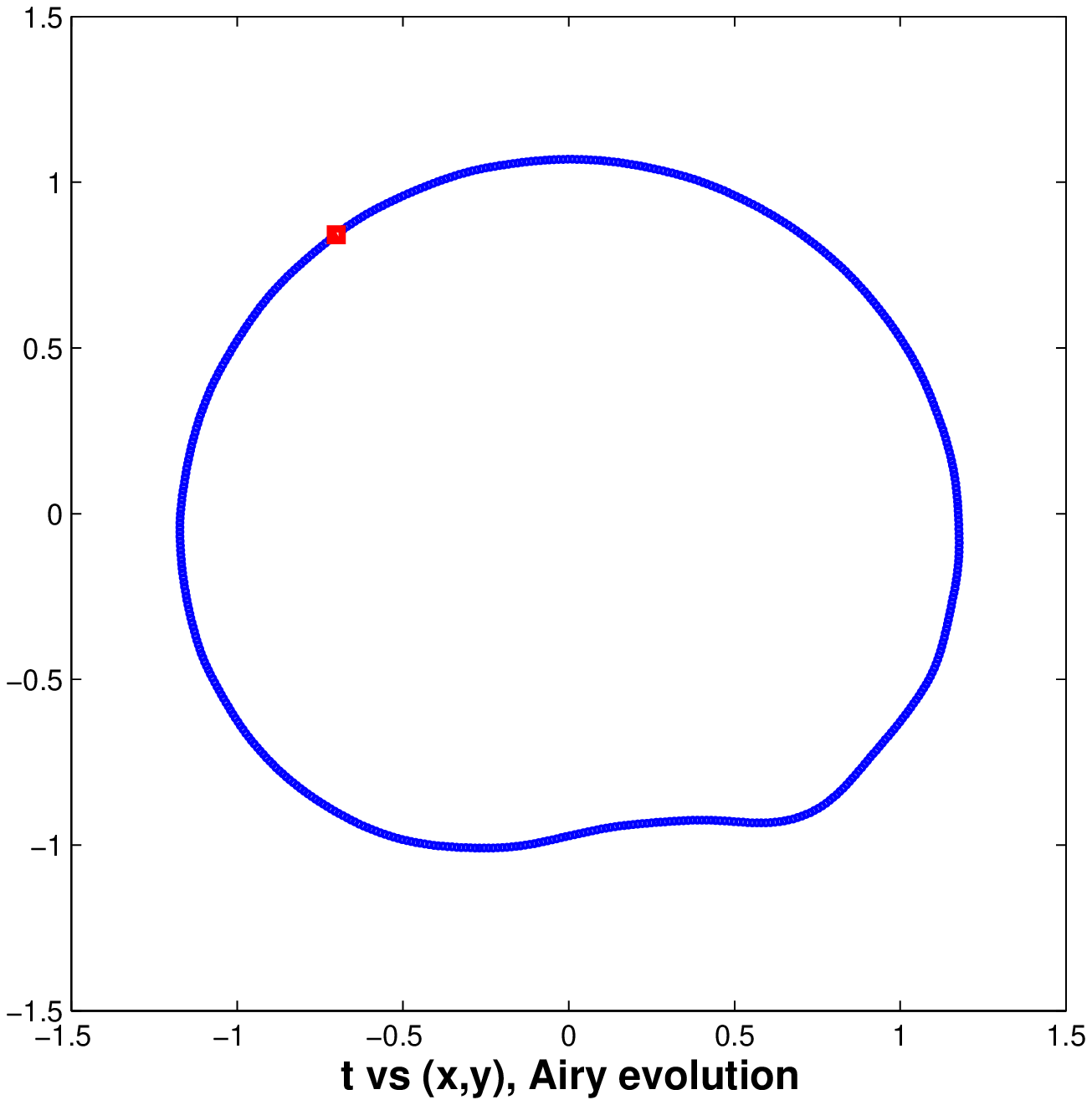}{}}\hfill
      \caption{T=3}\label{fig:CCNADBxy3}
    \end{subfigure}
    
    \begin{subfigure}[b]{.22\textwidth}
  \hbox{\includegraphics[width=\linewidth]{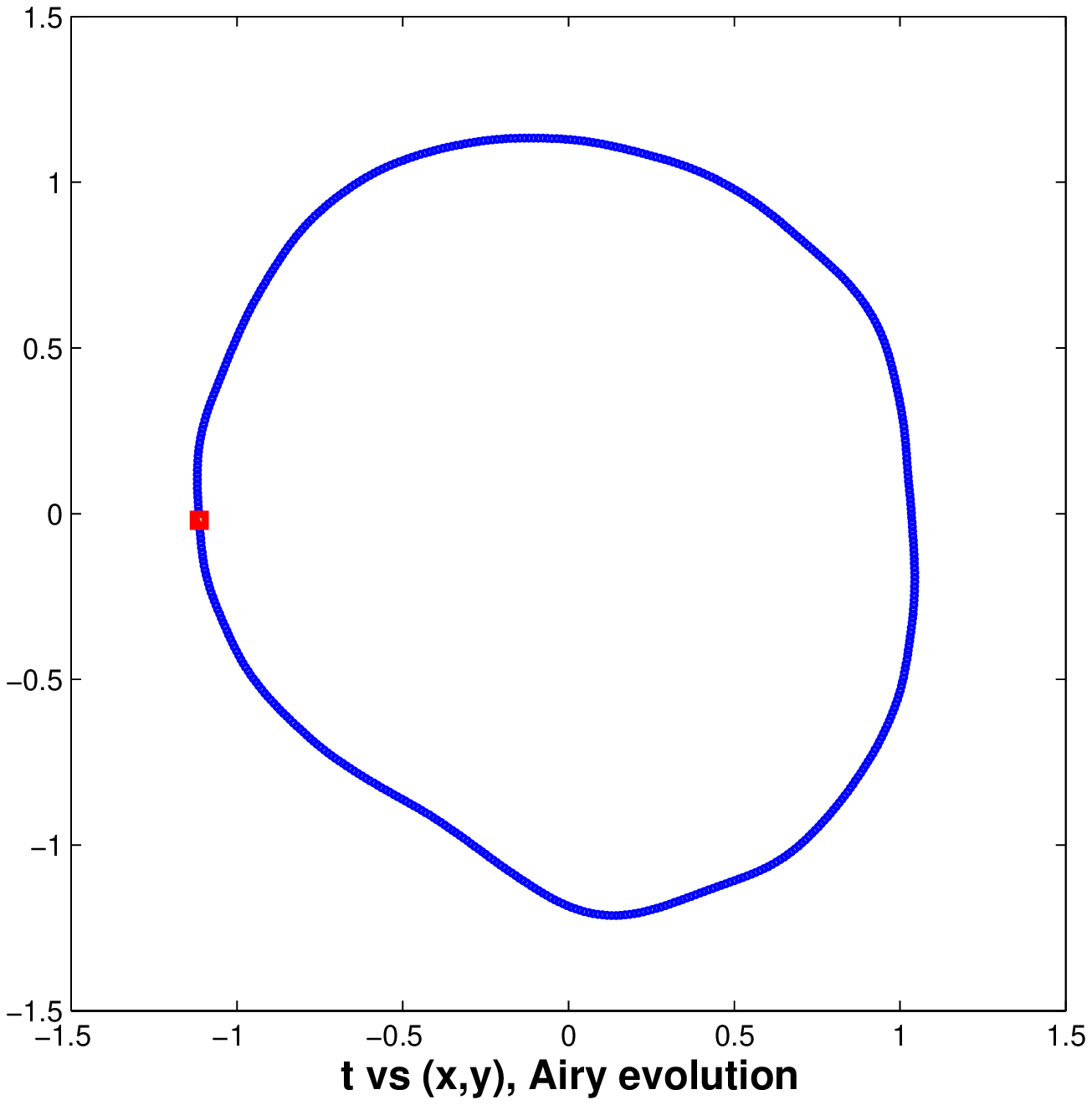}{}}\hfill
      \caption{T=4}\label{fig:CCNADBxy4}
    \end{subfigure}
    \begin{subfigure}[b]{.22\textwidth}
  \hbox{\includegraphics[width=\linewidth]{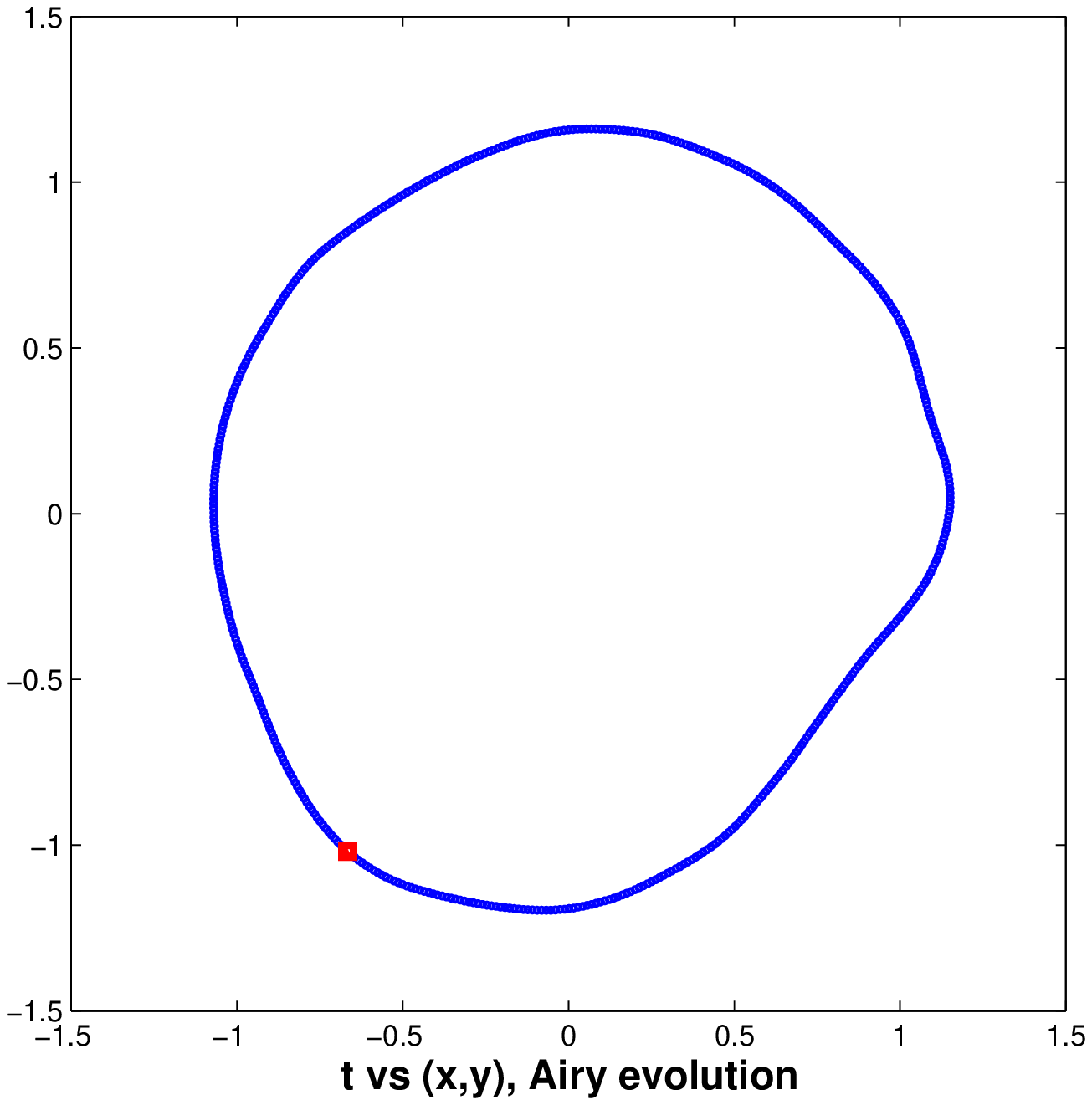}{}}\hfill
      \caption{T=5} \label{fig:CCNADBxy5}      
    \end{subfigure}
    \begin{subfigure}[b]{.22\textwidth}
  \hbox{\includegraphics[width=\linewidth]{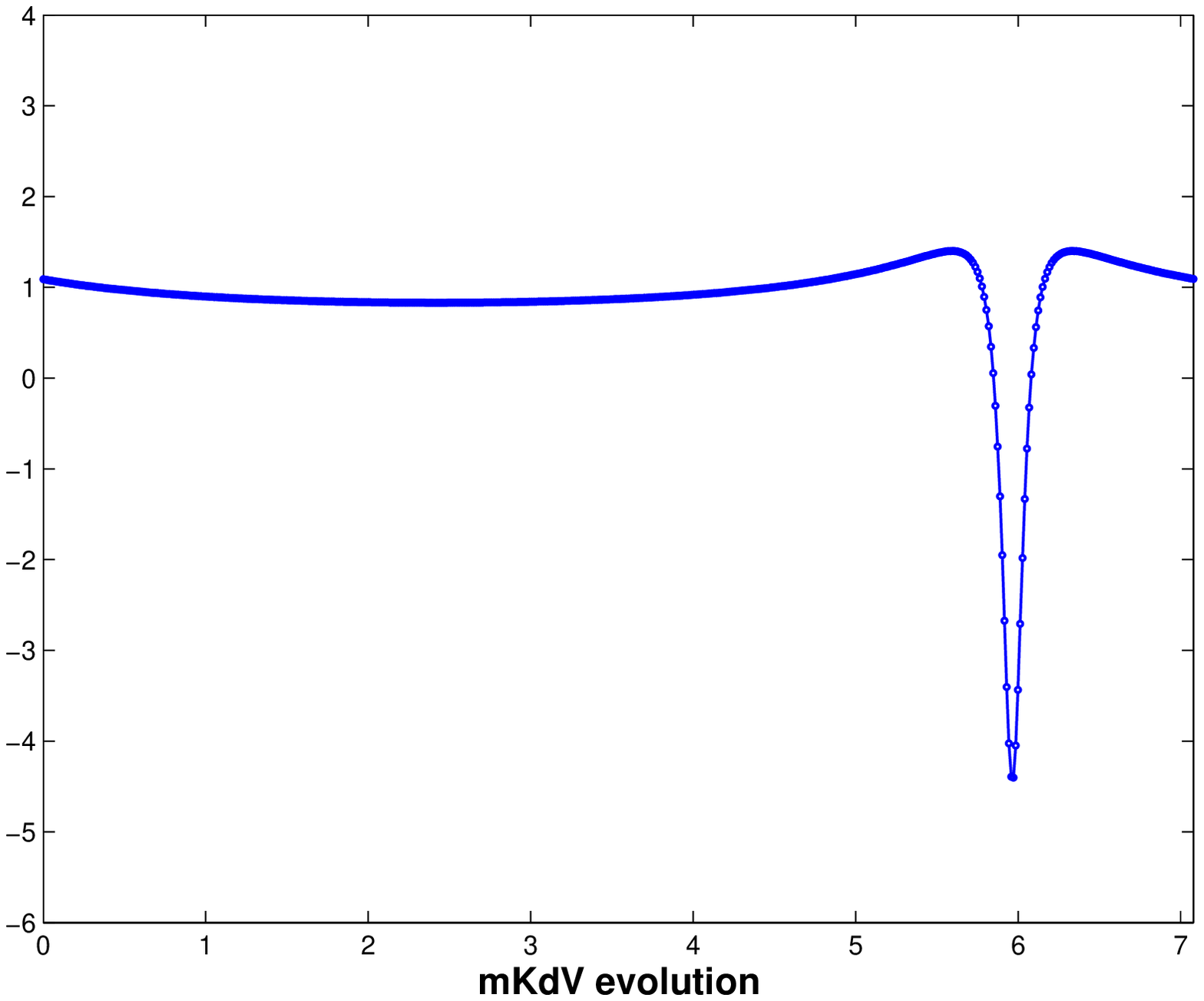}{}}\hfill
      \caption{T=0}\label{fig:CCNADBk0}
    \end{subfigure}
    \begin{subfigure}[b]{.22\textwidth}
  \hbox{\includegraphics[width=\linewidth]{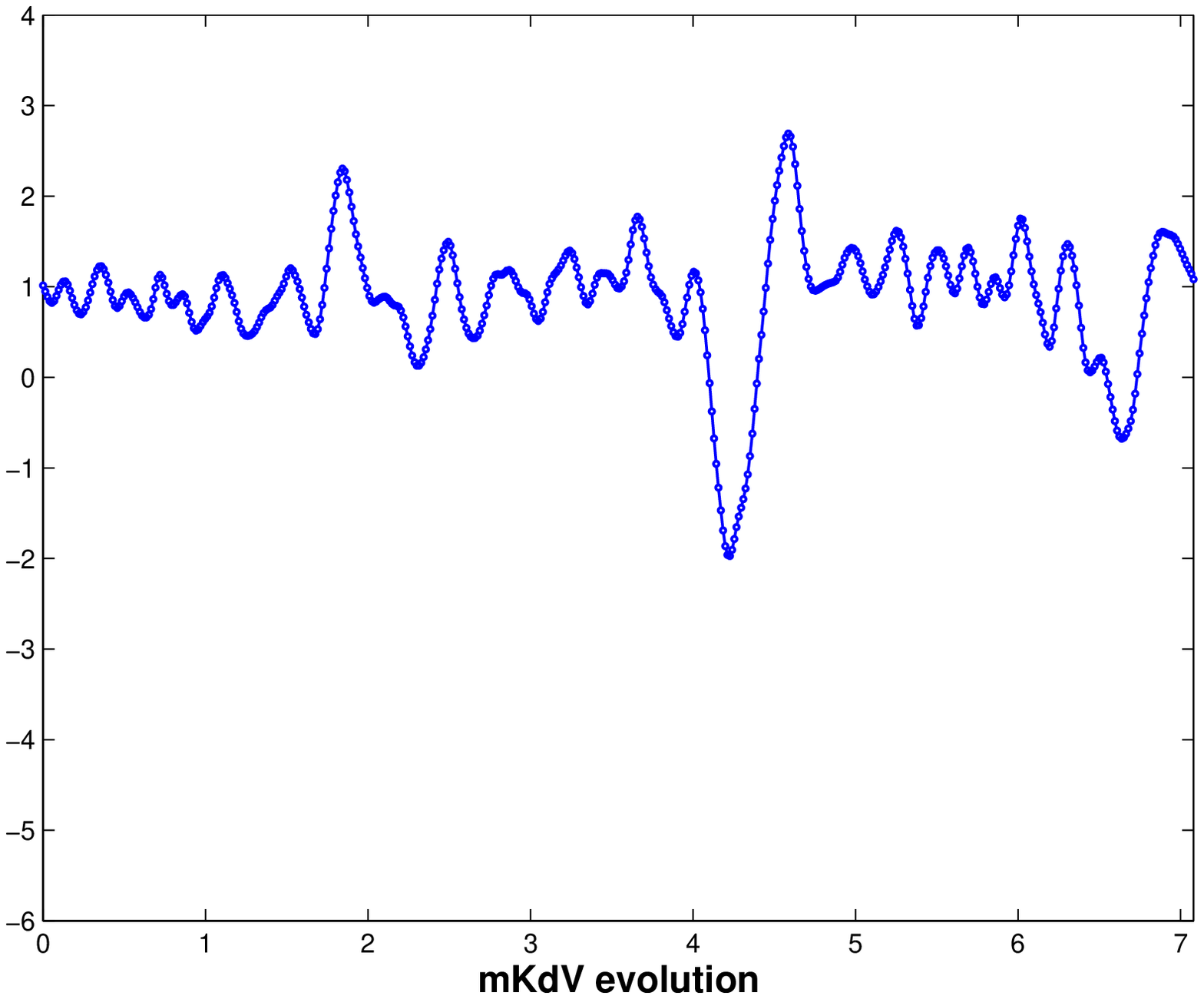}{}}\hfill
      \caption{T=1}\label{fig:CCNADBk1}
    \end{subfigure}
    
    \begin{subfigure}[b]{.22\textwidth}
  \hbox{\includegraphics[width=\linewidth]{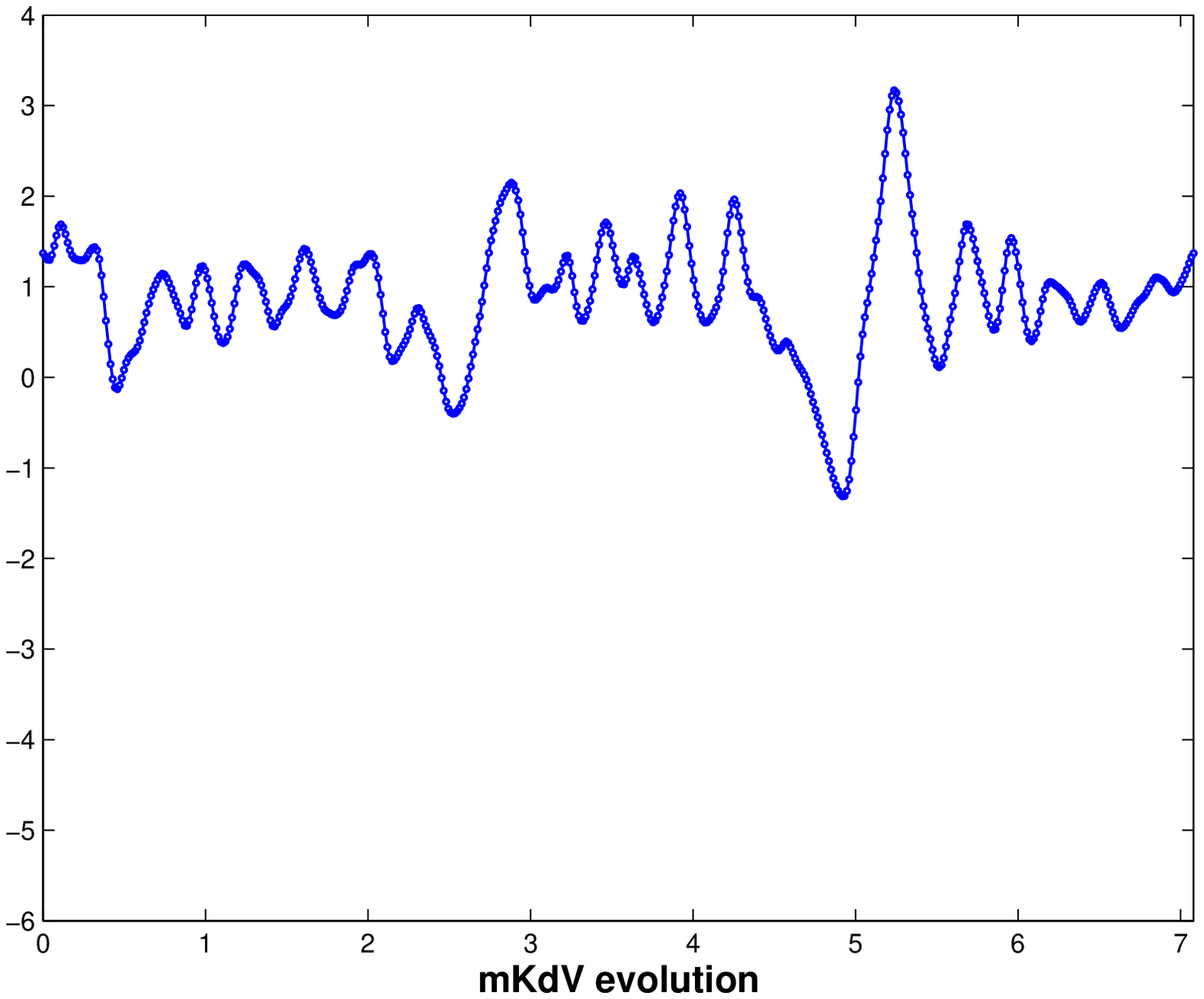}{}}\hfill
      \caption{T=2} \label{fig:CCNADBk2}
    \end{subfigure}
    \begin{subfigure}[b]{.22\textwidth}
  \hbox{\includegraphics[width=\linewidth]{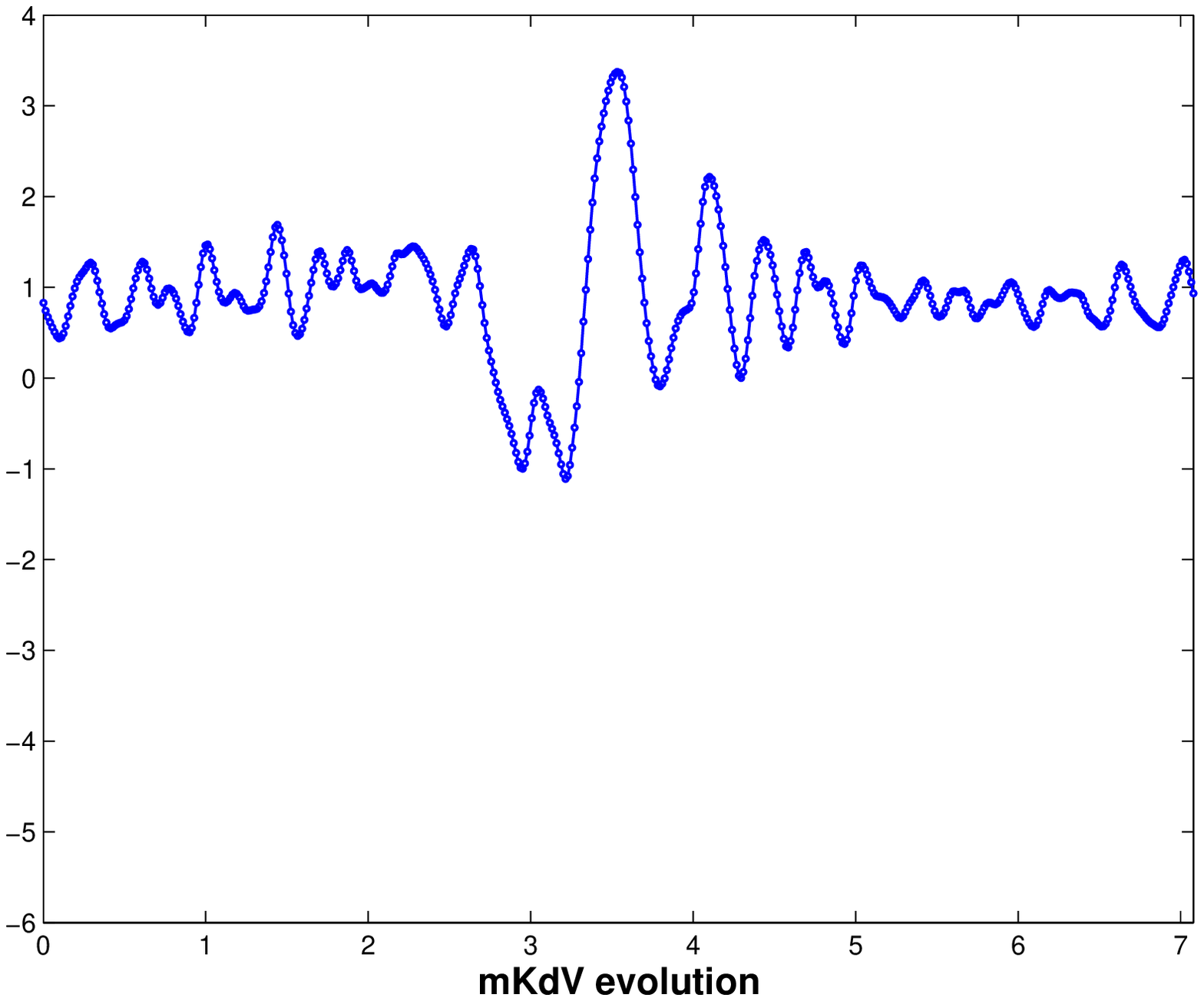}{}}\hfill
      \caption{T=3}\label{fig:CCNADBk3}
    \end{subfigure}
    \begin{subfigure}[b]{.22\textwidth}
  \hbox{\includegraphics[width=\linewidth]{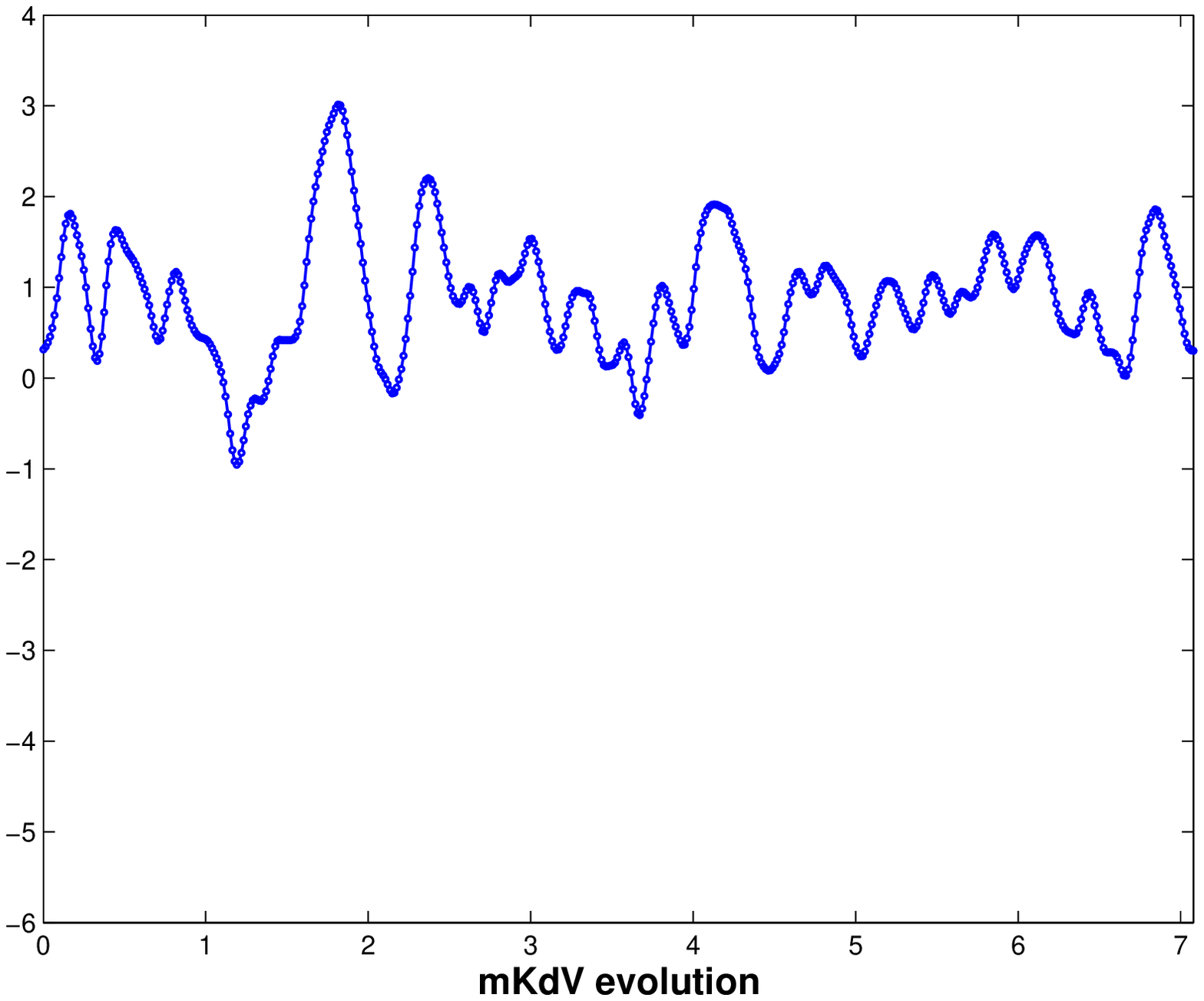}{}}\hfill
      \caption{T=4}\label{fig:CCNADBk4}
    \end{subfigure}
    \begin{subfigure}[b]{.22\textwidth}
  \hbox{\includegraphics[width=\linewidth]{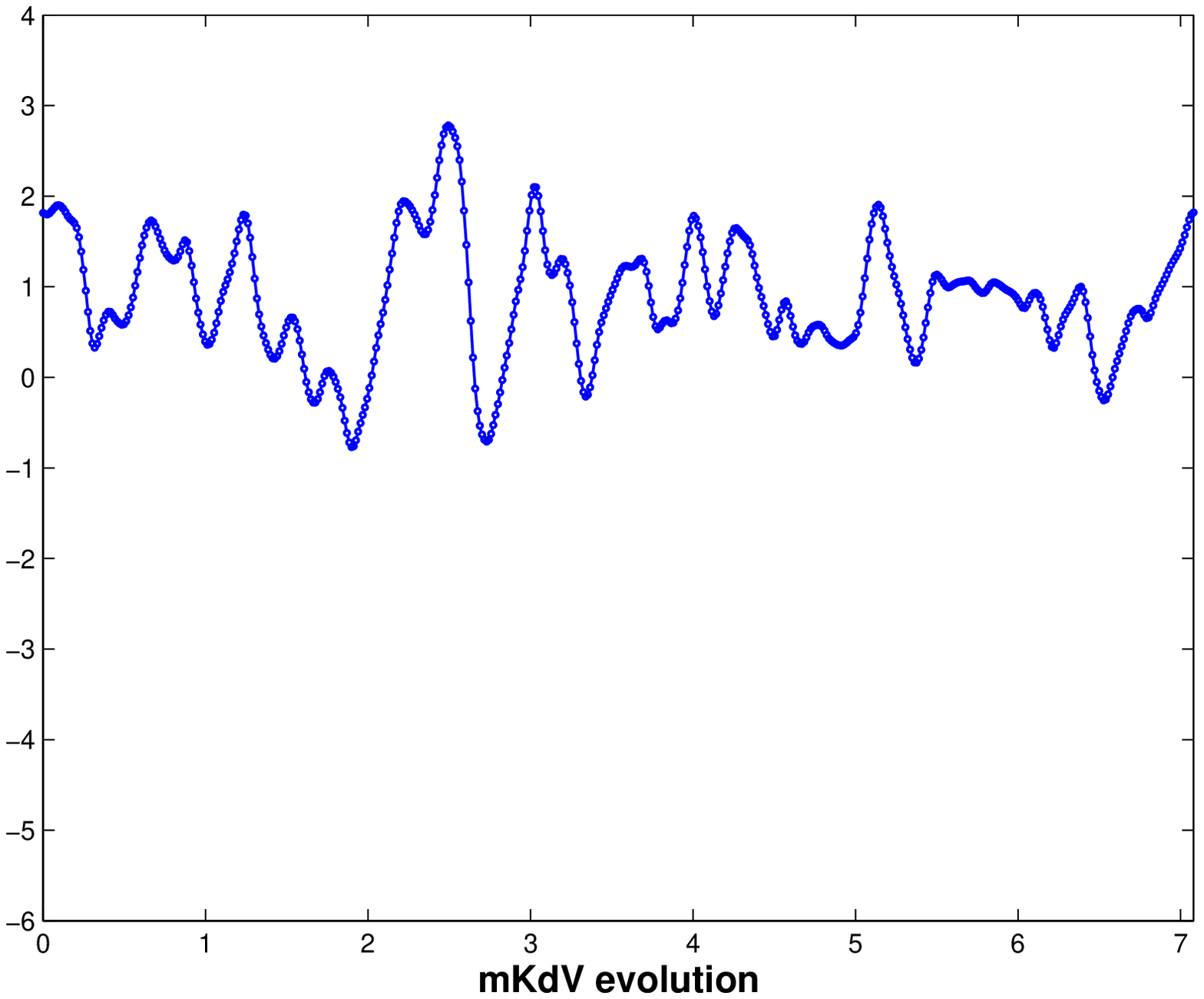}{}}\hfill
      \caption{T=5} \label{fig:CCNADBk5}
    \end{subfigure}
  \caption{The morphologies $(\ref{fig:CCNADBxy0})$-$(\ref{fig:CCNADBxy5})$ and curvatures $(\ref{fig:CCNADBk0})$-$(\ref{fig:CCNADBk5})$ for the evolution of a curve, starting from a cardioid. The results were generated using the CNADB scheme with $N=512$ and $\Delta t=1\times 10^{-5}$}
\label{fig:carevol}
\end{figure}

\section{Discussion and conclusion} \label{affinefuture}

In this paper, we presented the adaptation of the methods developed by Hou-Lowengrub-Shelley to the evolution of 2-D curves that follow Airy Flow under equal arc-length parametrization, and solutions of mKdV equation through the curvature of the curve.\\

Three pseudo-spectral schemes were analyzed; the first one uses an integrating factor technique and ADB method as time discretization, the second scheme is based on CN, and the CNADB scheme that combines the first two schemes at the first step. Linear analysis, and numerical conservation of first integrals of motion for mKdV equation confirmed the accuracy in each case. A fully discrete analysis of the equations confirmed the convergence properties (second-order accurate in time and spectral accurate in space). Numerical analysis displayed a dominant accuracy of ADB scheme over CN implementation. The scheme CNADB shares the stability properties of CN with smaller errors over time. Instabilities in the ADB implementation increase as the spatial resolution increase (mainly due aliasing error). We studied two filters to overcome these instabilities by damping out the effects of high wave numbers and small amplitudes. Filtering is not needed for numerical or analytical convergence. However, its selective application over ADB improves its stability while retaining accuracy. CNADB was preferred to compute evolutions due its stability properties for longer periods of time mainly constraint by the complexity of the initial shape. \\
      
Numerical analysis displayed the appropriate convergence rates and optimal temporal stability constraint $\frac{\Delta t}{h}\leq C$ where the constant $C$ is independent of discretizations but depend on the smoothness of the solution, its derivatives (up to order 7 for ADB and order 9 for CN),  and the interval of evolution. This confirms that the developed methods are efficient, stable, and convergent to the solutions of Airy flow and mKdV equations.

\begin{appendices}

\numberwithin{equation}{section}

  \renewcommand\thetable{\thesection\arabic{table}}
  \renewcommand\thefigure{\thesection\arabic{figure}}
  \section{} 
  \label{appe1}

 \subsection{ Dynamics of curvature $k$, arc-length variation $s_{\alpha}$ and angle between the tangent vector and $x$-axis}
 \label{calcdynksalthet}
   By continuity of second derivatives
   \begin{equation}
   \begin{split}
     X_{\alpha t} &=X_{t \alpha}=\frac{\partial X_t}{\partial \alpha}=s_{\alpha}\frac{\partial [V\textbf{n}+T\textbf{s}]}{\partial s}=s_{\alpha}[V_s\textbf{n}+V\textbf{n}_s+T_s\textbf{s}+T\textbf{s}_s],\\
        s_{\alpha t}&=\frac{x_{\alpha}x_{\alpha t}+y_{\alpha}y_{\alpha t}}{s_{\alpha}}=\frac{X_{\alpha}\cdot X_{\alpha t}}{s_{\alpha}}.
      \end{split}
       \end{equation}
        
  Now, since $X_{\alpha}=s_{\alpha} \textbf{s}$ and using the Frenet Formulas we obtain $\textbf{n}_s=k\textbf{s}$,  $\textbf{s}_s=-k\textbf{n}$ where $k$ represents the curvature. Therefore
  $$s_{\alpha t}=s_{\alpha} \textbf{s} \cdot [V_s\textbf{n}+Vk\textbf{s}+T_s\textbf{s}+T(-k\textbf{n})]=s_{\alpha}(T_s+Vk).$$
  
  Also, we can derive the rate of change for $\theta=\arctan(\frac{y_{\alpha}}{x_{\alpha}})$ (angle between the tangent vector to the curve and the $x$-axis) as follows  
  \begin{equation}
  \begin{split}
  \theta_t &=\frac{1}{1+(\frac{y_{\alpha}}{x_{\alpha}})^2}\frac{x_{\alpha}y_{\alpha t}-y_{\alpha}x_{\alpha t}}{(x_{\alpha}^2)}=\frac{x_{\alpha}y_{\alpha t}-y_{\alpha}x_{\alpha t}}{s_{\alpha}^2}\\
  &=-\textbf{n}\cdot [V_s\textbf{n}+Vk\textbf{s}+T_s\textbf{s}+T(-k\textbf{n})]=-V_s+kT,
\end{split}  
\end{equation}
therefore

  \begin{equation}\label{classc2}
  \begin{split}
  \theta_{\alpha t}=\theta_{t \alpha}=\frac{\partial [-V_s+kT] }{\partial \alpha}=s_{\alpha}\frac{\partial [-V_s+k T]}{\partial s}=s_{\alpha}[-V_{ss}+k_sT+kT_s].
  \end{split}
    \end{equation}

  Since $k=\frac{x_{\alpha}y_{\alpha \alpha}-x_{\alpha \alpha} y_{\alpha}}{s_{\alpha}^3}$ and $k=\frac{\theta_{\alpha}}{s_{\alpha}}$, we get
 \begin{equation}\label{kevt}
 \begin{split}
  k_t&=\frac{s_{\alpha} \theta_{\alpha t}-\theta_{\alpha} s_{\alpha t}}{s_{\alpha}^2}=\frac{s_{\alpha}^2[-V_{ss}+k_sT+kT_s]-s_{\alpha}^2k[T_s+kV]}{s_{\alpha}^2}\\
 & =-V_{ss}+k_sT-k^2V=-(\frac{\partial ^2}{\partial s \partial s}+k^2)V+Tk_s.
  \end{split}
 \end{equation}

 
\subsection{Linear analysis}\label{calclinanalyksflow}
Let $X(\alpha,t)$ be given in equation $(\ref{linearX})$. Then,

   \begin{equation}
  \begin{split}
  X_{\alpha}&=r_{\alpha}(cos\alpha,\sin\alpha)+r(-\sin\alpha,\cos\alpha)\\
  X_{\alpha \alpha}&=r_{\alpha \alpha}(\cos\alpha, \sin\alpha)+r(-\cos \alpha,-\sin\alpha)+r_{\alpha}(-2\sin\alpha,2\cos\alpha).
  \end{split}
    \end{equation}
    
  This implies that
\begin{equation}
  s_{\alpha}=\sqrt{(r_{\alpha}\cos \alpha+r(-\sin\alpha))^2+(r_{\alpha}\sin\alpha+r\cos\alpha)^2}=\sqrt{r^2+r_{\alpha}^2},
 \end{equation}
  and 
    \begin{equation}\label{curvf}
  k=\frac{r^2+2r_{\alpha}^2-rr_{\alpha \alpha}}{(\sqrt {r^2+r_{\alpha}^2})^3}.
  \end{equation}
  
The normal vector and differential of the arc-length are
   \begin{equation}
  \begin{split}
  \textbf{n}&=\frac{(r_{\alpha}\sin \alpha+r\cos\alpha,r\sin\alpha-r_{\alpha}\cos\alpha)}{\sqrt{r^2+r_{\alpha}^2}},\\
     s_{\alpha}&=\sqrt{R^2+2R(\delta_R\cos(m\alpha)-\delta_I\sin(m\alpha))}+O(\delta^2)\\
   &=R[1+\frac{\delta_R\cos(m\alpha)-\delta_I\sin(m\alpha)}{R}+O(\delta^2)].
  \end{split}   
   \end{equation}
   
  Therefore
  \begin{equation}
   \frac{1}{s_{\alpha}}=\frac{1}{R}[1-\frac{\delta_R\cos(m\alpha)-\delta_I\sin(m\alpha)}{R}+O(\delta^2)].
   \end{equation}

Using the expression for the curvature $(\ref{curvf})$, it follows that the normal velocity is
  \begin{equation}\label{ks}
  k_s=\frac{1}{s_{\alpha}}\frac{\partial \left[ \frac{r^2+2r_{\alpha}^2-rr_{\alpha \alpha}}{(\sqrt {r^2+r_{\alpha}^2})^3}\right]}{\partial \alpha}=\frac{1}{s_{\alpha}^4}(2rr_{\alpha}+3r_{\alpha}r_{\alpha\alpha}-rr_{\alpha\alpha\alpha})-\frac{3}{s_{\alpha}^5}s_{\alpha\alpha}(r^2+2r_{\alpha}^2-rr_{\alpha\alpha}).
   \end{equation}
   
   A simple computation shows
   \begin{equation}
  \begin{split} 
    &r_{\alpha}=m(-\delta_R\sin(m\alpha)-\delta_I\cos(m\alpha))\\
    & r_{\alpha \alpha}=m^2(-\delta_R\cos(m\alpha)+\delta_I\sin(m\alpha))\\
    & r_{\alpha\alpha\alpha}=m^3(\delta_R\sin(m\alpha)+\delta_I\cos(m\alpha))\\
    & s_{\alpha\alpha}=\frac{rr_{\alpha}+r_{\alpha}r_{\alpha \alpha}}{s_{\alpha}}. 
  \end{split}
   \end{equation}
  
    Thus, after substitution in $(\ref{ks})$ we can write
   \begin{equation}
  \begin{split}     
   &k_s=\frac{1}{s_{\alpha}^4}\left[2Rm[-\delta_R\sin(m\alpha)-\delta_I\cos(m\alpha)]-Rm^3(\delta_R\sin(m\alpha)+\delta_I\cos(m\alpha))\right]\\
   &-\frac{3}{s_{\alpha}^6}[rr_\alpha+r_{\alpha}r_{\alpha \alpha}][R^2+2R(\delta_R\cos(m\alpha)-\delta_I\sin(m\alpha))-Rm^2(-\delta_R\cos(m\alpha)+\delta_I\sin(m\alpha))]\\
   &=\frac{1}{R^3}\{[-m(\delta_R\sin(m\alpha)+\delta_I\cos(m\alpha))(2+m^2)]+3m(\delta_R\sin(m\alpha))+\delta_I\cos(m\alpha)\}+O(\delta^2)\\
   &=\frac{1}{R^3}(\delta_R\sin(m\alpha)+\delta_I\cos(m\alpha))(m-m^3)+O(\delta^2).
  \end{split}     
   \end{equation}
   
   The tangential velocity can be computed similarly:
   \begin{equation}
  \begin{split} 
  &T=\frac{\left(   \frac{r^2+2r_{\alpha}^2-rr_{\alpha \alpha}}{(\sqrt {r^2+r_{\alpha}^2})^3}    \right)^2}{2}=\frac{1}{2}\left( (\frac{1}{R^3}-3\frac{\gamma}{R^4})(R^2+2R\gamma-R(-m^2)\gamma)  \right)^2\\
&=\frac{1}{2}\left(  \frac{1}{R}-\frac{1}{R^2}3\gamma+\frac{\gamma}{R^2}(2+m^2)) \right)^2+O(\delta^2)=\frac{1}{2} \left(\frac{1}{R^2}+\gamma(m^2-1)   \right)
  \end{split}     
   \end{equation}
where $\gamma=\delta_R\cos(m\alpha)-\delta_I\sin(m\alpha)$. Therefore, the velocity is
   \begin{equation}
  \begin{split} 
   &W:=-k_s\textbf{n}+T\textbf{s}=\\
   &(\frac{1}{R^3}(\delta_R\sin(m\alpha)+\delta_I\cos(m\alpha))(m^3-m))\textbf{n}+\frac{1}{2} \left(\frac{1}{R^2}+\gamma(m^2-1)   \right) \textbf{s}+O(\delta^2),
  \end{split}     
   \end{equation}   
   whose projection is
\begin{equation} \label{velproj}  
 W \cdot (\cos \alpha,\sin\alpha)=\delta_I \left( \frac{m^3-1.5m}{R^3}  \right)\cos(m\alpha)+ \delta_R \left(\frac{m^3-1.5m}{R^3}  \right)\sin(m\alpha)+O(\delta^2).
  \end{equation}


\subsection{Direct calculation 1} \label{dircalc1} 
  We can write the equations based on CN scheme $(\ref{FirstCN})$ as follows
\begin{equation} \label{varCN1}
\widehat{\dot{\theta}_m^{j+1}}=\zeta_m^1\widehat{\dot{\theta}_m^{j-1}}+2\Delta t \zeta_m^2  \widehat{\dot{ NL}_m^j}+A_0(\Delta t^3),
\end{equation}
\begin{equation}\label{varCN2}
\widehat{\dot{\theta}_m^{j-1}}=\zeta_m^1\widehat{\dot{\theta}_m^{j-3}}+2\Delta t \zeta_m^2  \widehat{\dot{ NL}_m^{j-2}}+A_0(\Delta t^3),
\end{equation}
to obtain
  \begin{equation} 
  \begin{split}     
&\langle \widehat{\dot{\theta}^{j+1}}-\widehat{\dot{\theta}^{j-1}}, \widehat{\dot{\theta}^{j+1}}+\widehat{\dot{\theta}^{j-1}} \rangle=\\
&\langle\zeta_m^1\widehat{\dot{\theta}_m^{j-1}}+2\Delta t \zeta_m^2  \widehat{\dot{ NL}_m^j}-\left( \zeta_m^1\widehat{\dot{\theta}_m^{j-3}}+2\Delta t \zeta_m^2  \widehat{\dot{ NL}_m^{j-2}}+A_0(\Delta t^3) \right),\\
&\zeta_m^1\widehat{\dot{\theta}_m^{j-1}}+2\Delta t \zeta_m^2  \widehat{\dot{ NL}_m^j}+\zeta_m^1\widehat{\dot{\theta}_m^{j-3}}+2\Delta t \zeta_m^2  \widehat{\dot{ NL}_m^{j-2}}+A_0(\Delta t^3)\rangle\\
&=\langle \zeta_m^1\left(\widehat{\dot{\theta}_m^{j-1}} -\widehat{\dot{\theta}_m^{j-3}}\right)+2\Delta t \zeta_m^2\left( \widehat{\dot{ NL}_m^j}- \widehat{\dot{ NL}_m^{j-2}}\right)+A_0(\Delta t^3),\\
& \zeta_m^1\left(\widehat{\dot{\theta}_m^{j-1}} +\widehat{\dot{\theta}_m^{j-3}}\right)+2\Delta t \zeta_m^2\left( \widehat{\dot{ NL}_m^j}+ \widehat{\dot{ NL}_m^{j-2}}\right)+A_0(\Delta t^3)\rangle.
  \end{split}     
  \end{equation} 
   

\subsection{Nonlinear error estimates} \label{varnl}
In this proof, the hypothesis  $2\leq r$ is used, which is satisfied in both schemes ($4\leq r$ for ADB and $6\leq r$ for CN). 
\begin{proof}[Lemma \ref{nlvarlemma}]\label{proofnlvar}
 We start computing upper bounds for $|S_h\dot{f^j}|_{\infty},|\dot{f^j}|_{\infty}$.
 Notice that
\begin{equation}
h|\dot{f}^j|^2\leq \sum_{i=-N/2+1}^{N/2}|\dot{f_i^j}|^2h=||\dot{f^j}||_{l^2}^2,
\end{equation}
which implies\\
\begin{equation}\label{infnorm}
|\dot{f_m^j}|_{\infty}\leq h^{-1/2}||\dot{f^j}||_{l^2} \text{ and, }
|S_h\dot{f^j}|_{\infty},|S_h\dot{f^j}|_{\infty}\leq h^{-3/2}||\dot{f^j}||_{l^2}.\\
\end{equation}

Then, condition $2\leq r$ and the definition of $T^*$ shows that\\
\begin{equation}\label{bounDth}
|S_h\dot{\theta^j}|_{\infty},|\dot{\theta^j}|_{\infty}\leq C \text{ are bounded.}\\
\end{equation}

Now, we define 
\begin{equation} 	 
 NL(\alpha_m,t_j):=\frac{1}{2}(\frac{2\pi}{L})^3( \theta_{\alpha}(\alpha_m,t_j))^3 \text{ and }\: \widetilde{NL_m^j}:=\frac{1}{2}(\frac{2\pi}{\widetilde{L}})^3(S_h \widetilde{\theta_m^j})^3,
\end{equation} 
to write the nonlinear terms as follows:
\begin{equation}\label{nlvareq}
\dot{NL}_m^j:=\widetilde{NL_m^j}-NL(\alpha_m,t_j)=\dot{\xi}_m^jT_m^j+\dot{\xi}_m^j\dot{T}_m^j+\xi_m^j\dot{T}_m^j,
\end{equation}
where
\begin{equation} 	 
	 \xi_m^j:=\frac{2\pi}{L}S_h\theta_m^j,\: \xi(\alpha_m,t_j):=\frac{2\pi}{L} \theta_{\alpha}(\alpha_m,t_j),\: \widetilde{\xi_m^j}:=\frac{2\pi}{\widetilde{L}}S_h \widetilde{\theta_m^j},
\end{equation} 
\begin{equation} 	 
\dot{\xi}_m^j:=(\widetilde{\xi_m^j}- \xi_m^j)+(\xi_m^j-\xi(\alpha_m,t_j)),
\end{equation} 
and
\begin{equation} 	 
	 T_m^j:=\frac{1}{2}(\frac{2\pi}{L})^2(S_h\theta_m^j)^2,\: T(\alpha_m,t_j):=\frac{1}{2}(\frac{2\pi}{L})^2( \theta_{\alpha}(\alpha_m,t_j))^2,\: \widetilde{T_m^j}:=\frac{1}{2}(\frac{2\pi}{\widetilde{L}})^2(S_h \widetilde{\theta_m^j})^2.
\end{equation} 
	 
\begin{subsubsection}{Error in tangential velocity: } 	
We calculate the error for the tangential velocity\\
\begin{equation}\label{varTeq}
\dot{T}_m^j:=\widetilde{T_m^j}-T(\alpha_m,t_j)=\frac{1}{2}\left[ \widetilde{\xi_m^j}^2-(\xi_m^j)^2\right]=\frac{1}{2}\left[ 2\dot{\xi}_m^j \xi_m^j+(\dot{\xi}_m^j)^2\right].
\end{equation}

Since the truncation error $(\xi_m^j-\xi(\alpha_m,t_j))=O(h^{r+2})$, it follows that

	$$\dot{\xi}_m^j=\frac{2\pi}{\widetilde{L}}S_h\widetilde{\theta_m^j}-\frac{2\pi}{L}S_h\theta_m^j+O(h^{r+2})=2\pi S_h\dot{\theta}_m^j\dot{(L^{-1})}+\frac{2\pi}{L}S_h\dot{\theta}_m^j+2\pi S_h\theta_m^j\dot{(L^{-1})}+O(h^{r+2}),$$
	where
$$\dot{L^{-1}}=\frac{1}{\widetilde{L}}-\frac{1}{L}=\frac{-\dot{L}}{L^2}+\frac{(\dot{L})^2}{L^2[L+\dot{L}]},$$
	which implies $|\dot{L^{-1}}|\leq C|\dot{L}|$, and therefore
	\begin{equation}\label{Linveq}
	 A_0\dot{(L^{-1})}=A_0(\dot{L}).
	 \end{equation}

	Combining $(\ref{bounDth})$, $(\ref{Linveq})$ and
	 \begin{equation}\label{dicsreal}
	 S_h\theta_m^j=\theta_{\alpha}(\alpha_m,t_j)+O(h^{r+2}).
	\end{equation}
	 in the expression for $\dot{\xi}_m^j$ we see that
	\begin{equation}
	\dot{\xi}_m^j=\frac{2\pi}{L}S_h\dot{\theta}_m^j +A_0(\dot{L})+O(h^{r+2}).
	\end{equation}
	
	By hypothesis over time $(\ref{timeHyCN}),(\ref{timeHy})$ and computation $(\ref{infnorm})$ we find that
	\begin{equation}\label{a0varL}
	 |A_0(\dot{L})|_{\infty}\leq Ch^{-1/2}||\dot{L}||_{l^2}=h^{-1/2}O(h^{r+3})=O(h^{r+5/2}).
	 \end{equation}
	 
	 Hence, we rewrite
	\begin{equation}\label{xieq}
	\dot{\xi}_m^j=\frac{2\pi}{L}S_h\dot{\theta}_m^j +O(h^{r+2})=h^{-1}A_0(\dot{\theta}^j)+O(h^{r+2}).
	\end{equation}	 
	 
	 Also $||O(h^{r+2})||_{l^2}=O(h^{r+2})$ and the conditions $\frac{\Delta t}{h}\leq C$, $2\leq r$  imply that
	\begin{equation}\label{infNxieq}
	\begin{split}
	&||\dot{\xi}^j||_{l^2}=O(h^{-1}(h^r+\Delta t^2))+O(h^{r+2}) =O(h) \text{ and }\\
         &|\dot{\xi}|_{\infty}=h^{-1/2}||\dot{\xi}^j||_{l^2}\leq C\text{ are bounded.}
         \end{split}
	\end{equation}
	
         With this information $(\ref{infNxieq})$ back to $(\ref{varTeq})$ is possible to rewrite
	 $$\dot{T}_m^j=\frac{1}{2}\left[  2\left(\frac{2\pi}{L}S_h\dot{\theta}_m^j+O(h^{r+2})\right)\frac{2\pi}{L}S_h\theta_m^j+ \left(\frac{2\pi}{L}S_h\dot{\theta}_m^j+O(h^{r+2})\right)^2 \right].$$
	 
          Again using $(\ref{dicsreal})$, the fact that $\theta_{\alpha}$ is bounded and the upper bound $(\ref{infNxieq})$, the first term on the right hand side becomes
   	\begin{equation}          
        \begin{split}
          &\left(\frac{2\pi}{L}S_h\dot{\theta}_m^j+O(h^{r+2})\right)\frac{2\pi}{L}S_h\theta_m^j=\left(\frac{2\pi}{L}S_h\dot{\theta}_m^j+O(h^{r+2})\right) \left(\frac{2\pi}{L}\theta_{\alpha}(\alpha_m,t_j)+O(h^{r+2})\right)\\
          &=\left(\frac{2\pi}{L}S_h\dot{\theta}_m^j+O(h^{r+2})\right) \frac{2\pi}{L}\theta_{\alpha}(\alpha_m,t_j)+ \dot{\xi}_m^j   O(h^{r+2})=\theta_{\alpha}(\alpha_m,t_j)(\frac{2\pi}{L})^2S_h\dot{\theta}_m^j+O(h^{r+2}).
        \end{split}
	\end{equation}       
	  
          As shown previously  $S_h\dot{\theta}_m^j,\dot{\xi}_m^j$, are bounded. Thus, the second term can be computed as
   	\begin{equation}          
        \begin{split}          
          &\left(\frac{2\pi}{L}S_h\dot{\theta}_m^j+O(h^{r+2})\right)^2=\left(\frac{2\pi}{L}S_h\dot{\theta}_m^j+O(h^{r+2})\right)\dot{\xi}_m^j=\frac{2\pi}{L}S_h\dot{\theta}_m^j \dot{\xi}_m^j+O(h^{r+2})\\
          &=\frac{2\pi}{L}S_h\dot{\theta}_m^j\left(\frac{2\pi}{L}S_h\dot{\theta}_m^j+O(h^{r+2})\right)+O(h^{r+2})=\left(\frac{2\pi}{L}S_h\dot{\theta}_m^j\right)^2+O(h^{r+2})
        \end{split}
	\end{equation}     
	      
          Using $2\leq r$ we verify the inequality
	 \begin{equation}\label{Dhsquare}
	 ||(S_h\dot{\theta}_m^j)^2||_{l^2}\leq h^{-2}||\dot{\theta}_m^j||^2_{l^2}=O(h^{-2})O(h^r+\Delta t^2)||\dot{\theta}^j||_{l_2}=A_0(\dot{\theta}^j).
	 \end{equation}
	 
	 Therefore
	\begin{equation}\label{vartan}
	\dot{T}_m^j=(\frac{2\pi}{L})^2\theta_{\alpha}(\alpha_m,t_j)S_h\dot{\theta}_m^j+A_0(\dot{\theta}^j)+O(h^{r+2})=h^{-1}A_0(\dot{\theta}^j)+A_0(\dot{\theta}^j)+O(h^{r+2}).
	\end{equation}
	
	By hypothesis for time $T^*$ and $2\leq r$, we find that
	\begin{equation}\label{Tbound}
	||\dot{T^j}||_{l^2}=O(h^{-1}(h^r+\Delta t^2))+O(h^r+\Delta t^2)\leq C,\: |\dot{T^j}|_{\infty}\leq h^{-1/2}||\dot{T^j}||_{l^2}\leq C  
	\end{equation}
	\text{ are bounded quantities.} 
 \end{subsubsection}


\begin{subsubsection}{Error for nonlinear term: }  Combining equation ($\ref{xieq}$),(\ref{infNxieq})($\ref{vartan}$),($\ref{Tbound}$), we approximate (\ref{nlvareq}) to obtain
   	\begin{equation}          
        \begin{split}     
&\dot{NL}_m^j=\\
&\left\{ \frac{2\pi}{L}S_h\dot{\theta}_m^j +O(h^{r+2})\right\}T_m^j+\dot{\xi}_m^j\dot{T}_m^j+\xi_m^j\left\{  \frac{(2\pi)^2}{L^2}\theta_{\alpha}(\alpha_m,t_j) S_h\dot{\theta}_m^j+A_0(\dot{\theta}^j)+O(h^{r+2})  \right\}\\
&=\frac{2\pi}{L}S_h\dot{\theta}_m^jT_m^j+\dot{\xi}_m^j\dot{T}_m^j+\xi_m^j(\frac{2\pi}{L})^2{\theta_{\alpha}(\alpha_m,t_j)}^2 S_h\dot{\theta}_m^j+A_0(\dot{\theta}^j)+O(h^{r+2}).
        \end{split}
	\end{equation}

The second term on the previous expression can be analyzed using equations (\ref{bounDth})  as follows
   	\begin{equation}          
        \begin{split}    
&\dot{\xi}_m^j\dot{T}_m^j=\left( \frac{2\pi}{L}S_h\dot{\theta}_m^j+O(h^{r+2})\right)\left( (\frac{2\pi}{L})^2\theta_{\alpha}(\alpha_m,t_j)S_h\dot{\theta}_m^j+A_0(\dot{\theta}_m^j)+O(h^{r+2}) \right)\\
&= \frac{2\pi}{L}\left(S_h\dot{\theta}_m^j\right)^2 (\frac{2\pi}{L})^2\theta_{\alpha}(\alpha_m,t_j)+A_0(\dot{\theta^j})+O(h^{r+2}),
        \end{split}
	\end{equation}
which by the estimate $(\ref{Dhsquare})$ simplifies to $\dot{\xi}_m^j\dot{T}_m^j=A_0(\dot{\theta})+O(h^{r+2})$.

 As a consequence of the truncation error for the tangent velocity  $T_m^j-T(\alpha_m,t_j)=O(h^{r+2})$ and $\xi_m^j-\xi(\alpha_mt_j)=O(h^{r+2})$, we obtain
  	\begin{equation}          
        \begin{split} 
&\dot{NL}_m^j=\frac{2\pi}{L}S_h\dot{\theta}_m^j\left(\frac{1}{2}(\frac{2\pi}{L})^2\theta_{\alpha}^2(\alpha_m,t_j)+O(h^{r+2})\right)\\
&+\left(\frac{2\pi}{L}\theta_{\alpha}(\alpha_m,t_j)+O(h^{r+2}) \right)(\frac{2\pi}{L})^2{\theta_{\alpha}(\alpha_m,t_j)}^2 S_h\dot{\theta}_m^j+A_0(\dot{\theta}^j)+O(h^{r+2}).
        \end{split}
	\end{equation}

Finally, using equation $(\ref{bounDth})$ we attain an expression for the nonlinear error
\begin{equation}\label{nlvar0}
 \dot{NL}_m^j=\frac{3}{2}(\frac{2\pi}{L})^3\theta_{\alpha}^2(\alpha_m,t_j)S_h\dot{\theta}_m^j +A_0(\dot{\theta}^j)+O(h^{r+2})=h^{-1}A_0(\dot{\theta}^j)+A_0(\dot{\theta}^j)+O(h^{r+2}),
 \end{equation}
 and the upper bounds
 
 \begin{equation}\label{upperNL}
  \begin{split}
 ||\dot{NL}_m^j||_{l_2}&\leq \frac{3}{2}\frac{2\pi}{L}|k^2|_{\infty}h^{-1}||\dot{\theta}_m^j||_{l^2}+O(h^r+\Delta t^2)+O(h^{r+3/2})=O(h^{r-1}+h^{-1}\Delta t^2),\\
   |\dot{NL}^j|_{\infty}&=h^{-1/2}||\dot{NL}^j||_{l^2}\leq C,\\
\Delta t ||\dot{NL}_m^j||_{l_2}&=C\frac{\Delta t}{h}||\dot{\theta}_m^j||_{l^2}+O(h^r+\Delta t^2)=O(h^{r}+\Delta t^2), 
\end{split}
  \end{equation}
  
  for $j=1,...,n$, provided that $\frac{\Delta t}{h}$ is bounded. 
 \end{subsubsection}
  \end{proof}

 \begin{subsection}{Proof of convergence for Adams Bashforth (ADB) discretization.}
 \begin{proof}[Proof of Theorem \ref{disctheoADB}]
\label{discreteproofADB}  
 The error between the numerical and the exact solution (at a given time $t_j$) is given by
\begin{equation}
\dot{\theta}_m^j:=\widetilde{\theta_m^j}-\theta(\alpha_m,t_j).\label{ThetaerrorVar}
\end{equation}

Defining the auxiliary time, 
\begin{equation}\label{timeHy}
T^*=Sup\{t|t\leq T,|\dot{L}|<h^{r+3},||\dot{\theta}^j||_{l^2}=O(h^r+\Delta t^2)\},\\
\end{equation}
   for $j=0,1,...,n$  (we have an overall accuracy of $h^2$) we will show that the error at the step $n+1$ also satisfies the estimate $ ||\dot{\theta}^{n+1}||_{l^2}=O(h^r+\Delta t^2)$. Hence $T^*=T$ by induction.

 \paragraph*{Taylor approximations: \:} for the first step of the induction argument, we calculate upper bounds for the first step, based on a combination of Euler and integrating factor method (IFM) using the Taylor expansion:
 \begin{equation}\frac{\partial \Psi}{\partial t}=(rNL),\:\: r_m^t(t)=e^{i(2\pi m)^3tL^{-3}},\:\: \Psi(m,j)=r_m^t\widehat{\theta_m^j}.\end{equation}
  
 Expanding $\Psi$ around time $t_0$ and defining $\zeta_m=e^{-i(2\pi m)^3L^{-3}\Delta t}$, we obtain
 \begin{equation}\label{EStep}
 \widehat{ \theta_m^1}=\zeta_m ( \widehat{ \theta_m^0}+\Delta t \widehat{NL_m^0} )+\Delta t^2\frac{\zeta_m}{2}(\Psi_{tt})_m^0+O(\Delta t^3).
 \end{equation}
 
 The numerical solution satisfies  at the first step (Euler discretization)
  \begin{equation} \widehat{\widetilde{\theta_m^1}}=\zeta_m (\widehat{ \widetilde{ \theta_m^0}}+\Delta t \widehat{\widetilde{NL_m^0}} ) \end{equation}
 thus, we can write an expression for the error at the first step
  \begin{equation}\widehat{\dot{\theta}_m^{1}}=\zeta_m \widehat{\dot{\theta}_m^{0}} +\Delta t \zeta_m \widehat{\dot{NL}_m^{0}}+\Delta t^2\frac{\zeta_m}{2}(\Psi_{t t})_m^0+O(\Delta t^3).\end{equation}

Observe that $||\frac{\zeta}{2}\widehat{(\Psi_{t t})^0}||=\frac{1}{2}||\widehat{(\Psi_{t t})^0}||=\frac{1}{2\sqrt{2\pi}}||(\Psi_{t t})^0||_{l^2}$. We will see that the coefficients of $\Delta t^2$ term are bounded (independent of discretization) in $l^2$ norm.

Since
\begin{equation}{\Psi_t}_m=e^{i t(\frac{2\pi m}{L})^3}\widehat{NL_m},\end{equation}
then
\begin{equation}\label{psider2a}
{\Psi_{t t}}_{m}=e^{i t(\frac{2\pi m}{L})^3}\left((\widehat{NL^0_m})_t+\widehat{NL^0_m}i(\frac{2\pi m}{L})^3\right).
\end{equation}

Because $\theta$ is 2 times differentiable with respect to time (so we can commute derivatives), we can write $NL_t=-\frac{1}{2s_{\alpha}^3}\frac{3}{2}\theta_{\alpha}^2\theta_{\alpha t}=-\frac{1}{2s_{\alpha}^3}\frac{3}{2}\theta_{\alpha}^2\frac{1}{s_{\alpha}^3}[\theta_{\alpha\alpha \alpha \alpha}+(\frac{\theta_{\alpha}^3}{2})_{\alpha}]$ which involves spatial derivatives of order $4$ for theta. Hence by the assumption $4\leq r$, these derivatives are $L^2$ integrable. Also $(\widehat{NL_m})_{t}=(\widehat{{{NL_{t}}_m}})$ shows that $(\widehat{NL^0})_{t}$ is $l^2$ integrable.

To control the second term of $(\ref{psider2a})$ observe that 
\begin{equation}\widehat{(NL_{sss})^0_m}=-\widehat{NL^0_m}i(\frac{2\pi m}{L})^3,\end{equation}
and $NL_{sss}=\frac{1}{2s_{\alpha}^6}(\theta_{\alpha}^3)_{\alpha \alpha \alpha}$ involves $L^2$ integrable derivatives of order $4$ for theta. This shows that $\widehat{NL^0}i(\frac{2\pi m}{L})^3$ is bounded in the $l^2$ norm and consequently $||(\Psi_{t t})_m^0||_{l^2}$ is also bounded. In other words
 \begin{equation}\label{sharpEStep}
 \widehat{ \theta_m^1}=\zeta_m ( \widehat{ \theta_m^0}+\Delta t \widehat{NL_m^0} )+A_0(\Delta t^2),
 \end{equation}
 and
  \begin{equation}\label{vt2}
  ||\widehat{\dot{\theta}^1}||^2=\langle \zeta_m \widehat{\dot{\theta}_m^{0}} +\Delta t \zeta_m \widehat{\dot{NL}_m^{0}}+A_0(\Delta t^2), \zeta_m \widehat{\dot{\theta}_m^{0}} +\Delta t \zeta_m \widehat{\dot{NL}_m^{0}}+A_0(\Delta t^2)\rangle.
  \end{equation}
 
 Now, we analyze the error after the second step $(1\leq  j)$. Using Taylor's approximation we obtain
 \begin{equation}\label{realt}
  \widehat{\theta_m^{j+1}}=\widehat{\theta_m^j}\zeta_m+\frac{\Delta t}{2}(3\zeta_m \widehat{NL_m^j}-(\zeta_m)^2\widehat{NL_m^{j-1}})+\frac{5\Delta t^3}{12}(\Psi^{(3)})_m^{j}+O(\Delta t^4).
\end{equation}

On the other hand, the numerical solution $(\ref{discth})$ satisfies:\\
\begin{equation}\label{numt}
\widehat{\widetilde{\theta_m^{j+1}}}=\widehat{\widetilde{\theta_m^j}} \zeta_m+\frac{\Delta t}{2}(3 \zeta_m\widehat{\widetilde{NL_m^j}}- (\zeta_m)^2\widehat{\widetilde{NL_m^{j-1}}}).
\end{equation}

Subtracting $(\ref{realt})$ from $(\ref{numt})$  we obtain the following equation for the error in $\theta$:
\begin{equation}\label{discvartheta0}
\widehat{\dot{\theta}_m^{j+1}}=\zeta_m \widehat{\dot{\theta}_m^j}+\frac{\Delta t}{2}\widehat{\dot{ \mu}_m^j}+\frac{5\Delta t^3}{12}(\Psi^{(3)})_m^{j}+O(\Delta t^4),
\end{equation}
where 
\begin{equation}\label{nonleq}
\widehat{\dot{ \mu}_m^j}=3 \zeta_m\widehat{\dot{NL}_m^j}- (\zeta_m)^2\widehat{\dot{NL}_m^{n-1}}.
\end{equation}

Similarly way to the first step and for future estimates we show that the coefficient for the $\Delta t^3$ term in $(\ref{discvartheta0})$ is integrable in the $l^2$ norm.

From $(\ref{psider2a})$ we obtain

\begin{equation}\label{nltt}
\begin{split}
&{\Psi_{t t t}}_m=\\
&e^{i t(\frac{2\pi m}{L})^3}\left(  (\widehat{NL_m})_{t t}+(\widehat{NL_m})_{t}i(\frac{2\pi m}{L})^3 + i(\frac{2\pi m}{L})^3\left(  (\widehat{NL_m})_{t}+\widehat{NL_m}i(\frac{2\pi m}{L})^3\right)\right)\\
&=e^{i t(\frac{2\pi m}{L})^3}\left(  (\widehat{NL_m})_{t t}+2\underbrace{(\widehat{NL_m})_{t}i(\frac{2\pi m}{L})^3}_{T1} +\underbrace{( i(\frac{2\pi m}{L})^3)^2\widehat{NL_m}}_{T2}\right).
\end{split}
\end{equation}

Now 
\begin{equation}\label{NLterms}\begin{split}
&NL_{t t}=((\frac{\theta_s^3}{2})_t)_t=(\frac{1}{s_{\alpha}^3}\frac{3}{2}\theta_{\alpha}^2\theta_{\alpha t})_t=\frac{3}{2s_{\alpha}^3}\left(\theta_{\alpha}^2(\theta_{sss}+\frac{\theta_s^3}{2})_{\alpha}\right)_t\\
&=\frac{3}{2s_{\alpha}^3}\left(2\theta_{\alpha}\theta_{\alpha t}\theta_{t \alpha}+ \theta_{\alpha}^2(\theta_{sss}+(\frac{\theta_s^3}{2}))_{\alpha t}\right)\\
&=\frac{3}{2s_{\alpha}^3}\left(2\theta_{\alpha}\theta_{\alpha t}\theta_{t \alpha}+ \theta_{\alpha}^2(\theta_{sss}+(\frac{\theta_s^3}{2}))_{\alpha t}\right),
\end{split}
\end{equation}
 involves spatial derivatives of order $7$ for $\theta$, since $4\leq r$ (by hypothesis) we know these are $l^2$ integrable.   In addition temporal derivatives, and Fourier transform commute, we conclude that each term in $(\ref{NLterms})$ is $l^2$ integrable. Moreover,  $(NL_t)_{sss},\:(NL)_{ssssss}$ are also $l^2$ integrable provided $\theta$ is at least $7$ times differentiable. Consequently, terms $T1$ and $T2$ of $(\ref{nltt})$ are $l^2$  integrable too.

  This implies that $\Psi_{ttt}$ is also $l^2$ integrable and  we rewrite $(\ref{discvartheta0})$ as
 
\begin{equation}\label{discvartheta}
\widehat{\dot{\theta}_m^{j+1}}=\zeta_m \widehat{\dot{\theta}_m^j}+\frac{\Delta t}{2}\widehat{\dot{ \mu}_m^j}+A_0(\Delta t^3).
\end{equation}
 
To estimate the error consider the inner product
\begin{equation}\label{LHSADB}
\langle \widehat{\dot{\theta}^{j+1}}-\zeta \widehat{\dot{\theta}^{j-1}},  \widehat{\dot{\theta}^{j+1}}+\zeta \widehat{\dot{\theta}^{j-1}}\rangle=||\widehat{\dot{\theta}^{j+1}}||^2-|| \widehat{\dot{\theta}^{j-1}}||^2+2iIm(\langle \widehat{\dot{\theta}^{j+1}},\zeta  \widehat{\dot{\theta}^{j-1}} \rangle),
\end{equation}
 where we have used that $|\zeta_m|=1$ for each $m$ and $\zeta=(\zeta_{-N/2+1},..,\zeta_{N/2})$.

Using $(\ref{discvartheta})$ and
\begin{equation}\label{discvartheta2}
\widehat{\dot{\theta}_m^{j}}=\zeta_m \widehat{\dot{\theta}_m^{j-1}}+\frac{\Delta t}{2}\widehat{\dot{ \mu}_m^{j-1}}+A_0(\Delta t^3),
\end{equation}
into the main equation $(\ref{LHSADB})$ we obtain the right hand side 
\begin{equation}\label{RHS}
\begin{split}
&\langle \widehat{\dot{\theta}^{j+1}}-\zeta \widehat{\dot{\theta}^{j-1}},  \widehat{\dot{\theta}^{j+1}}+\zeta \widehat{\dot{\theta}^{j-1}}\rangle\\
&=\langle (\zeta-1) \widehat{\dot{\theta}^{j}}+ \frac{\Delta t}{2} (\widehat{\dot{\mu}^j}+\widehat{\dot{\mu}^{j-1}})+A_0(\Delta t^3), (\zeta+1) \widehat{\dot{\theta}^{j}}+ \frac{\Delta t}{2}( \widehat{\dot{\mu}^j}-\widehat{\dot{\mu}^{j-1}})+A_0(\Delta t^3) \rangle.
\end{split}
\end{equation}

By definition $(\ref{nonleq})$ of $\dot{\mu}_m^j$, using the estimate for the nonlinear error $\Delta t  ||\dot{NL}^j||_{l^2}=O(h^r+\Delta t^2)$ $(\ref{DtupperNL})$ and Plancherel theorem we have that
  \begin{equation}\label{eqDtmu}
  ||\Delta t\dot{ \mu}^j||_{l^2}=\Delta t|| 3 \zeta\widehat{\dot{NL}^j}- (\zeta)^2\widehat{\dot{NL}^{n-1}}||\leq C\Delta t(||\dot{NL}^j||_{l^2}+||\dot{NL}^{j-1}||_{l^2})=O(h^r+\Delta t^2),
  \end{equation}
for $j=1,...,n$.

Then we rewrite $(\ref{RHS})$ as follows
\begin{equation}
\begin{split}
&\underbrace{\langle (\zeta-1) \widehat{\dot{\theta}^{j}}, (\zeta+1) \widehat{\dot{\theta}^{j}}\rangle}_{J_1^j} +\\
&\underbrace{\langle (\zeta-1) \widehat{\dot{\theta}^{j}}, \frac{\Delta t}{2}( \widehat{\dot{\mu}^j}-\widehat{\dot{\mu}^{j-1}})\rangle+ \langle \frac{\Delta t}{2}( \widehat{\dot{\mu}^j}+\widehat{\dot{\mu}^{j-1}}),(\zeta+1)\widehat{\dot{\theta}^j} \rangle}_{J_2^j}\\
&+\underbrace{(\frac{ \Delta t}{2})^2\langle \widehat{\dot{\mu}^j}+\widehat{\dot{\mu}^{j-1}}, \widehat{\dot{\mu}^j}-\widehat{\dot{\mu}^{j-1}}) \rangle}_{J_3^j}+\\
& \underbrace{\langle A_0(\Delta t^3), (\zeta+1) \widehat{\dot{\theta}^{j}}+ \frac{\Delta t}{2}( \widehat{\dot{\mu}^j}-\widehat{\dot{\mu}^{j-1}})+A_0(\Delta t^3) \rangle}_{J_4^j}\\
&+\underbrace{\langle (\zeta-1) \widehat{\dot{\theta}^{j}}+ \frac{\Delta t}{2} (\widehat{\dot{\mu}^j}+\widehat{\dot{\mu}^{j-1}})+A_0(\Delta t^3),A_0(\Delta t^3) \rangle}_{J_5^j}.
\end{split}
\end{equation}

Adding those terms $(\ref{LHSADB})$ over time, we obtain a telescopic sum
\begin{equation}\label{upper0ADB}
\sum_{j=2}^{n}\langle \widehat{\dot{\theta}^{j+1}}-\zeta \widehat{\dot{\theta}^{j-1}},  \widehat{\dot{\theta}^{j+1}}+\zeta \widehat{\dot{\theta}^{j-1}}\rangle=||\widehat{\dot{\theta}^{n+1}}||^2+|| \widehat{\dot{\theta}^{n}}||^2-\left( ||\widehat{\dot{\theta}^{2}}||^2+|| \widehat{\dot{\theta}^{1}}||^2\right)+I_1,
\end{equation}
where $I_1$ is a purely imaginary term.

Now we analyze the sum over time of the right-hand side terms
\paragraph*{$J_1$ contribution: \:} a direct calculation shows that
\begin{equation}J_1^j=||\widehat{\dot{\theta}^j}||^2-||\widehat{\dot{\theta}^j}||^2+2iIm(\langle \zeta \widehat{\dot{\theta}^j},\widehat{\dot{\theta}^j} \rangle).\end{equation}

Thus, the sum over time is telescopic
\begin{equation}\label{upper1ADB}
\sum_{j=2}^{n}J_1^j=|| \widehat{\dot{\theta}^{n-1}}||^2+|| \widehat{\dot{\theta}^{n-2}}||^2-|| \widehat{\dot{\theta}^{1}}||^2-|| \widehat{\dot{\theta}^{0}}||^2+I_2,
\end{equation}
where $I_2$ is a purely imaginary term.

\paragraph*{$J_2$ contribution: \:} similarly
$$J_2^j=2Re\underbrace{\left(\langle \zeta \widehat{\dot{\theta}^j},\frac{\Delta t}{2}\widehat{\dot{\mu}^j}\rangle+\langle \widehat{\dot{\theta}^j},\frac{\Delta t}{2}\widehat{\dot{\mu}^{j-1}} \rangle\right)}_{J_*}+2iIm\left(\langle \frac{\Delta t}{2} \widehat{\dot{\mu}^{j-1}},\zeta \widehat{\dot{\theta}^j} \rangle+\langle\frac{\Delta t}{2}\widehat{\dot{\mu}^j},\widehat{\dot{\theta}^j} \rangle\right),$$
where
\begin{equation}
\begin{split}
J_*&=2Re\left(\langle \zeta \widehat{\dot{\theta}^j},\widehat{\dot{\theta}^{j+1}}-\zeta \widehat{\dot{\theta}^j}+A_0(\Delta t^3) \rangle+\langle \widehat{\dot{\theta}^j},\widehat{\dot{\theta}^j}-\zeta\widehat{\dot{\theta}^{j-1}}+A_0(\Delta t^3)\rangle \right)\\
&=2Re\left(\langle \zeta \widehat{\dot{\theta}^j},\widehat{\dot{\theta}^{j+1}}\rangle-||\widehat{\dot{\theta}^j}||^2+||\widehat{\dot{\theta}^j} ||^2-\langle \widehat{\dot{\theta}^j},\zeta\widehat{\dot{\theta}^{j-1}}\rangle \right)+A_0(\Delta t^3)A_0(\dot{\theta^j})\\
&=2Re\left(\langle \widehat{\dot{\theta}^{j+1}},\zeta \widehat{\dot{\theta}^j} \rangle-\langle \widehat{\dot{\theta}^j},\zeta\widehat{\dot{\theta}^{j-1}}\rangle \right)+\Delta tA_0(\Delta t^2)A_0(\dot{\theta^j}).
\end{split}
\end{equation}

The sum over time is
\begin{equation}\label{upADB1}
\begin{split}
\sum_{j=2}^{n}J_2^j&=2Re\left(\langle \widehat{\dot{\theta}^{n+1}},\zeta \widehat{\dot{\theta}^n} \rangle-\langle \widehat{\dot{\theta}^2},\zeta\widehat{\dot{\theta}^{1}}\rangle \right)+A_0(\Delta t^2)A_0(\dot{\theta^j})+I_3\\
&=2Re\left(\langle \zeta_m \widehat{\dot{\theta}_m^{n}}+\frac{\Delta t}{2}\widehat{\dot{ \mu}_m^{n}}+A_0(\Delta t^3) ,\zeta \widehat{\dot{\theta}^n} \rangle-\langle \widehat{\dot{\theta}^2},\zeta\widehat{\dot{\theta}^{1}}\rangle \right)+A_0(\Delta t^2)A_0(\dot{\theta^j})+I_3
\end{split}
\end{equation}where $I_3$ is a purely imaginary term.

For the first step, $\dot{\theta}_m^0$ is zero. In addition $(\ref{vt2}),(\ref{DtupperNL})$ and Plancherel theorem shows that
 \begin{equation}\label{vt2a}
  ||\dot{\theta}^1||_{l^2}^2= ||\widehat{\dot{\theta}^1}||^2=||A_0(\Delta t^2)||^2=O(\Delta t^4).
  \end{equation}
  
 For the second step, consider $(\ref{discvartheta})$ and approximation $(\ref{eqDtmu})$ to get 
  $$||\dot{\theta}^2||_{l^2}^2\leq O((h^r+\Delta t^2)^2),$$
 which by induction implies that $||\dot{\theta}^n||=O(h^r+\Delta t^2)$.
 
Considering only real terms in $(\ref{upADB1})$ and approximation $(\ref{eqDtmu})$ for the nonlinear error we obtain
\begin{equation}\label{upper2ADB}
\begin{split}
|Re(\sum_{j=2}^{n}J_2^j)|&\leq  2|| \zeta_m \widehat{\dot{\theta}_m^{n}}+\frac{\Delta t}{2}\widehat{\dot{ \mu}_m^{n}}+A_0(\Delta t^3)|| \cdot||\zeta \widehat{\dot{\theta}^n}|| +C(h^{r}\Delta t^2+\Delta t^4)\\
&\leq O(h^{r}+\Delta t^2)O(h^r+\Delta t^2)+O((h^{r}+\Delta t^2)^2)=O((h^{r}+\Delta t^2)^2).
\end{split}
\end{equation}

\paragraph*{$J_3$ contribution:\: } a direct calculation shows
\begin{equation}
J_3^j=(\frac{\Delta t}{2})^2(||\widehat{\dot{\mu}^j}||^2-||\widehat{ \dot{\mu}^{j-1}}||^2)+(\frac{\Delta t}{2})^22iIm(\langle\widehat{\dot{ \mu}^{j-1}},\widehat{\dot{\mu}^j}\rangle).
\end{equation}

 Then, the sum over time is also telescopic
 \begin{equation}\label{upper3ADB}
 \sum_{j=2}^nJ_3^j=(\frac{\Delta t}{2})^2(||\widehat{\dot{\mu}^j}||^2-|| \widehat{\dot{\mu}^{j-1}}||^2)+I_4,
 \end{equation}
 where $I_4$ is a purely imaginary term.

\paragraph*{$J_4,J_5$ contribution: \:} by induction and approximation $(\ref{eqDtmu})$ we find that $(\zeta+1) \widehat{\dot{\theta}^{j}}+ \frac{\Delta t}{2}( \widehat{\dot{\mu}^j}-\widehat{\dot{\mu}^{j-1}})+A_0(\Delta t^3)=A_0(\dot{\theta}^j+\Delta t^3)$, then using  Cauchy-Schwarz, triangle inequalities and Plancherel theorem we get
\begin{equation}|J_4|\leq \Delta t^3||A_0(\dot{\theta}^j+\Delta t^3)||_{l^2}\leq \Delta t^3O(h^r+\Delta t^3).\end{equation}

 Similarly, 
\begin{equation}|J_5|=\Delta t^3O(h^r+\Delta t^3).\end{equation}

Thus, the sum over time is

\begin{equation}\label{upper4ADB}
|\sum_{j=2}^nJ_4^j+J_5^j|=O(h^r\Delta t^2+\Delta t^4).
\end{equation}

With this $(\ref{upper1ADB})$,$(\ref{upper2ADB})$,$(\ref{upper3ADB})$,$(\ref{upper4ADB})$ information and considering the real part of $(\ref{upper0ADB})$ we write
\begin{equation}\label{RHS2}
\begin{split}
||\widehat{\dot{\theta}^{n+1}}||^2+|| \widehat{\dot{\theta}^{n}}||^2-\left( ||\widehat{\dot{\theta}^{2}}||^2+|| \widehat{\dot{\theta}^{1}}||^2\right)& \leq\\
& || \widehat{\dot{\theta}^{n-1}}||^2+|| \widehat{\dot{\theta}^{n-2}}||^2-|| \widehat{\dot{\theta}^{1}}||^2-|| \widehat{\dot{\theta}^{0}}||^2\\
&+O((h^{r}+\Delta t^2)^2)+O(h^r\Delta t^2+\Delta t^4).
\end{split}
\end{equation}

Therefore,
$$||\dot{\theta}^{n+1}||_{l^2}^2=O((h^{r}+\Delta t^2)^2)\Rightarrow ||\dot{\theta}^{n+1}||_{l^2}\leq C(h^r+\Delta t^2).$$
As a consequence, the upper bound holds for a longer time ($j=n+1$) than $T^*$ $(\ref{ThetaConvergence})$, and $T^*=T$ as desired.{\hfill\ensuremath{\square}}

\end{proof}

\end{subsection}


\subsection{Here we show the conservation of the quantities $(\ref{firstint})$ under Airy Flow over time.} \label{conservq} 

\paragraph{M1:}  $k=\theta_s/s_{\alpha}$ is a perfect derivative of a periodic function.  The result follows by the Fundamental Theorem of Calculus.\\
 \paragraph{M2:} Observe that
 \begin{equation}I:=\frac{\partial M2}{\partial t}=\int 2kk_t ds. \end{equation}
 
   We know that for airy flow
 \begin{equation}
 k_t=k_{sss}+3\frac{k^2k_s}{2}=(k_{ss}+\frac{k^3}{2})_s,
 \end{equation}
 then
 \begin{equation} I=\int 2kk_{sss}+3k^3k_s ds=2\int k k_{sss}ds+\int \frac{\partial (\frac{3}{4}k^4)}{\partial s}ds=2\int k k_{sss} ds.\end{equation}

Again, since $(kk_{ss})_s=kk_{sss}+k_{ss}k_s$ we have 
 \begin{equation} \int k k_{sss} ds=-\int k_sk_{ss} ds=-\frac{1}{2}\int \frac{\partial k_s^2}{\partial s}ds=0\Rightarrow I=0, \end{equation} 
and $M2$ is conserved over time.\\

\paragraph{M3:} Similarly, using integration by parts and periodicity of the functions we obtain
 \begin{equation}
 \begin{split}
 &J:=\frac{\partial M3}{\partial t}=\int (k_s k_{st}-\frac{1}{2}k^3k_t )ds =\int [-k_{ss}-\frac{1}{2}k^3]k_tds\\
&=-\int (k_{ss}+\frac{1}{2}k^3)(k_{ss}+\frac{k^3}{2})_sds=\frac{-1}{2}\int\frac{\partial (k_{ss}+\frac{k^3}{2})^2}{\partial s} ds=0.
\end{split}
\end{equation}

\end{appendices}
\section{Acknowledgements.} 

 Mariano Franco-de-Leon acknowledges the hospitality of the University of California, Irvine where preliminary work was
performed. Gratefully acknowledges economic support from the National Council of Science and Technology in Mexico (CONACyT), the University of California Institute for Mexico and the United States (UC Mexus), partial support from the Ministry of Public Education in Mexico, SEP (Secretaria de Educaci\'on P\'ublica),  and the Miguel Velez Fellowship.



%


\bibliography{Airy_mKdV}

\end{document}